\documentclass[11pt,a4paper,reqno]{amsproc}

\input{newheader.sty}

\firsthead{Version of \today\ generated from file \jobname\ at \currenttime}

\title{The Widom--Sobolev formula for discontinuous matrix-valued symbols}

\author[L.\ Bollmann]{Leon Bollmann}
\address[Leon Bollmann]{Mathematisches Institut,
  Ludwig-Maximilians-Universit\"at M\"unchen,
  Theresienstra\ss{e} 39,
  80333 M\"unchen, Germany}
\email{bollmann@math.lmu.de}

\author[P.\ M\"uller]{Peter M\"uller}
\address[Peter M\"uller]{Mathematisches Institut,
 Ludwig-Maximilians-Universit\"at M\"unchen,
  Theresienstra\ss{e} 39,
  80333 M\"unchen, Germany}
\email{mueller@lmu.de}

\begin{document}

\begin{abstract}
	We prove the Widom--Sobolev formula for the asymptotic behaviour of truncated Wiener--Hopf operators with discontinuous matrix-valued symbols for three different classes of test functions. 
The symbols may depend on both position and momentum except when closing the asymptotics for twice differentiable test functions with H\"older singularities. The cut-off domains are allowed to have piecewise differentiable boundaries. In contrast to the case where the symbol is smooth in one variable, the resulting coefficient in the enhanced area law we obtain here remains as explicit for matrix-valued symbols as it is for scalar-valued symbols.
\end{abstract}

\maketitle
\thispagestyle{empty}

\section{Introduction}

Initiated by Szeg\H{o} \cite{Szego1915,Szego1952}, the analysis of the determinant of 
large Toeplitz matrices has continuously attracted attention ever since, see, e.g.,  
\cite{Fisher.Hartwig.1969, MR0331107, Basor, MR1724795, MR2223704, MR2831118}.
Early generalisations from matrices to integral operators of Wiener--Hopf type with discontinuous symbols were restricted to one spatial dimension \cite{LandauWidom80,  Widom1982} but conjectured to have a natural extension to the higher-dimensional case. The special case of a half space was treated in \cite{Widom1990}.
It took more than another two decades until the substantial work of Sobolev \cite{sobolevlong} together with its extensions \cite{sobolevschatten, sobolevpw, sobolevc2} provided a comprehensive proof of this conjecture for a large class of pseudodifferential operators, which is now known as the Widom--Sobolev formula. 
Since the beginning of the 21st century, such questions of Szeg\H{o}-type asymptotics attracted a great deal of additional attention because of their relevance to the large-volume behaviour of entanglement entropies for non-interacting Fermi gases \cite{GioevKlich,Helling:2011gr}. This led to further developments with regard to applications of Szeg\H{o} asymptotics 
 to Schr\"odinger operators \cite{LeschkeSobolevSpitzerLaplace, PhysRevLett.113.150404, ElgartPasturShcherbina2016, LeschkeSobolevSpitzerMagnetic, LeschkeSobolevSpitzerMultiScale, LeschkeSobolevSpitzerTemperature, MuPaSc19, Dietlein18, MuSc20, 10.1063/5.0135006, PfSp22, Pfeiffer21, PfSp23}. 
 
It is a natural question in which way these results extend to Wiener--Hopf operators with matrix-valued symbols. In the case of Wiener--Hopf operators which only feature a discontinuity in a single variable, this question has already been answered in \cite{Widom1980} and extended in \cite{Widom1985FullExpansion}. In this case it turns out that there is a major difference between the scalar-valued and matrix-valued case. While the coefficient of the ensuing area law has an explicit integral form in the scalar-valued case, in the matrix-valued case no such explicit form is known \cite[bottom of p.\ 5]{Widom1985FullExpansion} even when the space dimension is one. 
Instead, it still remains to compute a function of a one-dimensional Wiener--Hopf operator. 
Especially in higher space dimensions, the coefficient is even less accessible due to the lack of symmetry of the domain. Yet, the recent paper \cite{FinsterSob} manages to extract certain scaling information from this coefficient for a symbol related to the multi-dimensional
Dirac operator.
This rests on multi-scale techniques in the spirit of \cite{LeschkeSobolevSpitzerMultiScale} and on a generalisation of the results from \cite{Widom1980} both in the allowed symbols and test functions.

Up to now the case of discontinuities in both variables, which is often referred to as just 
the ``discontinuous case'', has only been studied to some degree in one spatial dimension for matrix-valued symbols \cite{Widom1982, FinsterI}. The systematic route towards an understanding of Szeg\H{o} asymptotics with a logarithmically enhanced term for matrix-valued symbols in higher dimensions consists of generalising appropriate 
parts of \cite{sobolevlong,  sobolevschatten, sobolevpw, sobolevc2}
to matrix-valued symbols. This is the goal of the present paper. 
In contrast to the case with a single discontinuity, the resulting coefficient of the enhanced area law has an integral representation which is as explicit as in the scalar-valued case. 

As in the scalar-valued case, the Szeg\H{o} asymptotics which we prove here 
for the matrix-valued case can be applied to  the scaling behaviour of entanglement entropies.
Non-interacting relativistic fermions described by the Dirac operator constitute a prime example
for such an application, see the forthcoming paper \cite{BM-inprep}. The recent work \cite{FinsterSob} is also devoted to Szeg\H{o} asymptotics for a non-interacting Fermi gas of free Dirac particles. Whereas \cite{FinsterSob} are concerned with an area law as in \cite{Widom1980}, 
the main focus of \cite{BM-inprep} is on situations where there is a logarithmic enhancement to the area law. 

Now, let us describe and contextualise the results of the present paper in some more detail. 
The goal is to prove several variants of the Widom--Sobolev formula for traces of test functions of truncated Wiener--Hopf operators with matrix-valued symbols. The three main results are Theorem~\ref{Asymptotics-Analytic} for analytic test functions $h$, Theorem~\ref{Asymptotics-Smooth} for arbitrarily often differentiable $h$ and Theorem~\ref{Asymptotics-Hoelder} for $h$ which are twice differentiable except at finitely many points where they merely obey a H\"older condition, see Assumption~\ref{H-Functions}. Notably, the test functions considered in Theorem~\ref{Asymptotics-Hoelder} include all R\'enyi entropy functions, in particular the von Neumann entropy function. This is one of the primary motivations to consider test functions as in Assumption~\ref{H-Functions}. In contrast, the allowed Wiener--Hopf operators decrease in generality from Theorem~\ref{Asymptotics-Analytic} to Theorem~\ref{Asymptotics-Hoelder}. 
For example, the truncated Wiener--Hopf operator in Theorem~\ref{Asymptotics-Smooth} is of the form 
	\begin{align}
		\label{GL-def}
		G_L(A_1,A_2;\Lambda, \Gamma) :=& \mathbf{1}_\Lambda \Op_L(\mathbf{1}_\Gamma) 
			\Real \big[\Op_L^l(\Real A_{1}) \big]	\Op_L(\mathbf{1}_\Gamma) \mathbf{1}_\Lambda \notag\\
		&+ 	\mathbf{1}_\Lambda \Op_L(\mathbf{1}_{\Gamma^{c}}) \Real \big[ \Op_L^l(\Real A_{2}) \big]
			\Op_L(\mathbf{1}_{\Gamma^{c}}) \mathbf{1}_\Lambda,
	\end{align}
	where $A_{1}$ and $A_{2}$ are suitable smooth symbols with values in the $\C^{n\times n}$-matrices which may depend on both position and momentum. The (standard) left-quantisation functor $\Op_{L}^{l}$ of the symbols is defined in \eqref{left-quant}, and $\Re T := (T+T^{*})/2$ is the self-adjoint part of a (bounded) operator $T$. Besides the discontinuity in space due to the restriction to the volume $\Lambda\subset\R^{d}$, the truncated Wiener--Hopf operator \eqref{GL-def} features a second, more general jump discontinuity at the boundary of the	momentum region $\Gamma \subset\R^{d}$ where the change from $A_{1}$ to $A_{2}$ occurs. Because of the two discontinuities, 	Theorem~\ref{Asymptotics-Smooth} yields a two-term asymptotics 		
	\begin{align}
	\label{bsp-thm}
\tr_{L^2(\R^d)\otimes\C^n}\big[h\big(G_L(A_1,A_2;&\Lambda, \Gamma)\big)\big] \nonumber \\ =& \ L^d \Big[\mathfrak{W}_0\big(\tr_{\C^n}[h(\Real A_1)];\Lambda,\Gamma\big)+\mathfrak{W}_0\big(\tr_{\C^n}[h(\Real A_2)];\Lambda,\Gamma^c\big)\Big]\nonumber\\&+L^{d-1}\log L \ \mathfrak{W}_1\big(\mathfrak{U}(h;\Real A_1,\Real A_2);\partial\Lambda,\partial\Gamma\big)\nonumber\\ &+o(L^{d-1}\log L),
\end{align}
as $L\rightarrow\infty$ with a logarithmically enhanced area term. Interestingly, the 
coefficients $\mathfrak{W}_0$ and $\mathfrak{W}_1$ are the same as in the scalar case, see \eqref{coeff-w0-def} and \eqref{coeff-w1-def}. For $d=1$, this was already found by Widom \cite{Widom1982}. We point out that the first argument of $\mathfrak{W}_0$ is a matrix trace and, hence, scalar valued. The same is true for the first argument of $\mathfrak{W}_1$, where the matrix trace is hidden in the definition \eqref{frakU-def} of the
scalar-valued symbol $\mathfrak U$.

An operator sum as in \eqref{GL-def} describes the general jump discontinuity of a symbol along the boundary of $\Gamma$. It is useful for applications like to the free Dirac operator in  \cite{BM-inprep}. However, it also represents an additional technical challenge if the matrix-valued symbols $A_1$ and $A_2$ do not commute. A careful consideration in the proof of Theorem \ref{Asymptotics-Gamma-Gamma-Complement} still allows to get coefficients similar to the scalar-valued case.
Such operator sums were first studied in 
\cite[Thm.\ 5.2]{sobolevc2} for scalar-valued symbols. In addition to working with matrix-valued symbols, we require slightly weaker assumptions on $\Gamma$ and on the symbols $A_{1}$ and $A_{2}$.
The first situation requires either $\Gamma$ or $\Gamma^{c}$ to be bounded. Then, only one symbol needs to have compact support in momentum, namely the one that is in the same term of \eqref{GL-def} as the unbounded momentum region. The second situation allows both $\Gamma$ and $\Gamma^{c}$ to be unbounded. But then, both symbols $A_{1}$ and $A_{2}$ are required to be compactly supported in momentum, see Remark~\ref{Remark-Gamma-GammaC}. These less stringent requirements are possible because of less stringent requirements in some Schatten--von Neumann estimates that we prove: Lemma~\ref{Commutation-Lemma} and 
Lemma~\ref{Estimate-Log-Q} allow both $\Lambda$ and $\Gamma$ to be unbounded (admissible) domains. This fact also simplifies several other steps leading to Theorems~\ref{Asymptotics-Analytic}, \ref{Asymptotics-Smooth} and~\ref{Asymptotics-Hoelder}.

The extension of \eqref{bsp-thm} from polynomial test functions $h$ to more general test functions as in Theorems~\ref{Asymptotics-Analytic}, \ref{Asymptotics-Smooth} and~\ref{Asymptotics-Hoelder} is often referred to as the ``closing of the asymptotics.'' In the case $A_{2}=0$ and for a scalar-valued symbol $A_{1}$, the results for different classes of test functions and (merely piece-wise) differentiable admissible domains $\Lambda$ and $\Gamma$ can be assembled from different papers within the large body of work of Sobolev \cite{sobolevlong, sobolevschatten, sobolevpw, sobolevc2} with precursors by Widom \cite{Widom1990} and Gioev \cite{10.1155/IMRN/2006/95181}.
In this paper, we give a unified and complete treatment of such results for matrix-valued symbols, without sacrificing too much generality for a given class of test functions.

The plan of the paper is as follows. In Section~\ref{sec:prelim} we describe the basic notions and notations. Section~\ref{sec:poly} is devoted to the proof of the Widom--Sobolev formula for polynomial test functions. This is done in Section~\ref{subsec:polyproof} after adapting and slightly generalising Schatten--von Neumann estimates for commutators from \cite{sobolevlong, sobolevschatten} to matrix-valued symbols in Section~\ref{subsec:commute}. Section~\ref{subsec:sumsymbols} extends the result of Section~\ref{subsec:polyproof} to Wiener--Hopf operators  which are a sum of two terms as in \eqref{GL-def}. 
The closing of the asymptotics is done in Section~\ref{sec:closing-asymp}.
Section~\ref{subsec:analytic} treats analytic functions, Section~\ref{subsec:smooth} arbitrarily often differentiable functions and Section~\ref{subsec:Hoelder} functions which are twice differentiable except at finitely many points where they merely obey a H\"older condition.

%%%%%%%%%%%%%%%%%%%%%%%%%%%%%%%%%%%%%%%%%%%%%%%%%%%%%%%%%%%%%%
%%%%%%%%%%%%%%%%%%%%%%%%%%%%%%%%%%%%%%%%%%%%%%%%%%%%%%%%%%%%%%

\section{Notation and preliminaries}
\label{sec:prelim}

\subsection{Admissible domains}

We mostly follow the terminology in \cite{sobolevschatten}.

\begin{defn}
	Given a natural number $d \in\N\setminus\{1\}$, we call a subset $\Omega\subset\R^d$ 
	a \emph{basic domain}, if there exists a Lipschitz function 
	$\Phi:\R^{d-1}\rightarrow\R$ and a suitable choice (obtained by relabelling and rotation) of Cartesian coordinates 
	$ \R^{d} \ni x=(x_1,x_2,\ldots,x_d)$ such that 
	\begin{align}
		\Omega= \big\{x\in\R^d : x_d>\Phi(x_1,x_2,\ldots,x_{d-1})\big\}.
	\end{align}
	For $m\in\N$, we further call a basic domain a \emph{piece-wise $C^m$-basic domain}, if, in addition, 
	$\Phi$ is a piece-wise $C^m$-function. 
	In the case $d=1$, we call  $\Omega\subset\R$ a (piece-wise $C^m$-) basic domain, if it is an open 
	interval of the form $\,]a,\infty[\,$, respectively $\,]-\infty,a[\,$, for arbitrary $a\in\R$.
\end{defn}

\begin{defn}\label{Definition-Admissible-Domain}
We call $\Omega\subset\R^d$ an \emph{admissible domain}, if it can be locally represented as a basic domain, i.e. for all $x\in\R^d$ there is some $r>0$ such that for $B_r(x)$, the open ball of radius $r$ about $x$ in $\R^{d}$, we have 
\begin{align}\label{Def-Local-Representation}
B_r(x)\cap\Omega = B_r(x)\cap\Omega_x
\end{align}
for some basic domain $\Omega_x$.
For $m\in\N$, we further call an admissible domain a \emph{piece-wise $C^m$-admissible domain}, if it can be locally represented as a piece-wise $C^m$-basic domain, i.e.\ the domains $\Omega_x$ in (\ref{Def-Local-Representation}) are piece-wise $C^m$-basic domains. 
\end{defn}

\begin{rem}
	\begin{enumerate}[(a)]
	\item
		Given a basic domain $\Omega_0$, the domain $(\overline{\Omega_0})^c$ is also a basic domain by the 
		coordinate transformation $x_d\mapsto -x_d$ and replacing $\Phi$ by $-\Phi$. 
		This extends to admissible domains in the following way. Given an admissible domain $\Omega$ locally 
		represented by basic domains $\Omega_x$, the domain $(\overline{\Omega})^c$ is locally represented 
		by $(\overline{\Omega_x})^c$ and therefore again admissible.
	\item
		Note that the boundary of a basic domain has Lebesgue measure zero as it is the graph of a Lipschitz function. 
		This also extends to admissible domains, as $\R^d$ is heredetarily Lindel\"of -- every metrisable 
		space is heredetarily Lindel\"of, see, e.g., \cite[Ex.\ 3.8.A, Cor.\ 4.1.13]{engelking1989general} -- 
		and therefore each open cover of the boundary has a countable subcover.
	\item
		Given an admissible domain $\Omega$, the operator $1_{\Omega^c}$ of multiplication with the 
		indicator function of $\Omega^c$ agrees with the operator $1_{(\overline{\Omega})^c}$ on $L^2(\R^d)$, 
		as $\partial \Omega$ has measure zero. Therefore, the operator $1_{\Omega^c}$ acts as multiplication 
		with the indicator function of an admissible domain. This will be useful later in the paper.
	\end{enumerate}
	\label{Remark-Complement-Admissible}
\end{rem}

%%%%%%%%%%%%%%%%%%%%%%%%%%%%%%%%%%%%%%%%%%%%%%%%%%%%

\subsection{Complex-valued symbols}
Given a natural number $d\in\N$, we consider \emph{complex-valued amplitudes}
$a\in C^\infty_{b}(\R^d\times\R^d\times\R^d)$ which are smooth, i.e.\ arbitrarily often differentiable,
and have the property that $a$ and any partial derivative of $a$ of arbitrary order is bounded. 
The first two variables of $a$ play the role of space variables, the last one of a momentum variable. 
We denote the first variable by $x$, the second by $y$ and the third by $\xi$. 
Given any such amplitude, the following integral formula defines a bounded
\cite{cordes1975compactness} linear operator on the Hilbert space $L^{2}(\R^{d})$ of complex-valued 
square-integrable functions over $\R^d$, which leaves Schwartz space $\mathcal{S}(\R^d)$ invariant. Given 
a Schwartz function $u\in\mathcal{S}(\R^d)$, its action is defined by  
\begin{align}
\big(\Op_L^{lr}(a)u\big)(x):= \left(\frac{L}{2\pi}\right)^d \int_{\R^d} \!\int_{\R^d} \e^{\i L\xi(x-y)}a(x,y,\xi)u(y)\,\d y\,\d\xi
\end{align}
for every $x\in\R^{d}$. Here, $\i$ denotes the imaginary unit, and the integrations are with respect to 
Lebesgue measure in $\R^{d}$.
In the case that the function $a$ does not depend on both $x$ and $y$, we call $a$ 
a \emph{complex-valued symbol} and define the left and right operators of $a$ by
\begin{align}
\big(\Op_L^l(a)u\big)(x):= \left(\frac{L}{2\pi}\right)^d \int_{\R^d} \!\int_{\R^d} \e^{\i L\xi(x-y)}a(x,\xi)u(y)
\,\d y\,\d\xi
\end{align}
and
\begin{align}
\big(\Op_L^r(a)u\big)(x):= \left(\frac{L}{2\pi}\right)^d \int_{\R^d} \!\int_{\R^d} \e^{\i L\xi(x-y)}a(y,\xi)u(y)
\,\d y\,\d\xi.
\end{align}
If $a$ depends only on $\xi$, we no longer need the distinction between left and right operators and just write $\Op_L(a)$. In this case, we have 
\begin{equation}
	\label{mult-op-in-mom}
	\Op_L(a) = \mathcal{F}^{-1} a(\,\pmb\cdot /L) \mathcal{F},
\end{equation} 
where $\mathcal{F}$ is the unitary Fourier transform on $L^{2}(\R^{d})$ and, on the right-hand side, 
$a$ is to be understood as a multiplication operator (in Fourier space). Thus, in the case of 
\eqref{mult-op-in-mom}, $\Op_{L}(a)$ gives rise to a well-defined and bounded operator on 
$L^{2}(\R^{d})$, whenever the symbol $a$ is an essentially  bounded measurable function on $\R^{d}$.

We now introduce discontinuities in both variables $x$ and $\xi$ of the left operator of the symbol $a$. Let $\Lambda, \Gamma \subseteq \R^{d}$ be bounded measurable subsets. Then we define the following operator on $L^2(\R^d)$ 
\begin{align}
\label{scalar-T-op}
T_L(a):= T_L(a;\Lambda,\Gamma):= 1_\Lambda \Op_L(1_\Gamma) \Op_L^l(a) \Op_L(1_\Gamma) 1_\Lambda,
\end{align} 
where $1_\Lambda,1_\Gamma$ are the associated indicator functions, $1_{\Lambda}$ acts as a 
multiplication operator and $\Op_{L}(1_{\Gamma})$ is defined by \eqref{mult-op-in-mom}. 
The operator $\Op_L(1_\Gamma)1_\Lambda$ is trace class according to \cite[Chap.\ 11, Sect.\ 8, Thm.\ 11]{BirmanSolomjak}. Therefore, $T_L$ and its symmetrised version
\begin{align}
\label{scalar-S-op}
S_L(a):= S_L(a;\Lambda,\Gamma):= 1_\Lambda \Op_L(1_\Gamma) \Real \big[\Op_L^l(\Real a)\big] \Op_L(1_\Gamma) 1_\Lambda,
\end{align} 
which is self-adjoint and has a real-valued symbol, are both trace-class. Here, we also used $\Re$ to denote the self-adjoint part $\Re Q := (Q+Q^{*})/2$ of a bounded operator $Q$. 

There is a famous conjecture by Widom \cite{Widom1982}, stating that the asymptotic formula 
\begin{align}
	\label{eq:Widom-conj}
	\tr_{L^{2}(\R^{d})} \big[g\big(S_L(a)\big)\big] 
		=&L^d\mathfrak{W}_0\big( g(\Real a);\Lambda,\Gamma\big) 
			+ L^{d-1}\log L\; \mathfrak{W}_1\big(\mathfrak{A}(g;\Real a);\partial\Lambda,\partial\Gamma\big) \notag\\
		&+ o(L^{d-1}\log L)
\end{align}
holds for suitable test functions $g:\R\rightarrow \R$ with $g(0)=0$ as $L\rightarrow\infty$.
This was proved by Sobolev in \cite{sobolevlong, sobolevpw} for analytic and smooth test functions and extended in \cite{ sobolevschatten, sobolevc2} to test functions which are twice differentiable except at finitely many points where they merely obey a H\"older condition under the assumption that the symbol $a$ only depends on the variable $\xi$.
In all three cases $\Lambda$ is required to be a piece-wise $C^1$-admissible domain and $\Gamma$ is required to be a piece-wise $C^3$-admissible domain. 
The coefficients are given for continuous symbols $b$ by
\begin{align}
\label{coeff-w0-def}
	\mathfrak{W}_0(b;\Lambda,\Gamma):= \frac{ 1}{(2\pi)^d}\int_\Lambda\int_\Gamma b(x,\xi) \,\d\xi \,\d x
\end{align}
and 
\begin{align}
	\label{coeff-w1-def}
	\mathfrak{W}_1(b;\partial\Lambda,\partial\Gamma):= 
	\LEFTRIGHT\{.{\begin{array}{l@{\quad}l} \displaystyle \sum_{(x,\xi)\in\partial\Lambda\times\partial\Gamma}b(x,\xi), 
			& \text{for~} d=1, \\[4ex]
		\displaystyle\frac{1}{(2\pi)^{d-1}}\int_{\partial\Lambda}\int_{\partial\Gamma} b(x,\xi)\,
			\vert n_{\partial \Lambda}(x)\cdot n_{\partial \Gamma}(\xi)\vert \,\d S(\xi)\, \d S(x), 
			& \text{for~} d\ge 2, \end{array}}
\end{align}
where $n_{\partial \Lambda}$, respectively $n_{\partial \Gamma}$, denotes the vector field of exterior unit normals in $\R^{d}$ to the boundary $\partial \Lambda$, respectively $\partial  \Gamma$. We write $\d S$ for integration with respect to the $(d-1)$-dimensional surface measure induced by Lebesgue measure in $\R^{d}$. 
Finally, the symbol $\mathfrak{A}(g;b)$ in \eqref{eq:Widom-conj} is given by
\begin{align}
	\label{frakAsym}
	\mathfrak{A}(g;b)(x,\xi):=  \frac{1}{(2\pi)^2}\int_0^1\frac{g\big(t b(x,\xi)\big) - 
	tg\big(b(x,\xi)\big)}{t(1-t)}\,\d t
\end{align}
for every $x,\xi\in\R^{d}$ and for H\"older-continuous test functions $g$.

\begin{rem}
In the case of analytic or smooth test functions, the above asymptotics \eqref{eq:Widom-conj} holds for a more general class of symbols which are not required to be arbitrarily often differentiable but merely have a finite symbol norm $\mathbf{N}^{(m_{x},m_{\xi})}(a)$ for $m_{x}=m_{\xi}=d+2$. For the definition of the symbol norm, we refer to Definition \ref{defn-N-Norm-Matrix} below with $n=1$.
If the test function $g$ is a non-smooth function, which is relevant for applications, the required parameter $m_{\xi}$ in the norm increases significantly. This is the reason why we restrict ourselves to smooth symbols, thereby minimising some technical efforts.
\end{rem}

%%%%%%%%%%%%%%%%%%%%%%%%%%%%%%%%%%%%%%%%%%%%%%%%%%%%%%%%%%%%%%%%

\subsection{Matrix-valued symbols}

Given a complex matrix $A\in\C^{n\times n}$ with matrix dimension $n\in\N$, we denote its entry in the $\nu$th row and $\mu$th column by $(A)_{\nu\mu}$, where $\nu,\mu\in\{1, \ldots,n\}$. We now introduce \emph{matrix-valued amplitudes} $A \in C^{\infty}_{b}(\R^{d} \times \R^{d} \times \R^{d}, \C^{n\times n})$ which are arbitrarily often differentiable and have the property that $A$ and any partial derivative of $A$ of arbitrary order is bounded. We will always identify this space with the tensor product 
\begin{equation}
	C^{\infty}_{b}(\R^{d} \times \R^{d} \times \R^{d}, \C^{n\times n}) 
	= C^{\infty}_{b}(\R^{d} \times \R^{d} \times \R^{d}) \otimes \C^{n\times n}.
\end{equation}
Thus, $A \in C^{\infty}_{b}(\R^{d} \times \R^{d} \times \R^{d}, \C^{n\times n})$ is equivalent to requiring that the matrix entries of $A$ are corresponding complex-valued amplitudes, i.e.\ 
$(A)_{\nu\mu} \in C^{\infty}_{b}(\R^{d} \times \R^{d} \times \R^{d})$ for any $\nu,\mu\in\{1, \ldots,n\}$. 
If $A\in C^{\infty}_{b}(\R^{d} \times \R^{d}, \C^{n\times n})$ does not depend on both $x$ and $y$, we call $A$ a \emph{matrix-valued symbol}.

In analogy to the scalar case, we define bounded -- see Lemma \ref{Bounded-Lemma} --  matrix-valued operators 
$\Op_L^{lr}(A),\Op_L^l(A)$ and $\Op_L^r(A)$ on the product Hilbert space $L^2(\R^d)\otimes\C^n$. Their 
action on Schwartz functions $u\in\mathcal{S}(\R^d) \otimes \C^n$ is given by  
\begin{align}
	\label{amplitude-quant}
	\big(\Op_L^{lr}(A)u\big)(x):= \left(\frac{L}{2\pi}\right)^d \int_{\R^d}\! \int_{\R^d} \e^{\i L\xi(x-y)}A(x,y,\xi)u(y)\d y\d\xi,
\end{align}
\begin{align}
	\label{left-quant}
	\big(\Op_L^l(A)u\big)(x):= \left(\frac{L}{2\pi}\right)^d \int_{\R^d} \!\int_{\R^d} \e^{\i L\xi(x-y)}A(x,\xi)u(y)\d y\d\xi
\end{align}
and
\begin{align}
	\label{right-quant}
\big(\Op_L^r(A)u\big)(x):= \left(\frac{L}{2\pi}\right)^d \int_{\R^d} \!\int_{\R^d} \e^{\i L\xi(x-y)}A(y,\xi)u(y)\d y\d\xi,
\end{align}
for every $x\in\R^{d}$. We write $\Op_L(A)$, if $A$ only depends on the variable $\xi$. In this case, we have 
\begin{equation}
	\label{mult-op-matrix-in-mom}
	\Op_L(A) = (\mathcal{F}^{-1} \otimes \mathbb{1}_{n}) \, A(\,\pmb\cdot /L) \, (\mathcal{F} \otimes \mathbb{1}_{n}),
\end{equation} 
where, on the right-hand side, $A$ is to be understood as a multiplication operator in $L^2(\R^d)\otimes\C^n$. 
Thus, in the case of \eqref{mult-op-matrix-in-mom}, $\Op_{L}(A)$ gives rise to a well-defined and bounded operator on 
$L^{2}(\R^{d})\otimes\C^n$, whenever the symbol $A$ is an essentially bounded measurable matrix-valued 
function on $\R^{d}$.

\begin{defn}
	\label{defn-N-Norm-Matrix}
	Given non-negative integers $m_{x},m_{y},m_{\xi} \in \N_{0}$ and a matrix-valued amplitude 
	$A\in C^{\infty}_{b}(\R^{d} \times \R^{d} \times \R^{d}, \C^{n\times n})$, we introduce 
	the \emph{symbol norm}
	\begin{align}
		\label{Definition-N-Matrix}
		\mathbf{N}^{(m_{x},m_{y},m_{\xi})}(A):=\max_{\substack{|\alpha|\leq m_{x} \\ |\beta|\leq m_{y} \\ 
			|\gamma|\leq m_{\xi}}} 
		\;\sup_{x,y,\xi\in\R^{d}}\tr_{\C^n} \big|\partial_x^{\alpha} \partial_y^{\beta} \partial_\xi^{\gamma} A(x,y,\xi)\big|
		< \infty
	\end{align}
	of $A$. Here, $\alpha,\beta,\gamma\in\N_{0}^{d}$ are multi-indices, $|\alpha|:=\sum_{j=1}^{d}|\alpha_j|$ and 
$\partial_{u}^{\alpha}:=\frac{\partial^{|\alpha|}}{\partial_{u_1}^{\alpha_1}\cdot\ldots\cdot\partial_{u_d}^{\alpha_d}}$, where $\partial_{u_j}$ denotes the partial derivative with respect to the $j$th component of the variable $u\in\R^{d}$.
	If $A\in C^{\infty}_{b}(\R^{d} \times \R^{d}, \C^{n\times n})$ is a matrix-valued symbol 
	which does not depend on the variable $y$, respectively $x$,  we set 
	$\mathbf{N}^{(m_{x},m_{\xi})}(A) := \mathbf{N}^{(m_{x},m_{y},m_{\xi})}(A)$ for arbitrary $m_{y}$, respectively 
	$\mathbf{N}^{(m_{y},m_{\xi})}(A) := \mathbf{N}^{(m_{x},m_{y},m_{\xi})}(A)$ for arbitrary $m_{x}$.
\end{defn}

The indicator functions of the bounded measurable subsets $\Lambda, \Gamma \subset \R^{d}$ on $L^2(\R^d)\otimes\C^n$ 
are given by
\begin{align}
\mathbf{1}_\Lambda:=  1_\Lambda \otimes \mathbb{1}_{n}; \ \ \ \ \ \ \  \ \ \mathbf{1}_\Gamma:=  1_\Gamma \otimes \mathbb{1}_{n},
\end{align}
where the respective second factor denotes the $n\times n$-unit matrix.
We also introduce operators induced by matrix-valued symbols $A\in C^{\infty}_{b}(\R^{d} \times \R^{d}, \C^{n\times n})$ with discontinuities in both variables
\begin{align}
\label{def-T-Op}
T_L(A):= T_L(A;\Lambda,\Gamma):= \mathbf{1}_\Lambda \Op_L(\mathbf{1}_\Gamma) \Op_L^l(A) \Op_L(\mathbf{1}_\Gamma) \mathbf{1}_\Lambda
\end{align}
and 
\begin{align}
\label{def-S-Op}
S_L(A):= S_L(A;\Lambda,\Gamma):= \mathbf{1}_\Lambda \Op_L(\mathbf{1}_\Gamma) \Real\big[ \Op_L^l(\Real A) \big]\Op_L(\mathbf{1}_\Gamma) \mathbf{1}_\Lambda,
\end{align}
where we employ the same notation as in the scalar cases \eqref{scalar-T-op} and \eqref{scalar-S-op}.

We want to reduce the traces of these operators to the traces of operators with scalar-valued symbols.
For $\nu,\mu \in\{1,\ldots,n\}$, let $E_{\nu\mu}\in\C^{n\times n}$ be the canonical matrix unit with 
entry one in the $\nu$th row and $\mu$th column and all other entries equal to zero. We expand the matrix-valued symbol $A$ with respect to this matrix basis 
\begin{align}\label{Matrix-Tensor-Product}
A=\sum_{\nu,\mu=1}^n (A)_{\nu\mu} \otimes E_{\nu\mu}
\end{align}
and use linearity to obtain
\begin{align}
	T_L(A)=\sum_{\nu,\mu=1}^n T_L\big((A)_{\nu\mu} \otimes E_{\nu\mu}\big) 
	= \sum_{\nu,\mu=1}^n T_L\big((A)_{\nu\mu}\big)\otimes E_{\nu\mu}.
\end{align}
Given a trace-class operator $T$ on $L^2(\R^d)$ and a matrix $M\in\C^{n\times n}$, their elementary tensor product $T\otimes M$ is trace class on $L^2(\R^d)\otimes\C^n$ with (standard) trace $\tr_{L^2(\R^d)\otimes\C^n} T\otimes M = (\tr_{L^2(\R^d)}T )( \tr_{\C^n} M).$ 
In particular, the operator $T_L(A)$ is trace class with 
\begin{align}
\tr_{L^2(\R^d)\otimes\C^n}T_L(A)=\sum_{\nu=1}^n \tr_{L^2(\R^d)}T_L\big((A)_{\nu\nu}\big)=\tr_{L^2(\R^d)}[T_L(\tr_{\C^n} A)].
\end{align}
In the same way, the operator $S_L(A)$ is seen to be trace class. 

Motivated by the applications to the free Dirac operator in \cite{BM-inprep}, 
it will be useful to extend the asymptotics to slightly more general Wiener--Hopf operators which are 
not only restricted to the momentum region $\Gamma\subset\R^{d}$ as in \eqref{def-T-Op} and 
\eqref{def-S-Op} but, instead, exhibit a more general jump discontinuity at the boundary $\partial\Gamma$ with different matrix-valued symbols $A_1,A_2\in C^{\infty}_{b}(\R^{d} \times \R^{d}, \C^{n\times n})$ on the inside, respectively outside of $\partial\Gamma$. Since at least one of the domains $\Gamma$ and $\Gamma^c$ is not bounded we require the symbols $A_1$, respectively $A_2$, to be compactly supported in the second variable, 
when $\Gamma$, respectively $\Gamma^c$, is not bounded. This guarantees the trace-class property of the operators 
\begin{align}
\label{Definition_DL}
D_L(A_1,A_2):= D_L(A_1,A_2;\Lambda,\Gamma):=  T_L(A_1;\Lambda,\Gamma)+T_L(A_2;\Lambda,\Gamma^c)
\end{align}
and 
\begin{align}\label{Definition_GL}
G_L(A_1,A_2):= G_L(A_1,A_2;\Lambda,\Gamma):=  S_L(A_1;\Lambda,\Gamma)+S_L(A_2;\Lambda,\Gamma^c).
\end{align}
In this case, we also need to adapt the symbol $\mathfrak{A}(g;b)$ in \eqref{frakAsym} featuring in the 
$\mathfrak{W}_1$-coefficient to account for both $A_1$ and $A_2$. Similarly to the situation in $d=1$ dimension for matrix-valued symbols with general jump discontinuities in \cite{Widom1982}, the appropriate replacement appearing in Theorems~\ref{Asymptotics-Gamma-Gamma-Complement},
 \ref{Asymptotics-Analytic}, \ref{Asymptotics-Smooth} and \ref{Asymptotics-Hoelder} is the scalar symbol
\begin{equation}
	\label{frakU-def}
	\mathfrak{U}(g;B_1,B_2):=  \frac{1}{(2\pi)^2}\int_0^1 
	\frac{\tr_{\C^n}\big[g\big(B_1t+B_2(1-t)\big)-g(B_1)t-g(B_2)(1-t)\big]}{t(1-t)}\,\d t
\end{equation}
which is defined for bounded matrix-valued symbols $B_{1}, B_{2}$ and for H\"older continuous functions 
$g: \R \to\C$.

%%%%%%%%%%%%%%%%%%%%%%%%%%%%%%%%%%%%%%%%%%%%%%%%%%%%%%%%%%
%%%%%%%%%%%%%%%%%%%%%%%%%%%%%%%%%%%%%%%%%%%%%%%%%%%%%%%%%%

\section{Asymptotic formula for polynomials}
\label{sec:poly}

In the remaining part of this paper we use the letters $C, C_{1}, C_{2}, C_{\nu\mu}$, etc.\ to denote generic positive constants whose value may differ from line to line.

We first give a proof that the matrix-valued operators 
\eqref{amplitude-quant} -- \eqref{right-quant} are bounded operators on 
$L^2(\R^d)\otimes\C^n$ with operator norms $\|\cdot\|$ uniformly bounded in $L$. 
This is a simple reduction to the scalar case, which is contained in \cite[Chap.\ 3]{sobolevlong}.

\begin{lem}\label{Bounded-Lemma}
Let  $A\in C^{\infty}_{b}(\R^{d} \times \R^{d} \times \R^{d}, \C^{n\times n})$ be a matrix-valued amplitude. 
Then, for every $L\geq 1$ we have 
\begin{align}
\|\Op_L^{lr}(A)\|\leq C \mathbf{N}^{(m,m,d+1)}(A) < \infty,
\end{align}
where $m:=\big\lfloor \frac{d}{2}\big\rfloor+1$, and the constant $C$ is independent of $L$. 
Here, $\lfloor u\rfloor$ stands for the largest integer not 
exceeding $u\in\R$. Clearly, this carries over to
symbols $A\in C^{\infty}_{b}(\R^{d} \times \R^{d}, \C^{n\times n})$ so that
\begin{align}\label{Operatorn-Norm-Estimate-N}
\|\Op_L^t(A)\|\leq C \mathbf{N}^{(m,d+1)}(A) < \infty
\end{align}
for both $t\in\{l,r\}$.
\end{lem}

\begin{proof}
We use (\ref{Matrix-Tensor-Product}) and linearity to write 
\begin{align}
\|\Op_L^{lr}(A)\|=\bigg\|\sum_{\nu,\mu=1}^n \Op_L^{lr}\big( (A)_{\nu\mu} \big) \otimes E_{\nu\mu}\bigg\|\leq \sum_{\nu,\mu=1}^n \big\|\Op_L^{lr}\big((A)_{\nu\mu}\big)\big\|.
\end{align}
By \cite[Lemma 3.9]{sobolevlong} we estimate for every $\nu,\mu \in\{1,\ldots,n\}$
\begin{align}
\big\|\Op_L^{lr}\big((A)_{\nu\mu}\big)\big\|\leq C \mathbf{N}^{(m,m,d+1)}\big((A)_{\nu\mu}\big)
\end{align}
in terms of the symbol norm \eqref{Definition-N-Matrix} for $n=1$. This concludes the proof.
\end{proof}

%%%%%%%%%%%%%%%%%%%%%%%%%%%%%%%%%%%%%%%%%%%%%%%%%%%%%%%%%%%%%%%%%%%%%%%

\subsection{Commutation results}
\label{subsec:commute}

As usual we want to first prove the asymptotics for monomials $g$. We are therefore faced with the task of calculating
\begin{align}
\tr_{L^2(\R^d)\otimes\C^n}\big[\big(T_L(A)\big)^p\big]=\tr_{L^2(\R^d)\otimes\C^n}\left[\Big(\mathbf{1}_\Lambda \Op_L(\mathbf{1}_\Gamma) \Op_L^l(A) \Op_L(\mathbf{1}_\Gamma) \mathbf{1}_\Lambda\Big)^p\right]
\end{align}
for $p\in\N.$ 
In order to reduce this to the scalar-valued case, we need to deal with the matrix structure. Our first step will be to derive some results allowing the commutation of $\Op_L^t(A)$, $t\in\{l,r\}$ with both $\mathbf{1}_\Lambda$ and $\Op_L(\mathbf{1}_\Gamma)$ up to area terms in $L$. 

\begin{defn}
Let $\mathcal{H}$ be a separable Hilbert space. For $q\in \,]0,\infty[\,$ we define the \emph{Schatten--von Neumann class} $\mathcal{T}_q$ as the vector space of all compact (linear) operators $X$ on $\mathcal{H}$ with singular values $s_{k}(X) \in [0,\infty[\,$, $k\in\N$, such that 
\begin{equation}\label{def-q-norm}
\|X\|_q:=\Big(\sum_{k=1}^{\infty}s_{k}(X)^{q}\Big)^{\tfrac{1}{q}} < \infty.
\end{equation}
Definition \eqref{def-q-norm} induces a norm on $\mathcal{T}_q$ for $q\in [1,\infty[\,$. For $q\in \,]0,1[\,$ it induces a quasi-norm for which the $q$-triangle inequality
\begin{equation}
\|X+Y\|_{q}^{q}\leq \|X\|_{q}^{q} + \|Y\|_{q}^{q}
\end{equation}
holds for all $X,Y\in\mathcal{T}_q$.
\end{defn}
\begin{defn}
For $q\in \,]0,1]$ we write $X_L\sim_q Y_L$ for two $L$-dependent operators $X_L,Y_L\in \mathcal{T}_q$ of the corresponding Schatten--von Neumann class $\mathcal{T}_q$ over $L^2(\R^d)\otimes\C^n$, if there exists $C > 0$ such that $\| X_L-Y_L\|_q^q\leq C L^{d-1}$ for all $L\ge 1$. We further write $\sim$ for $\sim_1$. 
\end{defn}

We start with a result that requires more restrictive conditions on the symbol. In the scalar case with the symbol $a$ being compactly supported in both variables and both $\Lambda$ and $\Gamma$ being bounded admissible domains or basic domains, the desired commutation properties were already established in \cite{sobolevschatten}. Our first step will be to extend these properties to matrix-valued symbols $A$ and to general (potentially unbounded) admissible domains $\Lambda$ and $\Gamma$. While the extension to matrix-valued symbols is quite straightforward, the extension to unbounded domains takes more effort. However, this is needed to obtain less strict requirements in the main results of this paper, see Remark \ref{Remark-Gamma-GammaC}.

\begin{lem}\label{Commutation-Lemma}
Let $A\in C^{\infty}_{b}(\R^{d} \times \R^{d}, \C^{n\times n})$ be a matrix-valued symbol with compact support in both variables. 
Let $\Lambda$ and $\Gamma$ be admissible domains. Then for every $q\in \,]0,1]$ and 
$t\in\{l,r\}$ the commutators obey
\begin{align}\label{Commutation-Lambda-Matrix}
\big[\Op_L^t(A),\mathbf{1}_\Lambda\big] \sim_q 0
\end{align}
and
\begin{align}\label{Commutation-Gamma-Matrix}
\big[\Op_L^t(A),\Op_L(\mathbf{1}_\Gamma)\big] \sim_q 0.
\end{align}
\end{lem}
\begin{rem}
Suppose that in addition to the requirements in Lemma \ref{Commutation-Lemma} one has
\begin{equation}
\supp A\subseteq B_{s}(u)\times B_{\tau}(v)
\end{equation}
for some centres $u,v\in\R^d$ and radii $s,\tau>0$ such that $s\geq 1$ and  $L\tau\geq 1$. Then in the scalar-valued case Theorem 4.2 and Remark 4.3 in \cite{sobolevschatten} provide a more precise estimate for the commutators in Lemma \ref{Commutation-Lemma} depending on the radii $s$ and $\tau$ of the support of the symbol $A$. As the proof of Lemma \ref{Commutation-Lemma} will show this estimate extends to the matrix-valued case considered in Lemma \ref{Commutation-Lemma} in the following way: Let $m_{x}:=\big\lfloor \frac{d}{q}\big\rfloor+1$ and $m_{\xi}:=\big\lfloor \frac{d+1}{q}\big\rfloor+1$. Then there exists a constant $C$ which only depends on $q$ and $\Lambda$ such that 
\begin{align}
\big\|[\Op_L^t(A),1_\Lambda]\big\|_q^q\leq C(Ls\tau)^{d-1}\big(\mathbf{N}^{(m_{x},m_{\xi})}(A;s,\tau)\big)^{q},
\end{align}
where $\mathbf{N}^{(m_{x},m_{\xi})}(A;s,\tau):=\max_{\substack{|\alpha|\leq m_{x} \\ |\beta|\leq m_{\xi}}} 
		\;\sup_{x,\xi\in\R^{d}}s^{|\alpha|}\tau^{|\beta|}\tr_{\C^n} \big|\partial_x^{\alpha} \partial_\xi^{\beta} A(x,\xi)\big|$.
\end{rem}
\begin{proof}
Let $L >0$. We first want to reduce $(\ref{Commutation-Lambda-Matrix})$ to the scalar-valued case. We 
use \eqref{Matrix-Tensor-Product} and estimate with the triangle inequality
\begin{align}\label{Triangle-Commutation-Lemma}
	\big\|\big[\Op_L^t(A),\mathbf{1}_\Lambda\big]\big\|_q^q  
	%&=\bigg\|\bigg[\sum_{l,m=1}^n \Op_L^t\big((A)_{lm}\big)\otimes E_{lm}\, ,1_\Lambda\otimes \mathbb{1}_{n}\bing]\bigg\|_q^q & \nonumber \\  &
	\leq \sum_{\nu,\mu=1}^n \big\| \big[\Op_L^t\big((A)_{\nu\mu}\big),1_\Lambda\big]\otimes E_{\nu\mu}\big\|_q^q.
\end{align}
We now want to see that each individual term factorises in the tensor product. To do so, let $X\in\mathcal{T}_q$ be an operator on $L^2(\R^d)$ in the Schatten-von Neumann class corresponding to $q$ and $M\in\C^{n \times n}$ be a matrix. 
Using that $\vert X\otimes M\vert^q = \vert X\vert^q \otimes \vert M \vert^q$, we get
\begin{equation}\label{q-factorization}
	\| X\otimes M\|_q^q 
	= \tr_{L^2(\R^d)\otimes\C^n}\big[\vert X\vert^q \otimes \vert M \vert^q \big]
		%=\tr_{L^2(\R^d)\otimes\C^n}\big[\vert T\vert^q \otimes \vert M\vert^q\big] %\nonumber \\
		= \big(\tr_{L^2(\R^d)}\big[\vert X\vert^q\big] \big) \big(\tr_{\C^n}\big[\vert M \vert^q\big]\big) 
	=\| X\|_q^q \| M\|_q^q.
\end{equation}
We warn the reader that the above equation involves three different Schatten--von Neumann-$q$-norms, on $L^2(\R^d)\otimes\C^n$, on $L^2(\R^d)$ and on $\C^n$ from left to right.
With this at hand we get for each term in the last line of (\ref{Triangle-Commutation-Lemma})
\begin{align}
\big\| \big[\Op_L^t\big((A)_{\nu\mu}\big),1_\Lambda\big]\otimes E_{\nu\mu}\big\|_q^q
= \big\|\big[\Op_L^t\big((A)_{\nu\mu}\big),1_\Lambda\big]\big\|_q^q,
\end{align} 
which is the scalar case. 

If $\Lambda$ is a basic domain, Thm.\ 4.2 and Rem.\ 4.3 in \cite{sobolevschatten} yield constants $C_{\nu\mu}$, which only depend on $q$ and $\Lambda$, such that 
\begin{equation}\label{Commutator-Application-Basic-Domain}
\big\|[\Op_L^t((A)_{\nu\mu}),1_\Lambda]\big\|_q^q\leq C_{\nu\mu}(Ls\tau)^{d-1}\big(\mathbf{N}^{(m_{x},m_{\xi})}(A;s,\tau)\big)^{q}
\end{equation}
for all $\nu,\mu\in\{1,\ldots ,n\}$. Here we used that $|M_{\nu\mu}|\leq \tr |M|$ for every matrix element $M_{\nu\mu}$ of a matrix $M$.
Hence,
\begin{align}\label{Commutator-Result-Finish}
\big\| \big[\Op_L^t(A),\mathbf{1}_\Lambda\big]\big\|_q^q \leq C(Ls\tau)^{d-1}\big(\mathbf{N}^{(m_{x},m_{\xi})}(A;s,\tau)\big)^{q},
\end{align}
where $C:=\sum_{\nu,\mu=1}^{n} C_{\nu\mu}$ only depends on $q$ and $\Lambda$. Therefore, it remains to show that the scalar result \eqref{Commutator-Application-Basic-Domain} extends to arbitrary admissible domains $\Lambda$. It suffices to show 
\begin{align}\label{Pre-Commutator-Result}
\big\| 1_\Lambda \Op_L^t(a) (1-1_\Lambda)\big\|_q^q \leq C(Ls\tau)^{d-1}\big(\mathbf{N}^{(m_{x},m_{\xi})}(a;s,\tau)\big)^{q},
\end{align}
for an arbitrary scalar symbol $a\in C^{\infty}_{b}(\R^{d} \times \R^{d})$ with compact support in both variables, as the estimate for the commutator follows from (\ref{Pre-Commutator-Result}) together with its adjoint and the fact that (\ref{Pre-Commutator-Result}) holds for both the left and the right operator.

In order to prove (\ref{Pre-Commutator-Result}) define 
$\wtilde\Lambda:=\{x\in\R^d : \dist(x,\Lambda)<s\}$ and let $R>0$ such that the support of $a$ in the 
first (i.e.\ space) variable is contained in $B_R(0)$. As $\Lambda$ is admissible and 
$\overline{\wtilde\Lambda\cap B_R(0)}$ is compact, we can cover $\overline{\wtilde\Lambda\cap B_R(0)}$ 
with balls $B_\rho(x_j)$ with radius $\rho>0$ and centres $x_j \in\R^{d}$, where $j\in \mathcal{J} \subset\N$ runs through some finite index set.  
The balls are chosen such that $\Lambda\cap B_{4\rho}(x_j)=\Lambda_j\cap B_{4\rho}(x_j)$ for every $j\in\mathcal{J}$, where $\Lambda_j$ is a basic 
domain. We introduce a smooth partition of unity $\{\phi_j\}_{j\in\mathcal{J}}$ in $\R^{d}$ with 
$\supp\phi_j\subset B_\rho(x_j)$ for $j\in\mathcal{J}$ and  
\begin{align}
	\Phi\big{|}_{\overline{\wtilde\Lambda\cap B_R(0)}}=1, \qquad \text{where} \quad 
	\Phi := \sum_{j\in\mathcal{J}} \phi_{j}
\end{align}
as well as
\begin{equation}\label{Lemma-1-derivative-phi}
\sup_{x\in\R^d}|\partial_{x}^{\alpha}\phi_{j}(x)|\leq C\rho^{-|\alpha|}, \quad \alpha\in\N_{0}^{d},
\end{equation}
where the constant $C$ does not depend on $\rho$.
We start with the easier case $t=l$ and estimate
\begin{equation}
	\label{Lemma-1-Triangle-Partition}
	\big\|1_\Lambda \Op_L^l(a) (1-1_\Lambda)\big\|_q^q
	= \big\|1_\Lambda \Phi \Op_L^l(a) (1-1_\Lambda)\big\|_q^q 
	\leq \sum_{j\in\mathcal{J}} \big\|1_{\Lambda_j} \phi_j \Op_L^l(a) (1-1_\Lambda) \big\|_q^q.
\end{equation}
As the sum is finite, we just need to estimate the individual terms. We already have a basic domain 
to the left of the symbol. Therefore, it remains to also replace the occurrence of $1_\Lambda$ 
on the right-hand side by $1_{\Lambda_j}$ in order to complete the reduction.

For $j\in\mathcal{J}$, let $h_j\in C^\infty(\R^d)$ be a smooth function such that $\|h_j\|_\infty\leq 1$, $\supp(h_j)\subset B_{4\rho}(x_j)$ and $h_j|_{B_{2\rho}(x_j)}=1$. We obtain
\begin{align}
	\label{Lemma-1-Thm3.2}
	\big\|1_{\Lambda_j} \phi_j \Op_L^l(a) (1-1_\Lambda)\big\|_q^q
	&= \big\|1_{\Lambda_j} \phi_j \Op_L^l(a)(h_j+1-h_j) (1-1_\Lambda) \big\|_q^q \nonumber \\ 
	&\leq \big\|1_{\Lambda_j} \phi_j \Op_L^l(a) h_j(1-1_{\Lambda_j}) \big\|_q^q 
		+ \big\|\phi_j \Op_L^l(a) (1-h_j) \big\|_q^q \nonumber \\ 
	&\leq \big\|1_{\Lambda_j}  \Op_L^l(\phi_j a) (1-1_{\Lambda_j}) \big\|_q^q \notag\\
	& \quad + C(L\rho\tau)^{d-m_{\xi}q}\big(\mathbf{N}^{(m_{x},m_{\xi})}(\phi_j a;\rho,\tau)\big)^{q},
\end{align}
with a constant $C$ which only depends on $q$. Here we used \cite[Thm.\ 3.2]{sobolevschatten} with $\alpha=L$, $\ell_{0}=\ell=R=\rho$, $h_1=1_{B_{\rho}(x_j)}$, $h_2=1-h_j$ and $a=\phi_j a$ in the last step. This is possible, as the distance of the supports of $\phi_j$ and $1-h_j$ is at least $\rho$ by construction. Above, we used the notation $\Op_L^l(\phi_j a)$ which is abusive, as it does not specify the variable the function $\phi_j$ depends on. However, it should be still clear from the context and the definition of the function $\phi_j$ that the symbol $\phi_j a$ should be interpreted as 
$\phi_j a: \R^{d}\times\R^{d}\ni (x,\xi)\mapsto\phi_{j}(x)a(x,\xi)$. 

As $\rho$ only depends on $\Lambda$, the bound \eqref{Lemma-1-derivative-phi} allows us to estimate the second term in the last line of \eqref{Lemma-1-Thm3.2} from above by
\begin{align}\label{Lemma-1-Thm3.2-Estimate-2}
C(L\tau)^{d-m_{\xi}q}\big(\mathbf{N}^{(m_{x},m_{\xi})}(a;1,\tau)\big)^{q}\leq C (Ls\tau)^{d-1}\big(\mathbf{N}^{(m_{x},m_{\xi})}(a;s,\tau)\big)^{q},
\end{align}
where the constant $C$ now only depends on $q$ and $\Lambda$ and we used that $L\tau\geq 1$, $m_{\xi}q>d+1>1$
as well as $s\geq 1$ in the last inequality. The estimates \eqref{Lemma-1-Triangle-Partition}, \eqref{Lemma-1-Thm3.2} combined with \eqref{Lemma-1-Thm3.2-Estimate-2} and \eqref{Commutator-Application-Basic-Domain} yield \eqref{Pre-Commutator-Result} for $t=l$.

For the more complicated case $t=r$ write 
\begin{align}
	\big\|1_\Lambda \Op_L^r(a) (1-1_\Lambda) \big\|_q^q
	&= \big\| 1_\Lambda  \Op_L^r(a)(\Phi + 1 - \Phi) (1-1_\Lambda) \big\|_q^q  \notag \\  
	&\leq \sum_{j\in\mathcal{J}} \big\| 1_\Lambda  \Op_L^r(a) \phi_j (1-1_{\Lambda_j}) \big\|_q^q
		+ \big\| 1_\Lambda  \Op_L^r(a)1_{B_R(0)}(1 - \Phi) \big\|_q^q.
\end{align}
Note that $\supp (1-\Phi) \subset \big(\,\overline{\wtilde\Lambda \cap B_R(0)}\,\big)^c \subset  
	\wtilde\Lambda^c \cup \big(B_R(0)\big)^c$ and therefore 
	$\supp \big(1_{B_R(0)}(1-\Phi)\big)\subset \wtilde\Lambda^c$. Hence, the distance between the supports 
	of $1_\Lambda$ and $1_{B_R(0)}(1-\Phi)$ is at least $s$ and \cite[Thm.\ 3.2]{sobolevschatten} yields 
\begin{align}
	\big\| 1_\Lambda \Op_L^r(a) (1-1_\Lambda) \big\|_q^q 
	\leq \sum_{j\in\mathcal{J}} \big\| 1_\Lambda  \Op_L^r(a) \phi_j (1-1_{\Lambda_j}) \big\|_q^q
		+ C(Ls\tau)^{d-m_{\xi}q}\big(\mathbf{N}^{(m_{x},m_{\xi})}(a;s,\tau)\big)^{q}.
\end{align}
From here on we continue similarly to the case of the left operator by inserting $h_j+(1-h_j)$ to 
the left of $\Op_L^r(a)$. This gives (\ref{Pre-Commutator-Result}) for $t=r$. 
Relation $(\ref{Commutation-Gamma-Matrix})$ follows in the same way by interchanging the roles 
of the variables $x$ and $\xi$.
\end{proof}

If we have a symbol with compact support in both variables, there are several useful estimates for the left and right operators in the trace norm \cite[Chap.\ 3]{sobolevlong}. We also want to generalise these to the matrix-valued case.

\begin{lem}
	\label{Exchange-Lemma}
	Let $A,B\in C^{\infty}_{b}(\R^{d} \times \R^{d}, \C^{n\times n})$ be matrix-valued symbols and 
	$F\in C^{\infty}_{b}(\R^{d} \times \R^{d} \times \R^{d}, \C^{n\times n})$ a matrix-valued amplitude. 
	We assume that $B$  has compact support in both variables 
	and $F$ is compactly supported in $\xi$ and in at least one of the variables $x$ or $y$. 
	We write $D$ for the symbol given by $D(x,\xi):= F(x,x,\xi)$ for $x,\xi\in\R^{d}$. Then, 
	for $t\in\{l,r\}$ we have
	\begin{align}
		\label{Left-Right}
		\Op_L^l(B)\sim \Op_L^r(B)  
	\end{align}
	and 
	\begin{align}
		\label{Left-LeftRight}
		\Op_L^{lr}(F)\sim \Op_L^l(D).
	\end{align}
	Further, we have
	\begin{align}
		\label{Merge-Symbols}
		\Op_L^t(A)\Op_L^t(B) \sim \Op_L^t(AB)
	\end{align}
	and
	\begin{align}
		\label{Merge-Symbols-2}
		\Op_L^t(B)\Op_L^t(A) \sim \Op_L^t(BA).
	\end{align}
\end{lem}

\noindent
In the proof of Lemma~\ref{Exchange-Lemma} we provide more information on the constants arising in \eqref{Left-Right} -- \eqref{Merge-Symbols-2}, see \eqref{Left-Right-Precise}, \eqref{Left-LeftRight-Precise}, \eqref{Merge-Symbols-Precise} and
\eqref{Merge-Symbols-Precise-2}.

\begin{proof}[Proof of Lemma~\ref{Exchange-Lemma}]
Let $L \ge 1$. 
In order to reduce $(\ref{Left-LeftRight})$ to the scalar case, we again use $(\ref{Matrix-Tensor-Product})$ and $(\ref{q-factorization})$. The estimates in the scalar case can be found in \cite[Lemma 3.12]{sobolevlong}. Note that this Lemma even provides the estimate 
\begin{align}
	\big\| \Op_L^{lr}\big((F)_{\nu\mu}\big)-\Op_L^l\big((D)_{\nu\mu}\big)\big\|_1
 	\leq C L^{d-1}\mathbf{N}^{(d+1,d+1,d+2)}\big((F)_{\nu\mu}\big),
\end{align}
for every $\nu,\mu \in\{1,\ldots,n\}$, where the constant $C$ still  depends on the support of the 
amplitudes $(F)_{\nu\mu}$, but is otherwise independent of $(F)_{\nu\mu}$. Therefore we obtain 
\begin{align}\label{Left-LeftRight-Precise}
\| \Op_L^{lr}(F)- \Op_L^l(D)\|_1\leq C L^{d-1} \mathbf{N}^{(d+1,d+1,d+2)}(F)
\end{align}
in terms of the norm \eqref{Definition-N-Matrix} for matrix-valued amplitudes.
The (new) constant $C$ still depends on the support of the amplitude $F$, 
but is independent of $F$ otherwise. In the special case $F(x,y,\xi) = B(y,\xi)$ we obtain
\begin{align}\label{Left-Right-Precise}
\| \Op_L^{r}(B)- \Op_L^l(B)\|_1\leq C L^{d-1} \mathbf{N}^{(d+1,d+2)}(B),
\end{align}
which proves $(\ref{Left-Right})$.

For $(\ref{Merge-Symbols})$ we use $(\ref{Matrix-Tensor-Product})$ to write
\begin{align}
	\Op_L^t(A)\Op_L^t(B) 
	&= \bigg(\sum_{\nu,\mu=1}^{n} \Op_L^t\big((A)_{\nu\mu}\big)\otimes E_{\nu\mu}\bigg) 
		 \bigg(\sum_{\sigma,\tau=1}^{n} \Op_L^t\big((B)_{\sigma\tau}\big)\otimes E_{\sigma\tau}\bigg) \nonumber \\ 
	&= \sum_{\nu,\mu,\tau=1}^{n} \Big(\Op_L^t\big((A)_{\nu\mu}\big) \, \Op_L^t\big((B)_{\mu\tau}\big)\Big) \otimes E_{\nu\tau}
\end{align}
and
\begin{align}
	\Op_L^t(AB)=\sum_{\nu,\tau=1}^{n} \Op_L^t\big((AB)_{\nu\tau}\big) \otimes E_{\nu\tau}
	= \sum_{\nu,\mu,\tau=1}^{n} \Op_L^t\big((A)_{\nu\mu}(B)_{\mu\tau}\big) \otimes E_{\nu\tau}.
\end{align}
Combining these two equalities, the triangle inequality and $(\ref{q-factorization})$, we get
\begin{equation}
 	\big\| \Op_L^t(A)\Op_L^t(B) - \Op_L^t(AB) \big\|_1    
	\leq \sum_{\nu,\mu,\tau=1}^{n} \big\| \Op_L^t\big((A)_{\nu\mu}\big) \Op_L^t\big((B)_{\mu\tau}\big)  
			- \Op_L^t\big((A)_{\nu\mu}(B)_{\mu\tau}\big) \big\|_1. 
\end{equation}
The corresponding scalar estimates for $t=l$ are contained in \cite[Cor.\ 3.13]{sobolevlong}. The estimates for $t=r$ follow from taking the adjoint. Cor.\ 3.13 in \cite{sobolevlong} gives even a more precise estimate
\begin{multline}
	\big\| \Op_L^t\big((A)_{\nu\mu}\big) \Op_L^t((B)_{\mu\tau})  - \Op_L^t\big((A)_{\nu\mu}(B)_{\mu\tau}\big) \big\|_1  \\
	\leq C L^{d-1} \mathbf{N}^{(d+1,d+2)}\big((A)_{\nu\mu}\big) \, \mathbf{N}^{(d+1,d+2)}\big((B)_{\mu\tau}\big)
\end{multline}
for every $\nu,\mu,\tau \in\{1,\ldots,n\}$, where $C$ still depends on the support of the symbol $B$, 
but is independent of $B$ otherwise.
Therefore,
\begin{align}
	\label{Merge-Symbols-Precise}
 	\big\| \Op_L^t(A) \Op_L^t(B) - \Op_L^t(AB) \big\|_1 
	\leq C L^{d-1} \mathbf{N}^{(d+1,d+2)}(A)\, \mathbf{N}^{(d+1,d+2)}(B),
\end{align}
where $C$ again depends on the support of the symbol $B$, but is independent of $B$ otherwise. This proves \eqref{Merge-Symbols}. The proof of \eqref{Merge-Symbols-2} is analogous and yields 
\begin{align}
	\label{Merge-Symbols-Precise-2}
 	\big\| \Op_L^t(B) \Op_L^t(A) - \Op_L^t(BA) \big\|_1 
	\leq C L^{d-1} \mathbf{N}^{(d+1,d+2)}(A)\, \mathbf{N}^{(d+1,d+2)}(B). 
\end{align}
\end{proof}
In order to prove the asymptotic formula for polynomials it is convenient to localise the problem, i.e. to locally replace the admissible domains $\Lambda$ and $\Gamma$ by appropriate basic domains. The following lemma will be helpful in that regard.

\begin{lem}
	\label{Matrix-Commutation}
	Let $A,B\in C^{\infty}_{b}(\R^{d} \times \R^{d}, \C^{n\times n})$ be matrix-valued symbols and 
	assume that $B$ is compactly supported in both variables.
	Let $\Lambda$ and $\Gamma$ be 
	admissible domains. Further, let $\Lambda_0$ 
	and $\Gamma_0$ be basic domains such that
	\begin{align}
		B\big|_{\Lambda\times\R^d} = B\big|_{\Lambda_0\times\R^d} \qquad \text{and} \quad   
		B\big|_{\R^d\times\Gamma} = B\big|_{\R^d\times\Gamma_0}.
	\end{align}
 	For $p\in\N$ we have
	\begin{align}
		\label{Chain-Commutation}
		\Op_L^l(B) \big(T_L(A;\Lambda,\Gamma)\big)^p \sim \Op_L^l(BA^p) 
			\big( T_L(\mathbb{1}_{n};\Lambda_0,\Gamma_0) \big)^p.
	\end{align}
\end{lem}

\begin{proof}
The symbol $B$ is compactly supported in both variables. Therefore, we can apply Lemma $\ref{Exchange-Lemma}$ and Lemma $\ref{Commutation-Lemma}$. By assumption we have $\mathbf{1}_\Lambda \Op_L^l(B)=\mathbf{1}_{\Lambda_0}\Op_L^l(B)$ and $\Op_L^l(B)\Op_L(\mathbf{1}_\Gamma)=\Op_L^l(B)\Op_L(\mathbf{1}_{\Gamma_0})$. Therefore, we get 
\begin{align}
\Op_L^l(B)\mathbf{1}_\Lambda \sim \mathbf{1}_\Lambda \Op_L^l(B) = \mathbf{1}_{\Lambda_0}\Op_L^l(B) 
\end{align}
and
\begin{align}
\Op_L^l(B)\Op_L(\mathbf{1}_\Gamma)=\Op_L^l(B)\Op_L(\mathbf{1}_{\Gamma_0})\sim \Op_L(\mathbf{1}_{\Gamma_0}) \Op_L^l(B).
\end{align}
Furthermore, we can apply $(\ref{Merge-Symbols})$ to get
\begin{align}
\Op_L^l(B)\Op_L^l(A)\sim \Op_L^l(BA).
\end{align}
We note that the three relations above still hold, if one replaces the symbol $B$ by $BA^m$ for some $m\in \N$. Repeatedly applying the relations, starting from the left, yields
\begin{align}
	\Op_L^l(B) \big(T_L(A;\Lambda,\Gamma)\big)^p 
	\sim \big(T_L(\mathbb{1}_{n};\Lambda_0,\Gamma_0) \big)^p \Op_L^l(BA^p).
\end{align}
The symbol $BA^p$ is still compactly supported in both variables. Therefore, we can repeatedly apply Lemma $\ref{Commutation-Lemma}$ to get the desired result.
\end{proof}

%%%%%%%%%%%%%%%%%%%%%%%%%%%%%%%%%%%%%%%%%%%%%%%%%%%%%%%%%%%%%%

\subsection{Asymptotic formula}

\label{subsec:polyproof}

The next crucial ingredient will be a local asymptotic formula for basic domains. 
This is a generalisation  to matrix-valued symbols of the appropriate scalar results, 
namely \cite[Thm.\ 4.1]{sobolevpw} in the case that $d\geq 2$ and the initial result by 
Widom \cite{Widom1982} in the case $d=1$.

\begin{thm}\label{Local-Asymptotics}
	Let $p\in\N$, let $\Lambda$ be a piece-wise $C^1$-basic domain and $\Gamma$ be a piece-wise $C^3$-basic domain.
	Further let $B\in C^{\infty}_{b}(\R^{d} \times \R^{d}, \C^{n\times n})$ be a matrix-valued symbol with compact support in both variables. Then
	\begin{align}
		\tr_{L^2(\R^d)\otimes\C^n}\Big[\Op_L^l(B) \big( 
			T_L(\mathbb{1}_{n};\Lambda,\Gamma)\big)^p\Big]  %\notag\\
		%=& \ 
		&=L^d \mathfrak{W}_0\big(\tr_{\C^n}[B];\Lambda,\Gamma\big) \nonumber \\
		&\quad + L^{d-1}\log L\;\mathfrak{W}_1\big(\tr_{\C^n}[B]\,
			\mathfrak{A}(\id^p;1);\partial\Lambda,\partial\Gamma\big) \nonumber \\ 
		&\quad + o(L^{d-1}\log L),
	\end{align}
	as $L \rightarrow \infty$. Here, $\id^p$ is the monomial of order $p$, the coefficients 
	$\mathfrak{W}_0$ and $\mathfrak{W}_1$ are defined in \eqref{coeff-w0-def} and \eqref{coeff-w1-def},
	respectively, and we refer to \eqref{frakAsym} for the definition of the symbol $\mathfrak{A}$.
\end{thm}

\begin{proof}
Writing out the symbol in form of a tensor product and using that 
$\tr_{\C^n}[E_{\nu\mu}]=\delta_{\nu,\mu}$, where $\delta_{\nu,\mu}$ is the Kronecker delta, we arrive at
\begin{align}
	\tr_{L^2(\R^d)\otimes\C^n}\Big[\Op_L^l(B) & \big(T_L(\mathbb{1}_{n};\Lambda,\Gamma)\big)^p\Big] 
		\nonumber\\
	&=\tr_{L^2(\R^d)\otimes\C^n}\bigg[\sum_{\nu,\mu=1}^n \Op_L^l\big((B)_{\nu\mu} \otimes E_{\nu\mu}\big)
			\big(T_L(1)\otimes \mathbb{1}_{n}\big)^p\bigg] \nonumber\\ 
	&=\tr_{L^2(\R^d)\otimes\C^n}\bigg[\sum_{\nu,\mu=1}^n \Op_L^l\big((B)_{\nu\mu}\big) \big(T_L(1)\big)^p
			\otimes E_{\nu\mu}\bigg] \nonumber\\ 
	&=\sum_{\nu=1}^n \tr_{L^2(\R^d)}\Big[ \Op_L^l\big((B)_{\nu\nu}\big) \big(T_L(1)\big)^p\Big]\nonumber\\ 
	&=\tr_{L^2(\R^d)}\Big[\Op_L^l\big(\tr_{\C^n}[B]\big) \big(T_L(1)\big)^p\Big].
\end{align}
The asymptotics for this last term is given by \cite[Thm.\ 4.1]{sobolevpw} in the case that $d\geq 2$:
\begin{align}
	\tr_{L^2(\R^d)}\left[\Op_L^l\big(\tr_{\C^n}[B]\big) \big(T_L(1)\big)^p\right]
	=& L^d \mathfrak{W}_0\big(\tr_{\C^n}[B];\Lambda,\Gamma\big)\nonumber\\
		&+ L^{d-1}\log L \ \mathfrak{W}_1\big(\tr_{\C^n}[B]\,\mathfrak{A}(\id^p;1);\partial\Lambda,\partial\Gamma\big)
			\nonumber\\ 
		&+ o(L^{d-1}\log L),
\end{align}
as $L\rightarrow\infty$.  \\
For the case $d=1$ we use the initial result by Widom \cite{Widom1982}. As translation, reflection and time reversal are unitary operators, we just need to consider the case 
$\Lambda=\Gamma = \,]0,\infty[\,$. Commuting $\Op_L^l(\tr_{\C^n}[B])$ to the right is possible up to area terms by Lemma $\ref{Commutation-Lemma}$. Taking $f\equiv \id^p$, we are now in the situation of \cite[Eq.\ (12)]{Widom1982} up to multiplication with the constant $\tr_{\C^n}[B](0,0)$. The desired asymptotics follows.
\end{proof}

With these tools at hand, we are now ready to prove the asymptotics for polynomials

\begin{thm}
	\label{Asymptotics-Polynomial}
	Let $p\in\N$, let $A, B \in C^{\infty}_{b}(\R^{d} \times \R^{d}, \C^{n\times n})$ be matrix-valued symbols. 
	Let $\Lambda$ be a bounded piece-wise 
	$C^1$-admissible domain and $\Gamma$ be a piece-wise $C^3$-admissible domain. 
	We assume $\Gamma$ to be bounded or $A$ to be compactly supported in the second variable. Then
	\begin{align}
		\tr_{L^2(\R^d)\otimes\C^n}\left[\Op_L^l(B)\big(T_L(A;\Lambda,\Gamma)\big)^p\right]
		=& \ L^d \mathfrak{W}_0\big(\tr_{\C^n}[BA^p];\Lambda,\Gamma\big) \nonumber\\ 
			&+ L^{d-1}\log L \ \mathfrak{W}_1\big(\tr_{\C^n}[BA^p]\, 
				\mathfrak{A}(\id^p;1);\partial\Lambda,	\partial\Gamma\big)\nonumber\\ &+o(L^{d-1}\log L),
	\end{align}
	as $L \rightarrow \infty$. Here, $\id^p$ is the monomial of order $p$, the coefficients 
	$\mathfrak{W}_0$ and $\mathfrak{W}_1$ are defined in \eqref{coeff-w0-def} and \eqref{coeff-w1-def},
	respectively, and we refer to \eqref{frakAsym} for the definition of the symbol $\mathfrak{A}$.
\end{thm}

\begin{proof}
We give a proof in the case that $\Gamma$ is unbounded and $A$ is compactly supported in the second variable. The other case works similarly and is slightly easier.

Let $R>0$ such that the support of $A$ in the second variable is contained in $B_R(0)$. 
By the definition of $T_L$ we have 
\begin{multline}
	\Op_L^l(B)\big(T_L(A;\Lambda,\Gamma)\big)^p \\ 
	= \Op_L^l(B)\mathbf{1}_\Lambda \Op_L(\mathbf{1}_\Gamma) \Op_L^l(A) \Op_L(\mathbf{1}_{B_R(0)}) \Op_L(\mathbf{1}_\Gamma) \mathbf{1}_\Lambda\big(T_L(A;\Lambda,\Gamma)\big)^{p-1}.
\end{multline}
Both $\Lambda$ and $B_R(0)\cap\Gamma$ are bounded, therefore we can cover their closures with finitely many open balls such that $\Lambda$, respectively $\Gamma$, is represented by a basic domain denoted by $\Lambda_j$ for $j\in\mathcal{J}$, respectively $\Gamma_k$ for $k\in\mathcal{K}$, when restricted to any of such balls. Here, $\mathcal{J},\mathcal{K} \subset \N$ are two finite index sets. 
We denote a partition of unity subordinate to the covering of $\Lambda$ by 
$\{\phi_j\}_{j\in\mathcal{J}}$ and a partition of unity subordinate to the covering 
of $B_R(0)\cap\Gamma$ by $\{\psi_k\}_{k\in\mathcal{K}}$. 
By the construction of the coverings we get
\begin{align}\label{thm-poly-eq-1}
	\Op_L(\mathbf{1}_{B_R(0)}) \Op_L(\mathbf{1}_\Gamma) \mathbf{1}_\Lambda 
	= \sum_{j\in\mathcal{J},\; k\in\mathcal{K}} \Op_L(\mathbf{1}_{B_R(0)\cap\Gamma}) \Op_L^r(\phi_j\psi_k) \mathbf{1}_\Lambda.
\end{align}
We note that the symbol $\phi_j\psi_k$ of the right operator  $\Op_L^r(\phi_j\psi_k)$ is to be understood as $\phi_j\psi_k:\R^{d}\times\R^{d}\ni (y,\xi)\mapsto\phi_{j}(y)\psi_{k}(\xi)$. As in \eqref{Lemma-1-Thm3.2} this is an abuse of notation and it is only clear from the definitions of the functions $\{\phi_j\}_{j\in\mathcal{J}}$ and $\{\psi_k\}_{k\in\mathcal{K}}$. 
With \eqref{thm-poly-eq-1}, we obtain
\begin{multline}
\Op_L^l(B)\mathbf{1}_\Lambda \Op_L(\mathbf{1}_\Gamma) \Op_L^l(A)\Op_L(\mathbf{1}_\Gamma)\mathbf{1}_\Lambda \\ = \sum_{j\in\mathcal{J},\; k\in\mathcal{K}}\Op_L^l(B)\mathbf{1}_\Lambda \Op_L(\mathbf{1}_\Gamma) \Op_L^l(A)\Op_L(\mathbf{1}_\Gamma) \Op_L^r(\phi_j\psi_k)\mathbf{1}_\Lambda.
\end{multline}
We note that the symbol $\phi_j\psi_k$ is compactly supported in both variables. Therefore, we can establish the following relations with the help of Lemma $\ref{Commutation-Lemma}$ and $\ref{Exchange-Lemma}$
for every $j\in\mathcal{J}$ and $k\in\mathcal{K}$
\begin{align}
 \Op_L(\mathbf{1}_\Gamma) \Op_L^r(\phi_j\psi_k)\sim \Op_L^r(\phi_j\psi_k)&\Op_L(\mathbf{1}_{\Gamma})\sim \Op_L^l(\phi_j\psi_k)\Op_L(\mathbf{1}_\Gamma), \nonumber \\ \nonumber \\
 \Op_L^l(A)\Op_L^l(\phi_j\psi_k)\sim \Op_L^l(A\phi_j\psi_k)\sim& \Op_L^l(\phi_j\psi_k)\Op_L^l(A)\sim \Op_L^r(\phi_j\psi_k)\Op_L^l(A), \nonumber \\ \nonumber \\
 \mathbf{1}_\Lambda \Op_L^l(\phi_j\psi_k) \sim& \Op_L^l(\phi_j\psi_k)\mathbf{1}_\Lambda.
\end{align}
Combining them, yields
\begin{align}
	\Op_L^l(B)\big(T_L(A;\Lambda,\Gamma)\big)^p
	&\sim \sum_{j\in\mathcal{J},\; k\in\mathcal{K}} \Op_L^l(B)\Op_L^l(\phi_j\psi_k)
			\big(T_L(A;\Lambda,\Gamma)\big)^p \nonumber \\ 
	&\sim\sum_{j\in\mathcal{J},\; k\in\mathcal{K}}\Op_L^l(B_{j,k}) \big(T_L(A;\Lambda,\Gamma)\big)^p,
\end{align}
where the matrix-valued symbol $B_{j,k}$ is defined by $B_{j,k}(x,\xi):= B(x,\xi)\phi_j(x) \psi_k(\xi)$
for $x,\xi\in\R^{d}$. We note that the symbol $B_{j,k}$ is compactly supported in both variables and that 
\begin{align}
B_{j,k}|_{\Lambda\times\R^d}=B_{j,k}|_{\Lambda_j\times\R^d} \ \ \text{and} \ \ B_{j,k}|_{\R^d\times\Gamma}=B_{j,k}|_{\R^d\times\Gamma_k}.
\end{align}
Therefore, we are able to apply Lemma $\ref{Matrix-Commutation}$ and conclude 
\begin{align}
	\Op_L^l(B)\big(T_L(A;\Lambda,\Gamma)\big)^p
	\sim \sum_{j\in\mathcal{J},\; k\in\mathcal{K}} \Op_L^l(B_{j,k}A^p)
			\big(T_L(\mathbb{1}_{n};\Lambda_j,\Gamma_k)\big)^p.
\end{align}

Now both $\Lambda_j$ and $\Gamma_k$ are basic domains. Therefore, we can apply Theorem $\ref{Local-Asymptotics}$ to get the asymptotics
\begin{align}
	&\hspace{-1em}\tr_{L^2(\R^d)\otimes\C^n}\left[\Op_L^l(B_{j,k}A^p) 
			\big(T_L(\mathbb{1}_{n};\Lambda_j, \Gamma_k)\big)^p\right] \nonumber\\
	&= L^d \mathfrak{W}_0\big(\tr_{\C^n}[B_{j,k}A^p];\Lambda_j,\Gamma_k\big)
			+  L^{d-1}\log L\; \mathfrak{W}_1\big(\tr_{\C^n}[B_{j,k}A^p]\,\mathfrak{A}(\id^p;1);
				\partial\Lambda_j,\partial\Gamma_k\big) \notag\\
	&\quad +o(L^{d-1}\log L),
\end{align}
as $L\rightarrow\infty$. As $\Lambda$ and $\Gamma$ are locally represented by $\Lambda_j$ and $\Gamma_k$, we can replace each occurrence of the basic domains by $\Lambda$ respectively $\Gamma$. Using the linearity of the coefficients, we get
\begin{align}
	\tr_{L^2(\R^d)\otimes\C^n} &\Big[\sum_{j\in\mathcal{J},\; k\in\mathcal{K}}\Op_L^l(B_{j,k}A^p)
		\big(T_L(\mathbb{1}_{n};\Lambda_j,\Gamma_k)\big)^p\Big] \notag\\
	&= L^d \mathfrak{W}_0\big(\tr_{\C^n}[BA^p];\Lambda,\Gamma\big)
			+  L^{d-1}\log L\; \mathfrak{W}_1\big(\tr_{\C^n}[BA^p]\,\mathfrak{A}(\id^p;1);\partial\Lambda,
				\partial\Gamma\big)	\notag\\
	&\quad + o(L^{d-1}\log L),
\end{align}
which concludes the proof of the theorem.
\end{proof}

%%%%%%%%%%%%%%%%%%%%%%%%%%%%%%%%%%%%%%%%%%%%%%%%%%%%%%%%%%%%%%%%%%%%

\subsection{Extension to more general Wiener--Hopf operators}
\label{subsec:sumsymbols}

In this section we will generalise the results from Section~\ref{subsec:polyproof} to treat more general 
jump discontinuities at the boundary $\partial\Gamma$ of the momentum region $\Gamma\subset \R^{d}$
with different matrix-valued symbols $A_{1}$ and $A_{2}$ 
on the inside, respectively outside. 
The non-commutativity of $A_{1}$ and $A_{2}$ represents an additional technical challenge in the proof of the next theorem for matrix-valued symbols as compared to the scalar case \cite[Thm.\ 5.2]{sobolevc2} . 
We recall the definitions of the Wiener–Hopf operators $D_{L}(A_{1}, A_{2})$ and $G_{L}(A_{1}, A_{2})$ in \eqref{Definition_DL} and \eqref{Definition_GL}.

\begin{thm}\label{Asymptotics-Gamma-Gamma-Complement}
Let $p\in\N$, let $A_1,A_2\in C^{\infty}_{b}(\R^{d} \times \R^{d}, \C^{n\times n})$ be matrix-valued symbols and assume $A_2$ to be compactly supported in the second variable. Let $\Lambda$ be a bounded piece-wise $C^1$-admissible domain and $\Gamma$ be a piece-wise $C^3$-admissible domain. We assume $\Gamma$ to be bounded or $A_1$ to be compactly supported in the second variable. Then
\begin{align}\label{Asymptotics-Gamma-Gamma-Complement-Claim}
\tr_{L^2(\R^d)\otimes\C^n}\big[ \big(D_L(A_1,A_2;\Lambda,\Gamma)\big)^p\big]=& \ L^d \big[ \mathfrak{W}_0\big(\tr_{\C^n}[ A_1^p];\Lambda,\Gamma\big)+\mathfrak{W}_0\big(\tr_{\C^n}[ A_2^p];\Lambda,\Gamma^c\big) \big]\nonumber\\&+L^{d-1}\log L \ \mathfrak{W}_1\big(\mathfrak{U}(\id^p; A_1, A_2);\partial\Lambda,\partial\Gamma\big)\nonumber\\ &+o(L^{d-1}\log L),
\end{align}
	as $L \rightarrow \infty$. Here, $\id^p$ is the monomial of order $p$, the coefficients 
	$\mathfrak{W}_0$ and $\mathfrak{W}_1$ are defined in \eqref{coeff-w0-def} and \eqref{coeff-w1-def},
	respectively, and we refer to \eqref{frakU-def} for the definition of the symbol $\mathfrak{U}$.
\end{thm}
\begin{proof}
We give a proof in the case that $\Gamma$ is bounded and $A_1$ is not compactly supported in the second variable. The other case works similarly.

As $\Lambda$ and $\Gamma$ are bounded we can find real-valued functions $\phi,\psi\in C_c^\infty(\R^d)$ such that 
\begin{align}\label{Definition-Phi-Psi}
\phi|_\Lambda = 1  \ \ \text{and}  \ \ \psi|_\Gamma= 1.
\end{align}
By construction we have $\mathbf{1}_\Lambda \Op_L(\mathbf{1}_\Gamma)=\mathbf{1}_\Lambda \Op_L^l(\phi\psi) \Op_L(\mathbf{1}_\Gamma)$. We rewrite 
\begin{align}
T_L(A_1;\Lambda,\Gamma) =\mathbf{1}_\Lambda \Op_L^l(\phi\psi) \Op_L(\mathbf{1}_\Gamma)\Op_L^l(A_1)\Op_L(\mathbf{1}_\Gamma)\mathbf{1}_\Lambda.
\end{align}
Using Lemma $\ref{Commutation-Lemma}$ and $\ref{Exchange-Lemma}$, we get
\begin{align}\label{Gamma-Reduction-Compact-Symbol}
\mathbf{1}_\Lambda \Op_L^l(\phi\psi) \Op_L(\mathbf{1}_\Gamma)\Op_L^l(A_1)\Op_L(\mathbf{1}_\Gamma)\mathbf{1}_\Lambda \sim  \mathbf{1}_\Lambda  \Op_L(\mathbf{1}_\Gamma)\Op_L^l(\phi\psi A_1)\Op_L(\mathbf{1}_\Gamma)\mathbf{1}_\Lambda.
\end{align}
As the symbol $\phi\psi A_1$ is compactly supported in both variables, we can commute further to get
\begin{align}\label{Gamma-Reduction}
T_L(A_1;\Lambda,\Gamma)\sim\mathbf{1}_\Lambda \Op_L^l(\phi\psi A_1)\Op_L(\mathbf{1}_\Gamma)\mathbf{1}_\Lambda = \mathbf{1}_\Lambda \Op_L^l(A_1)\Op_L(\mathbf{1}_\Gamma)\mathbf{1}_\Lambda.
\end{align}
For the operator $T_L(A_2;\Lambda,\Gamma^c)$ we can argue in a similar fashion. Let $\zeta\in C_c^\infty(\R^d)$ be a function such that $\zeta$ equals $1$ on the support of $A_2$ in its second variable. Write
\begin{align}\label{Gamma-Complement-Reduction}
	T_L(A_2;\Lambda,\Gamma^c) 
	&= \mathbf{1}_\Lambda \Op_L(\mathbf{1}_{\Gamma^c}) \Op_L^l(A_2)\Op_L(\mathbf{1}_{\Gamma^c})
			\Op_L^r(\phi\zeta)\mathbf{1}_\Lambda \nonumber \\ 
	&\sim  \mathbf{1}_\Lambda \Op_L(\mathbf{1}_{\Gamma^c}) \Op_L^l(A_2\phi\zeta)
			\Op_L(\mathbf{1}_{\Gamma^c})\mathbf{1}_\Lambda\nonumber \\ 
	&= \mathbf{1}_\Lambda \Op_L(\mathbf{1}_{\Gamma^c}) \Op_L^l(A_2\phi)
			\Op_L(\mathbf{1}_{\Gamma^c})\mathbf{1}_\Lambda \nonumber \\ 
	&\sim \mathbf{1}_\Lambda  \Op_L^l(A_2)\Op_L(\mathbf{1}_{\Gamma^c})\mathbf{1}_\Lambda.
\end{align}
Here the commutation with $\Op_L(\mathbf{1}_{\Gamma^c})$ is possible because $\Op_L(\mathbf{1}_{\Gamma^c})=\Op_L(\mathbf{1}_{(\overline{\Gamma})^c})$, as is explained in Remark~\ref{Remark-Complement-Admissible}(b).
Combining (\ref{Gamma-Reduction}), (\ref{Gamma-Complement-Reduction}) and (\ref{Gamma-Reduction}) for $A_1-A_2$, we get
\begin{align}
D_L(A_1,A_2)= T_L(A_1;\Lambda,\Gamma)+T_L(A_2;\Lambda,\Gamma^c)\sim T_L(A_2;\Lambda,\R^d)+T_L(A_1-A_2;\Lambda,\Gamma).
\end{align}
The next step will be to compute the trace of $\big(D_L(A_1,A_2)\big)^p$ for $p\in\N$. The operators 
$T_L(A_2;\Lambda,\R^d)$ and $T_L(A_1-A_2;\Lambda,\Gamma)$ do not even commute up to area terms because the matrix-valued
symbols $A_2$ and $B:= A_1-A_2$ do not necessarily commute. Therefore, we need to be more careful than in the scalar-valued case. In order to write out the result 
we use the expansion
\begin{align}\label{Expansion-Matrices}
(X+Y)^p=\sum_{k=0}^p\frac{1}{k!(p-k)!}\sum_{\pi\in \mathcal S_p}Z_{\pi(1)}^{(k)}\cdots Z_{\pi(p)}^{(k)}
\end{align}
for $X$ and $Y$ elements of an associative algebra and 
\begin{align}
Z_{j}^{(k)}:= X \ \ \text{if} \ \ j>k, \ \ \ \ \
Z_{j}^{(k)}:= Y \ \ \text{if} \ \ j\leq k, \ \ \ \ \ \text{for} \ \ k\in\{0,\ldots,p\}, j\in\{1,\ldots,p\}.
\end{align}
Here $\mathcal S_p$ denotes the symmetric group on the set $\{1,\ldots,p\}$. Applying this with $X=T_L(A_2;\Lambda,\R^d)$ and 
$Y=T_L(B;\Lambda,\Gamma)$, yields
\begin{equation}\label{Expansion-D_L}
	\big(D_L(A_1,A_2)\big)^p \sim \sum_{k=0}^p\frac{1}{k!(p-k)!}\sum_{\pi\in \mathcal S_p}Z_{\pi(1)}^{(k)}\cdots Z_{\pi(p)}^{(k)} .
\end{equation}
We note that $X=T_L(\phi A_2;\Lambda,\R^d)$, where the symbol $\phi A_2$ is compactly supported in both variables.
The goal is to move all occurrences of $\Op_L^l(\phi A_2)$ and $\Op_L^l(B)$ to the left in order to apply Theorem \ref{Asymptotics-Polynomial} for the asymptotic evaluation of $\tr_{L^2(\R^d)\otimes\C^n}  \big[\big(D_L(A_1,A_2)\big)^p\big]$. To this end, we introduce the symbols
\begin{align}
C_{j}^{(k)}:= A_2 \ \ \text{if} \ \ j>k, \ \ \ \ \
C_{j}^{(k)}:= B \ \ \text{if} \ \ j\leq k, \ \ \ \ \ \text{for} \ \ k\in\{0,\ldots,p\}, j\in\{1,\ldots,p\}.
\end{align}
Since the symbol $B$ is not necessarily supported in both variables, the established commutation results in Section \ref{subsec:commute} do not apply directly to it. For this reason we will treat the term with $k=p$ in \eqref{Expansion-D_L} later, as it does not contain a factor $\Op_L^l(\phi A_2)$. For $k\in\{0,\ldots,p-1\}$ such a factor is present, Lemma $\ref{Commutation-Lemma}$ permits to commute it with $\mathbf{1}_\Lambda$ and $\Op_L(\mathbf{1}_\Gamma)$ up to area terms and Lemma $\ref{Exchange-Lemma}$ allows to merge it up to to area terms with other $\Op_L^l(\phi A_2)$ or $\Op_L^l(B)$, the result being a left operator of a symbol that is compactly supported in both variables. In this way, we start with the rightmost factor of $\Op_L^l(\phi A_2)$ and move it to the right thereby uniting it with all $\Op_L^l(B)$ to the right of it in a single $\Op_L^l$. Subsequently moving this operator to the very left, we obtain
\begin{equation}\label{Expansion-Commuted}
	Z_{\pi(1)}^{(k)}\cdots Z_{\pi(p)}^{(k)} \sim \mathbf{1}_\Lambda 
	\Op_L^l\Big(\phi^{p-k}C_{\pi(1)}^{(k)}\cdots C_{\pi(p)}^{(k)}\Big)
	\Big( \mathbf{1}_\Lambda \Op_L(\mathbf{1}_\Gamma)\mathbf{1}_\Lambda\Big)^k
\end{equation}
for $k\in\{0,\ldots,p-1\}$. The factor $\mathbf{1}_\Lambda $ on the left is only relevant if $k=0$, otherwise it can be absorbed in $\big( \mathbf{1}_\Lambda \Op_L(\mathbf{1}_\Gamma)\mathbf{1}_\Lambda\big)^k$ up to area terms.

Next, we claim that
\begin{align}\label{Trace-Evaluation}
	\tr_{L^2(\R^d)\otimes\C^n} & \big[\big(D_L(A_1,A_2)\big)^p\big] \notag\\
	&= L^d\mathfrak{W}_0\Big(\tr_{\C^n}[A_2^p];\Lambda,\R^d\Big) \nonumber\\
	&\quad + L^d \ \sum_{k=1}^p\frac{1}{k!(p-k)!}\sum_{\pi\in \mathcal S_p} \, \mathfrak{W}_0\Big(\tr_{\C^n}\big[C_{\pi(1)}^{(k)}\cdots C_{\pi(p)}^{(k)}\big];\Lambda,\Gamma\Big)
		\nonumber\\
	&\quad + L^{d-1}\log L \; \sum_{k=1}^p\frac{1}{k!(p-k)!}\sum_{\pi\in \mathcal S_p} \,
			\mathfrak{W}_1\Big(\tr_{\C^n}\big[C_{\pi(1)}^{(k)}\cdots C_{\pi(p)}^{(k)}\big]\,\mathfrak{A}(\id^k;1);\partial\Lambda,
				\partial\Gamma\Big) \nonumber \\ 
	&\quad + o(L^{d-1}\log L),
\end{align}
as $L\to\infty$. The first line on the right-hand side arises from the direct computation of the term with $k=0$ in \eqref{Expansion-D_L} and \eqref{Expansion-Commuted}. The terms with $k\in\{1,\ldots,p-1\}$ in  \eqref{Expansion-D_L} and \eqref{Expansion-Commuted} give rise to the corresponding terms in the sums of the second and third line of \eqref{Trace-Evaluation} by applying Theorem~\ref{Asymptotics-Polynomial} with $B$ there given by $\phi^{p-k}C_{\pi(1)}^{(k)}\cdots C_{\pi(p)}^{(k)}$ and $A$ there given by $\mathbb{1}_{n}$. The terms with $k=p$ in the second and third line of \eqref{Trace-Evaluation} are obtained by applying Theorem \ref{Asymptotics-Polynomial} directly to the term with $k=p$ in \eqref{Expansion-D_L} by choosing $B$ in Theorem \ref{Asymptotics-Polynomial} as $\mathbb{1}_{n}$ and $A$ in Theorem \ref{Asymptotics-Polynomial} as $B$ from this proof.

In order to conclude the proof it just remains to rewrite the terms on the right-hand side of \eqref{Trace-Evaluation}.
Adding the first two terms, we get
\begin{align}
	\label{3.66}
	L^d \Bigg( & \mathfrak{W}_0  \big(\tr_{\C^n}[A_2^p];\Lambda,\R^d\big)
		+ \sum_{k=1}^p\frac{1}{k!(p-k)!}\sum_{\pi\in \mathcal S_p} \, \mathfrak{W}_0\Big(\tr_{\C^n}\big[C_{\pi(1)}^{(k)}\cdots C_{\pi(p)}^{(k)}\big];\Lambda,\Gamma\Big)\Bigg)
			\nonumber\\
	&= \left(\frac{L}{2\pi}\right)^d \int_\Lambda\bigg(\int_{\R^d} \tr_{\C^n}[A_2^p(x,\xi)] \,\d\xi \notag\\
	&\hspace{3cm}	+ \int_\Gamma \sum_{k=1}^p\frac{1}{k!(p-k)!}\sum_{\pi\in \mathcal S_p} 
		\tr_{\C^n}\big[C_{\pi(1)}^{(k)}\cdots C_{\pi(p)}^{(k)}(x,\xi)\big] \, \d\xi\bigg)\d x
			\nonumber\\
	&= \left(\frac{L}{2\pi}\right)^d \int_\Lambda \int_{\R^d} \tr_{\C^n}
			\big[(A_2+B\mathbf{1}_\Gamma)^p(x,\xi)\big]\,\d\xi\, \d x\nonumber \\ 
	&= L^d \Big( \mathfrak{W}_0 \big(\tr_{\C^n}[A_1^p];\Lambda,\Gamma\big)
			+\mathfrak{W}_0 \big(\tr_{\C^n}[A_2^p];\Lambda,\Gamma^c\big)\Big),
\end{align}
where we used the expansion \eqref{Expansion-Matrices} with $X=A_2$ and $Y=B\mathbf{1}_\Gamma$ in the second step.
For the third term in \eqref{Trace-Evaluation} we need to evaluate
\begin{align}
\label{3.67}
\sum_{k=1}^p\frac{1}{k!(p-k)!}\sum_{\pi\in \mathcal S_p} \,\tr_{\C^n}\big[C_{\pi(1)}^{(k)}\cdots C_{\pi(p)}^{(k)}\big]\,\mathfrak{A}(\id^k;1).
\end{align}
The definition of $\mathfrak{A}$ and linearity of the trace imply that \eqref{3.67} is equal to
\begin{align}
	\frac{1}{(2\pi)^2} \int_0^1 & \frac{\sum_{k=1}^p\frac{1}{k!(p-k)!}\sum_{\pi\in \mathcal S_p}
			\tr_{\C^n}\big[C_{\pi(1)}^{(k)}\cdots C_{\pi(p)}^{(k)}\big] (t^{k}-t)}{t(1-t)}\,\d t \nonumber\\ 
	&= \frac{1}{(2\pi)^2} \int_0^1 \frac{\tr_{\C^n}\big[\sum_{k=0}^p\frac{1}{k!(p-k)!}\sum_{\pi\in \mathcal S_p} 
			C_{\pi(1)}^{(k)}\cdots C_{\pi(p)}^{(k)}(t^k-t)-A_2^{p}(1-t)\big]}{t(1-t)}\,\d t \nonumber \\ 
	&= \frac{1}{(2\pi)^2} \int_0^1 \frac{\tr_{\C^n}\big[(A_2+Bt)^p
				-(A_2+B)^pt - A_2^{p}(1-t)\big]}{t(1-t)}\,\d t,
\end{align}
where we used the expansion \eqref{Expansion-Matrices} twice in the last step.
Inserting $B=A_1-A_2$, we arrive at the expression
\begin{align}
	\frac{1}{(2\pi)^2} \int_0^1 \frac{\tr_{\C^n}\big[\big(A_1t+A_2(1-t)\big)^p
			-A_1^pt-A_2^p(1-t)\big]}{t(1-t)}\,\d t
	= \mathfrak{U}(\id^p;A_1,A_2)
\end{align}
for \eqref{3.67}.
Together with the linearity of $\mathfrak{W}_{1}$ and \eqref{3.66} we get the asymptotic formula
\begin{align}
	\tr_{L^2(\R^d)\otimes\C^n}\big[\big(D_L(A_1,A_2)\big)^p\big] 
	&= L^d\mathfrak{W}_0\big(\tr_{\C^n}[A_1^p];\Lambda,\Gamma\big)
			+L^d\mathfrak{W}_0\big(\tr_{\C^n}[A_2^p];\Lambda,\Gamma^c\big)\nonumber \\
	&\quad + L^{d-1}\log L \; \mathfrak{W}_1\big(\mathfrak{U}(\id^p;A_1,A_2);\partial\Lambda,
			\partial\Gamma\big)\nonumber \\ 
	&\quad +o(L^{d-1}\log L),
\end{align}
	as $L\rightarrow\infty$.
\end{proof}

It is straightforward to also get an asymptotic formula for the operator $G_L(A_1,A_2)$.

\begin{cor}
	\label{Asymptotics-Gamma-Gamma-Complement-G_L}
	Let $p\in\N$, let $A_1,A_2,\Lambda$ and $\Gamma$ be as in Theorem \ref{Asymptotics-Gamma-Gamma-Complement}. 
	Then we get the following asymptotic formula
	\begin{align}
		\tr_{L^2(\R^d)\otimes\C^n}\left[\big(G_L(A_1,A_2;\Lambda,\Gamma)\big)^p\right]
		=& \ L^d \mathfrak{W}_0 \big(\tr_{\C^n}[ (\Real A_1)^p];\Lambda,\Gamma \big) \notag\\
			&+ L^{d}\mathfrak{W}_0 \big(\tr_{\C^n}[(\Real A_2)^p];\Lambda,\Gamma^c \big)  \nonumber\\
			&+L^{d-1}\log L \ \mathfrak{W}_1 \big(\mathfrak{U}(\id^p; \Real A_1, \Real A_2); 
				\partial\Lambda,	\partial\Gamma\big)\nonumber\\ 
			&+o(L^{d-1}\log L),
	\end{align}
	as $L\rightarrow\infty$.
\end{cor}

\begin{proof}
By the definition of $G_L$ we have
\begin{align}
	G_L(A_1,A_2)-D_L(\Real A_1,\Real A_2) &= S_L(A_1;\Lambda,\Gamma)-T_L(\Real A_1;\Lambda,\Gamma)
		\nonumber \\ 
	&\quad + S_L(A_2;\Lambda,\Gamma^c) - T_L(\Real A_2;\Lambda,\Gamma^c).
\end{align}
As in (\ref{Gamma-Reduction}) and (\ref{Gamma-Complement-Reduction}), we can  assume that our symbols are compactly supported in both variables.
Using that $\big(\Op_L^l(\Real A_1)\big)^*=\Op_L^r(\Real A_1)$ and applying Lemma \ref{Exchange-Lemma}, 
we get
\begin{equation}
	\big\|S_L(A_1;\Lambda,\Gamma) - T_L(\Real A_1;\Lambda,\Gamma) \big\|_1 
	\leq  \frac{1}{2} \, \big\|\Op_L^r(\Real A_1)-\Op_L^l(\Real A_1) \big\|_1 
	\leq CL^{d-1}.
\end{equation}
Also applying this procedure with $A_2$ and $\Gamma^c$, we get
\begin{align}\label{Similarity-Real-Part}
G_L(A_1,A_2)\sim D_L(\Real A_1,\Real A_2).
\end{align}
Iterative use of (\ref{Similarity-Real-Part}) together with H\"older's inequality and the fact that $G_L(A_1,A_2)$ is uniformly bounded in $L$, yields
\begin{align}
\big(G_L(A_1,A_2)\big)^p\sim \big(D_L(\Real A_1,\Real A_2)\big)^p.
\end{align}
It just remains to apply Theorem \ref{Asymptotics-Gamma-Gamma-Complement} to get the desired asymptotic formula.
\end{proof}

%%%%%%%%%%%%%%%%%%%%%%%%%%%%%%%%%%%%%%%%%%%%%%%%%%%%%%%%%%%%%%%%%%
%%%%%%%%%%%%%%%%%%%%%%%%%%%%%%%%%%%%%%%%%%%%%%%%%%%%%%%%%%%%%%%%%%

\section{Closing the asymptotics}
\label{sec:closing-asymp}

We now want to establish the asymptotics from the theorem above for more general functions than polynomials. In the case of a non-self-adjoint operator, we can only extend this to analytic functions. For self-adjoint operators more general functions are accessible. \\
We first need to generalise some additional trace-class estimates from \cite{sobolevschatten} to matrix-valued symbols and unbounded domains.
\begin{lem}\label{Trace class estimates}
Let $B\in C^{\infty}_{b}(\R^{d} \times \R^{d}, \C^{n\times n})$ be a matrix-valued symbol with compact support in both variables. Let $\Lambda$ and $\Gamma$ be admissible domains. Then, for every $L\geq 2$, $t\in\{l,r\}$ and $q\in \,]0,1]$ we have
\begin{align}\label{Estimate-Log-Q}
	\big\|\mathbf{1}_\Lambda \Op_L(\mathbf{1}_\Gamma) \Op_L^t(B)\Op_L(\mathbf{1}_\Gamma)
		\mathbf{1}_{\Lambda^c}\big\|_q^q 
	\leq C L^{d-1}\log L,
\end{align}
and
\begin{align}\label{Estimate-Area-Q}
\big\|\Op_L(\mathbf{1}_\Gamma) \Op_L^t(B)\Op_L(\mathbf{1}_{\Gamma^c})\big\|_q^q \leq C L^{d-1}.
\end{align}
The constants $C>0$ depend on the symbol $B$ and the number $q$ but are independent of $\Lambda,\Gamma$ and $L$.
\end{lem}
\begin{proof}
We start by proving (\ref{Estimate-Log-Q}). 
By Lemma \ref{Commutation-Lemma} we get 
\begin{align}\label{Lemma-Lambda-Gamma-Matrix-Commutation}
\mathbf{1}_\Lambda \Op_L(\mathbf{1}_\Gamma) \Op_L^t(B)\Op_L(\mathbf{1}_\Gamma)\mathbf{1}_{\Lambda^c}\sim_q \mathbf{1}_\Lambda \Op_L(\mathbf{1}_\Gamma) \Op_L^t(B)\mathbf{1}_{\Lambda^c}.
\end{align}
By applying (\ref{q-factorization}) we get
\begin{align}
	\label{Lemma-Lambda-Gamma-Matrix-Reduction}
	\big\| \mathbf{1}_\Lambda \Op_L(\mathbf{1}_\Gamma) \Op_L^t(B)\mathbf{1}_{\Lambda^c} \big\|_q^q
	&= \bigg\| \sum_{\nu,\mu=1}^{n} 1_{\Lambda} \Op_L(1_{\Gamma}) \Op_L^t\big((B)_{\nu\mu}\big) 
			1_{\Lambda^c}\otimes E_{\nu\mu} \bigg\|_q^q \nonumber \\ 
	&\leq \sum_{\nu,\mu=1}^{n} \big\| 1_{\Lambda}\Op_L(1_{\Gamma}) \Op_L^t\big((B)_{\nu\mu}\big) 
		1_{\Lambda^c} \big\|_q^q.
\end{align}
Therefore, it suffices to show 
\begin{align}\label{Estimate-Log-Q-Scalar}
\|1_\Lambda \Op_L(1_\Gamma) \Op_L^t(b) 1_{\Lambda^c}\|_q^q\leq C L^{d-1}\log L,
\end{align}
for an arbitrary scalar-valued symbol $b\in C^{\infty}_{b}(\R^{d} \times \R^{d})$ with compact support in both variables.
It follows that there is some $R>0$ such that the support of $b$ is contained in $B_R(0)\times B_R(0)$.

As in the proof of Lemma \ref{Commutation-Lemma} define $\wtilde\Lambda:=\{x\in\R^d : \dist(x,\Lambda)<R\}$ and cover $\overline{\wtilde\Lambda\cap B_R(0)}$ with balls $B_\rho(x_j)$ of 
radius $\rho$ and centres $x_j$ such that $\Lambda\cap B_{4\rho}(x_j)=\Lambda_j\cap B_{4\rho}(x_j)$, where $\Lambda_j$ is a basic domain for every $j \in \mathcal{J} \subset\N$, a finite index set. Cover $\overline{\Gamma\cap B_R(0)}$ in the same manner with balls $B_{\rho_2}(\xi_k)$ of radius $\rho_2$ and centres $\xi_k$ for $k \in \mathcal{K} \subset\N$, a finite index set. Let $\{\phi_j\}_{j\in\mathcal{J}}$, $\{\psi_k\}_{k\in\mathcal{K}}$ be smooth and finite partitions of unity subordinate to these coverings, i.e.\ we have $\supp \phi_j \subset B_\rho(x_j)$, $\supp \psi_k \subset B_{\rho_2}(\xi_k)$, as well as 
\begin{align}
 \Phi\big{|}_{\overline{\wtilde\Lambda\cap B_R(0)}}=1 \ \ \text{ and } \ \  
 \Psi\big{|}_{\overline{\Gamma\cap B_R(0)}}=1,
\end{align}
where $\Phi := \sum_{j\in\mathcal{J}} \phi_j$ and $\Psi := \sum_{k\in\mathcal{K}} \psi_k$.
First we treat the more difficult case $t=r$.
We write
	\begin{align}
		\label{Lemma-Lambda-Gamma-Ref-1}
		\big\| 1_{\Lambda} \Op_L(1_{\Gamma})\Op_L^r(b) 1_{\Lambda^c} \big\|_q^q 
		&=  \big\| 1_{\Lambda} \Op_L(1_{\Gamma}) \Op_L^r\big( b(\Phi +1 - \Phi)\Psi\big) 
				1_{\Lambda^c} \big\|_q^q \nonumber \\ 
		&\leq \big\| 1_{\Lambda} \Op_L(1_{\Gamma}) \Op_L^r(b\Phi\Psi) 1_{\Lambda^c} \big\|_q^q &\nonumber \\ 
		&\quad + \big\| 1_{\Lambda} \Op_L(1_{\Gamma}) \Op_L^r\big(b\big(1 - \Phi)\big) \big\|_q^q.
\end{align}
As the symbol $b(1 - \Phi)$ is still smooth and compactly supported in both variables, 
we can apply Lemma \ref{Commutation-Lemma} to get 
\begin{align}
	1_{\Lambda} \Op_L(1_{\Gamma})\Op_L^r\big(b(1 - \Phi)\big) 
	\sim_q 1_{\Lambda} \Op_L^r\big(b(1 - \Phi)\big) \Op_L(1_{\Gamma}).
\end{align}
Note that as in the proof of Lemma \ref{Commutation-Lemma} the support of $b(1 - \sum_j\phi_j)$ in the first variable is contained in $\wtilde\Lambda^c$ and therefore is of distance at least $R$ to the support of the function $1_\Lambda$. Therefore, we can apply \cite[Thm.\ 3.2]{sobolevschatten} to get
\begin{equation}
	\label{Lemma-Lambda-Gamma-Ref-2}
	\big\| 1_{\Lambda} \Op_L^r\big(b(1 - \Phi)\big) \Op_L(1_{\Gamma}) \big\|_q^q \leq
	\big\| 1_{\Lambda} \Op_L^r (b ) (1 - \Phi) \big\|_q^q 
	\leq C.
\end{equation}
It remains to estimate 
\begin{align}\label{Lemma-Lambda-Gamma-Ref-3}
	\big\| 1_{\Lambda} \Op_L(1_{\Gamma}) \Op_L^r(b\Phi\Psi) 1_{\Lambda^c} \big\|_q^q 
	\le  \sum_{\substack{j\in\mathcal{J} \\ k\in\mathcal{K}}} \big\| 1_{\Lambda} \Op_L(1_{\Gamma_k}) \Op_L^r(b\phi_j \psi_k) 
		1_{\Lambda_j^c} \big\|_q^q.
\end{align}
As the sum is finite, it suffices to evaluate the individual terms. Using Lemma 
\ref{Commutation-Lemma} again, we get 
\begin{align}\label{Lemma-Lambda-Gamma-Ref-4}
1_{\Lambda} \Op_L(1_{\Gamma_k}) \Op_L^r(b\phi_j \psi_k) 1_{\Lambda_j^c} \sim_q 1_{\Lambda}  \Op_L^r(b\phi_j \psi_k) \Op_L(1_{\Gamma_k}) 1_{\Lambda_j^c}.
\end{align}
Let $h_j\in C^\infty(\R^d)$ be a smooth function such that $\|h_j\|_\infty\leq 1$, $\supp(h_j)\subset B_{4\rho}(x_j)$ and $h_j|_{B_{2\rho}(x_j)}=1$. Then 
\begin{align}\label{Lemma-Lambda-Gamma-Ref-5}
	\big\| 1_{\Lambda}  \Op_L^r(b\phi_j \psi_k) \Op_L(1_{\Gamma_k}) 1_{\Lambda_j^c}\big\|_q^q 
	&= \big\| 1_{\Lambda} (h_j+1-h_j) \Op_L^r(b\phi_j \psi_k) \Op_L(1_{\Gamma_k}) 1_{\Lambda_j^c}
	 		\big\|_q^q \nonumber \\ 
	&\leq \big\| 1_{\Lambda_j} \Op_L^r(b\phi_j \psi_k) \Op_L(1_{\Gamma_k}) 1_{\Lambda_j^c}\big\|_q^q 
			\nonumber \\ 
	&\quad + \| (1-h_j) \Op_L^r(b \psi_k) \phi_j \|_q^q.
\end{align}
To get a bound for the last term we use \cite[Thm.\ 3.2]{sobolevschatten} again, as the supports of $(1-h_j)$ and $\phi_j$ have distance at least $\rho$. This yields
\begin{align}\label{Lemma-Lambda-Gamma-Ref-6}
\| (1-h_j) \Op_L^r(b \psi_k) \phi_j \|_q^q \leq C.
\end{align}
In the remaining term all occurrences of admissible domains are replaced by basic domains and applying \cite[Thm.\ 4.6]{sobolevschatten} gives 
\begin{align}\label{Lemma-Lambda-Gamma-Ref-7}
	\big\| 1_{\Lambda_j}\Op_L^r(b\phi_j \psi_k)\Op_L(1_{\Gamma_k})1_{{\Lambda_j}^c} \big\|_q^q
	\leq C L^{d-1}\log L.
\end{align}
Combining (\ref{Lemma-Lambda-Gamma-Ref-1}) -- (\ref{Lemma-Lambda-Gamma-Ref-7}) yields 
(\ref{Estimate-Log-Q-Scalar}) and concludes the proof of (\ref{Estimate-Log-Q}) for $t=r$. The simpler case $t=l$ starts from rewriting the operator in (\ref{Estimate-Log-Q-Scalar}) as 
\begin{align}
	1_\Lambda \Op_L(1_\Gamma) \Op_L^l(b) 1_{\Lambda^c}
	\sim_q 1_\Lambda  \Op_L^l(b) \Op_L(1_\Gamma) 1_{\Lambda^c}
	= \sum_{\substack{j\in\mathcal{J} \\ k\in\mathcal{K}}} 1_{\Lambda_j} \phi_j \Op_L^l(b) \psi_k \Op_L(1_{\Gamma_k})1_{\Lambda^c}.
\end{align}
The further steps mirror the ones of the case $t=r$, starting from (\ref{Lemma-Lambda-Gamma-Ref-4}).

Finally, Inequality (\ref{Estimate-Area-Q}) is a direct consequence of Lemma $\ref{Commutation-Lemma}.$
\end{proof}

%%%%%%%%%%%%%%%%%%%%%%%%%%%%%%%%%%%%%%%%%%%%%%%%%%%%%%%%%%%%%

\subsection{Analytic functions}
\label{subsec:analytic}

We consider functions $g$ with $g(0)=0$ which are analytic in a disc $B_R(0) \subset\C$ about the origin with sufficiently large radius $R>0$, i.e. there exists $\omega_m\in\C,m\in\N,$ such that
\begin{align}
	\label{g-fct-def}
	g(z)=\sum_{m \in\N} \omega_m z^m, 
\end{align}
for all $z\in B_R(0).$ 
The trace of $D_L(g(A_1),g(A_2))$ can be computed explicitly and coincides (up to area terms) with the expected volume term.
\begin{lem}\label{Volume-Analytic}
Let $A_1,A_2\in C^{\infty}_{b}(\R^{d} \times \R^{d}, \C^{n\times n})$ be matrix-valued symbols with $A_2$ being compactly supported in the second variable. Further, let $\Lambda$ and $\Gamma$ be bounded admissible domains. Then for any function $g$ analytic in a disc $B_R(0) \subset\C$ of radius $R>\|\Op_L^l(A_1)\|+\|\Op_L^l(A_2)\|$ with $g(0)=0$, we have
\begin{align}
	\tr_{L^2(\R^d)\otimes\C^n}\big[D_L\big(g(A_1),g(A_2)\big)\big]
	&= L^d \Big[ \mathfrak{W}_0\big(\tr_{\C^n}[g(A_1)];\Lambda,\Gamma\big)
			+\mathfrak{W}_0\big(\tr_{\C^n}[g(A_2)];\Lambda,\Gamma^c\big) \Big]\nonumber\\
	&\quad + O(L^{d-1}),
\end{align}
as $L\rightarrow\infty.$
\end{lem}
\begin{proof}
Let $\phi,\psi\in C_c^\infty(\R^d)$ be as in (\ref{Definition-Phi-Psi}), i.e. $\phi|_\Lambda = 1$  and   $\psi|_\Gamma= 1$. First note that the function $g$ being analytic guarantees that $g(A_j)\in C^{\infty}_{b}(\R^{d} \times \R^{d}, \C^{n\times n})$ is a well-defined symbol for $j\in\{1,2\}$. Further the property $g(0)=0$ ensures that these symbols are again compactly supported whenever the corresponding symbol $A_j$ is compactly supported.
Therefore, the previous results (\ref{Gamma-Reduction}) and (\ref{Gamma-Complement-Reduction}) on 
$C^{\infty}_{b}$-symbols apply and we get
\begin{align}\label{Volume-Analytic-Commutation}
	D_L\big(g(A_1),g(A_2)\big) 
	\sim \mathbf{1}_\Lambda Op_L^l\big(g(A_1)\phi\psi\big) \Op_L(\mathbf{1}_\Gamma)\mathbf{1}_\Lambda
		+\mathbf{1}_\Lambda \Op_L^l\big(g(A_2)\phi\big)\Op_L(\mathbf{1}_{\Gamma^c})\mathbf{1}_\Lambda.
\end{align} 
The trace of the right-hand side of (\ref{Volume-Analytic-Commutation}) can be computed explicitly by integrating its kernel along the diagonal which yields
\begin{multline}
	\tr_{L^2(\R^d)\otimes\C^n}\Big[\mathbf{1}_\Lambda \Op_L^l\big(g(A_1)\phi\psi\big)
			\Op_L(\mathbf{1}_\Gamma)\mathbf{1}_\Lambda
		+ \mathbf{1}_\Lambda \Op_L^l\big(g(A_2)\phi\big) 
			\Op_L(\mathbf{1}_{\Gamma^c})\mathbf{1}_\Lambda\Big] \\ 
	= L^d \Big[ \mathfrak{W}_0\big(\tr_{\C^n}[g(A_1)];\Lambda,\Gamma\big)
		+\mathfrak{W}_0\big(\tr_{\C^n}[g(A_2)];\Lambda,\Gamma^c\big) \Big].
\end{multline}
\end{proof}
The crucial remaining step is to estimate the following trace norm
\begin{align}\label{goal-analytic}
	\big\|g\big(D_L(A_1,A_2)\big)-D_L\big(g(A_1),g(A_2)\big)\big\|_1
\end{align}
and see that it only gives an enhanced area law with coefficient depending on the function $g$. We divide this task into several steps and begin by reducing the question to symbols which are compactly supported in both variables. For some of the steps we will need a large radius of convergence. In order to define it, we will first define the radius
\begin{align}\label{Definition_T_A}
t_A:= \sup_{L \ge 1}\|\Op_L^l(A)\|+\mathbf{N}^{(d+1,d+2)}(A),
\end{align}
for a single matrix-valued symbol $A\in C^{\infty}_{b}(\R^{d} \times \R^{d}, \C^{n\times n})$, where $\mathbf{N}^{(d+1,d+2)}(A)$ is defined in (\ref{Definition-N-Matrix}).

\begin{lem}
	\label{Compact-Support-Analytic}
	Let $A_1,A_2\in C^{\infty}_{b}(\R^{d} \times \R^{d}, \C^{n\times n})$ be matrix-valued symbols with $A_2$ being compactly supported 
	in the second variable. Further, let $\Lambda$ and $\Gamma$ be bounded admissible domains. 
	Then, there exist symbols $B_1,B_2\in C^{\infty}_{b}(\R^{d} \times \R^{d}, \C^{n\times n})$, compactly supported in both variables and with 
	\begin{equation}
		(A_1-B_1)\big|_{\Lambda\times\Gamma}=0,\qquad 
		(A_2-B_2)\big|_{\Lambda\times\R^d}=0
	\end{equation}	
	such that for any function $g$ analytic in a disc $B_R(0) \subset\C$ of radius 
	$R>t_0:=t_{A_1}+t_{A_2}+t_{B_1}+t_{B_2}$ with $g(0)=0$ we have
	\begin{align}\label{Compactification-Analytic-g-inside}
		D_L\big(g(A_1),g(A_2)\big) \sim D_L\big(g(B_1),g(B_2)\big)
	\end{align}
	and
	\begin{align}\label{Compactification-Analytic-g}
		g\big(D_L(A_1,A_2)\big) \sim g\big(D_L(B_1,B_2)\big).
	\end{align}
\end{lem}

\begin{proof} Let $L \ge 1$. We find the compactly supported symbols by introducing smooth cutoff functions in the same way as in Lemma \ref{Volume-Analytic}: Let $\phi,\psi\in C_c^\infty(\R^d)$ be as in (\ref{Definition-Phi-Psi}) and define $B_1:=A_1 \phi\psi$ and $B_2:=A_2 \phi$.  As in (\ref{Volume-Analytic-Commutation}) we get 
\begin{align}
	D_L\big(g(A_1),g(A_2)\big) \sim \mathbf{1}_\Lambda \Op_L^l\big(g(A_1)\phi\psi\big)
		\Op_L(\mathbf{1}_\Gamma)\mathbf{1}_\Lambda
	+ \mathbf{1}_\Lambda \Op_L^l\big(g(A_2)\phi\big) \Op_L(\mathbf{1}_{\Gamma^c})\mathbf{1}_\Lambda.
\end{align} 
Since $\phi|_\Lambda=1$ and $\psi|_\Gamma=1$ we have
\begin{align}
	\mathbf{1}_\Lambda \Op_L^l\big(g(A_1)\phi\psi\big) & \Op_L(\mathbf{1}_\Gamma) 
			\mathbf{1}_\Lambda 
		+ \mathbf{1}_\Lambda \Op_L^l\big(g(A_2)\phi\big) \Op_L(\mathbf{1}_{\Gamma^c})
			\mathbf{1}_\Lambda \nonumber \\
	&= \mathbf{1}_\Lambda \Op_L^l\big(g(B_1)\big) \Op_L(\mathbf{1}_\Gamma)\mathbf{1}_\Lambda
		+ \mathbf{1}_\Lambda \Op_L^l\big(g(B_2)\big) \Op_L(\mathbf{1}_{\Gamma^c})\mathbf{1}_\Lambda 
			\nonumber \\ 
	&\sim D_L\big(g(B_1),g(B_2)\big),
\end{align} 
where we used Lemma \ref{Commutation-Lemma} in the last step. This proves (\ref{Compactification-Analytic-g-inside}).

In order to prove (\ref{Compactification-Analytic-g}), we first note
\begin{align}\label{Compactification-Analytic}
D_L(A_1,A_2)\sim D_L(B_1,B_2),
\end{align} 
by setting $g=\id$ in (\ref{Compactification-Analytic-g-inside}). We now need to calculate the trace norm $\|(D_L(A_1,A_2))^m-(D_L(B_1,B_2))^m\|_1$ for arbitrary $m\in\N$. We note that the case $m=0$ need not be considered here due to $g(0)=0$. Repeatedly applying $(\ref{Compactification-Analytic})$ together with H\"older's inequality and the fact that $D_L(A_1,A_2)$ and $D_L(B_1,B_2)$ are bounded uniformly in $L$ by $t_0$, we get
\begin{align}
	\big\|\big(D_L(A_1,A_2)\big)^m - \big(D_L(B_1,B_2)\big)^m \big\|_1 \leq L^{d-1} Cm t_0^{m-1},
\end{align}
where the constant $C$ depends neither on $m$ nor $L$. 
Therefore, using the notation of \eqref{g-fct-def}, we obtain
\begin{align}
	\big\| g\big(D_L(A_1,A_2)\big) - g\big(D_L(B_1,B_2)\big) \big\|_1 
	&= \bigg\|\sum_{m=1}^\infty \omega_m \Big\{ \big(D_L(A_1,A_2)\big)^m - \big(D_L(B_1,B_2)\big)^m \Big\}\bigg\|_1 
		\nonumber \\ 
	&\leq L^{d-1} C\sum_{m=1}^\infty m|\omega_m| t_{0}^{m-1},
\end{align}
as $g$ is analytic in $B_R(0)$.
\end{proof}
The next Lemma \cite[Lemma 12.1]{sobolevlong} deals with the projections in the operator $D_L$.

\begin{lem}
	\label{Projections-Analytic}
	Let $X$ be a trace-class operator in a separable Hilbert space $\mathcal{H}$. 
	Let $P$ be an orthogonal projection in $\mathcal{H}$ and $g$ be as in \eqref{g-fct-def} a function 
	that is analytic in a disc $B_{R}(0) \subset\C$ of radius $R>\|X\|$. Then we have 
	\begin{align}
		\|g(PXP)-Pg(X)P\|_1\leq g^{|1|}(\|X\|)\|PX(\mathbb{1}_{\mathcal{H}}-P)\|_1,
	\end{align}
	where
	\begin{align}
		\label{Definition_G_1}
		g^{|1|}(z):=\sum_{m=2}^\infty (m-1)|\omega_m| z^{m-1}, \qquad z\in B_{R}(0).
	\end{align}	
\end{lem}

\noindent
The next result deals with interchanging $g$ and $\Op_L^l$. 
\begin{lem}\label{Operator-Analytic}
Let $A\in C^{\infty}_{b}(\R^{d} \times \R^{d}, \C^{n\times n})$ be a matrix-valued symbol with compact support in both variables and 
let $g$ be a function that is analytic in a disc of radius $R>t_A$ about the origin and such that $g(0)=0$. Then we have
\begin{align}
	g\big(\Op_L^l(A)\big) \sim \Op_L^l\big(g(A)\big).
\end{align}
\end{lem}

\begin{proof}
Let $L\ge 1$. It suffices to show 
\begin{align}
	\label{op-ana-to-prove}
	\big\|\big(\Op_L^l(A)\big)^m-\Op_L^l(A^m)\big\|_1 \leq D L^{d-1} (m-1)^{2(d+2)}t_A^{m-1}
\end{align}
for every $m\in\N$ with the finite constant $D := C \mathbf{N}^{(d+1,d+2)}(A)$, where the constant $C$ 
does not depend on $L$ and $m$ and is specified below \eqref{induction-2}. Indeed, the power-series expansion \eqref{g-fct-def} of $g$ then yields 
\begin{align}
	\big\|g\big(\Op_L^l(A)\big) - \Op_L^l\big(g(A)\big)\big\|_1 
	&= \bigg\|\sum_{m=1}^\infty \omega_m \Big\{\big(\Op_L^l(A)\big)^m-\Op_L^l(A^m) \Big\}\bigg\|_1 \nonumber \\ 
	&\leq D L^{d-1} \sum_{m=1}^\infty |\omega_m|(m-1)^{2(d+2)}t_A^{m-1},  
\end{align}
which gives a finite constant, as $g$ is analytic in a disc of sufficiently large radius. 
We prove \eqref{op-ana-to-prove} by induction on $m$. 

For $m=1$, there is nothing to prove in \eqref{op-ana-to-prove}.  
%follows from (\ref{Merge-Symbols-Precise}) with 
%$D = C \mathbf{N}^{(d+1,d+2)}(A)$, where $C$ is the constant in (\ref{Merge-Symbols-Precise}). 
Now, suppose \eqref{op-ana-to-prove} 
holds for $m \in\N$. In order to prove \eqref{op-ana-to-prove} for $m+1$ we estimate
\begin{align}
	\big\|\big(\Op_L^l(A)\big)^{m+1} - \Op_L^l(A^{m+1})\big\|_1
	&\leq \big\|\big(\Op_L^l(A)\big)^{m+1} - \Op_L^l(A^{m})\Op_L^l(A) \big\|_1 \nonumber \\
	&\quad +\big\|\Op_L^l(A^{m})\Op_L^l(A) - \Op_L^l(A^{m+1})\big\|_1.
\end{align}
For the first term, we use the induction hypothesis and H\"older's inequality
\begin{align}
	\label{induction-1}
	\big\|\big(\Op_L^l(A)\big)^{m+1} - \Op_L^l(A^{m}) \Op_L^l(A) \big\|_1 
	&\leq D L^{d-1}(m-1)^{2(d+2)}t_A^{m-1}\|\Op_L^l(A)\| \notag\\
	&\le D L^{d-1}m^{2(d+2)}t_A^{m-1}\|\Op_L^l(A)\|.
\end{align}
The second term is treated with Inequality (\ref{Merge-Symbols-Precise}) applied to the symbols $A^{m}$ and $A$
\begin{align}
	\label{induction-2}
	\big\|\Op_L^l(A^{m})\Op_L^l(A) - \Op_L^l(A^{m+1})\big\|_1 
	&\leq C L^{d-1} \mathbf{N}^{(d+1,d+2)}(A^m)\, \mathbf{N}^{(d+1,d+2)}(A)  \notag\\
	&= D L^{d-1} \mathbf{N}^{(d+1,d+2)}(A^m).
\end{align}
Since the arising constant $C$ from (\ref{Merge-Symbols-Precise}) depends only on the support of the symbols $A^{m}$ and $A$, it is independent of $m$ and $L$.
It follows from the definition \eqref{Definition-N-Matrix} of the symbol norm that 
\begin{equation}
	\label{norm-powers}
 	\mathbf{N}^{(d+1,d+2)}(A^m) \le m^{2(d+2)} \big(\mathbf{N}^{(d+1,d+2)}(A)\big)^m 
 	\le m^{2(d+2)} t_{A}^{m-1}  \mathbf{N}^{(d+1,d+2)}(A).
\end{equation}
Inserting \eqref{norm-powers} into \eqref{induction-2}, adding \eqref{induction-1} and using once again the definition
\eqref{Definition_T_A} of $t_{A}$, we conclude the proof of the induction step.
\end{proof}

Combining the previous results, we get the desired estimate for (\ref{goal-analytic}).

\begin{lem}\label{Combination-Analytic}
	Let $A_1,A_2\in C^{\infty}_{b}(\R^{d} \times \R^{d}, \C^{n\times n})$ be matrix-valued symbols with $A_2$ being compactly supported in the second variable. Further, let $\Lambda$ and $\Gamma$ be bounded admissible domains. For any function $g$ analytic in a disc of radius $R>t_0$ with $g(0)=0$, there exist constants $C_1,C_2>0$, with $C_1$ independent of $g$, $\Lambda$ and $\Gamma$, such that
	\begin{equation}
		\label{Combination-Analytic-Statement}
		\big\|g\big(D_L(A_1,A_2)\big) - D_L\big(g(A_1),g(A_2)\big)\big\|_1 
		\leq C_1g^{|1|}(t_0)  L^{d-1}\log L  + C_2  L^{d-1} ,
	\end{equation}
	for every $L\geq 2$. The radius $t_0$ is defined in Lemma \ref{Compact-Support-Analytic} 
	and \eqref{Definition_T_A}.
\end{lem}

\begin{proof}
Let $L\geq 2$. As Lemma $\ref{Compact-Support-Analytic}$ yields
\begin{align}
	g\big(D_L(A_1,A_2)\big) \sim g\big(D_L(B_1,B_2)\big)
\end{align} 
and
\begin{align}
	D_L\big(g(A_1),g(A_2)\big) \sim D_L\big(g(B_1),g(B_2)\big),
\end{align}
it suffices to show (\ref{Combination-Analytic-Statement}) with $A_1$ and $A_2$ replaced by the compactly supported symbols $B_1$ and $B_2$. We prove this starting with the operator $g(D_L(B_1,B_2))$ on the left-hand side. 

Lemma $\ref{Projections-Analytic}$ applied to $\Op_L(\mathbf{1}_\Gamma)\Op_L^l(B_1)\Op_L(\mathbf{1}_\Gamma)+\Op_L(\mathbf{1}_{\Gamma^c})\Op_L^l(B_2)\Op_L(\mathbf{1}_{\Gamma^c})$ and $\mathbf{1}_\Lambda$ combined with Lemma $\ref{Trace class estimates}$ yields 
\begin{align}
 	\Big\|& g\big(D_L(B_1,B_2)\big)-\mathbf{1}_\Lambda \  g\Big( 
			\Op_L(\mathbf{1}_\Gamma)\Op_L^l(B_1)\Op_L(\mathbf{1}_\Gamma) 
	+ \Op_L(\mathbf{1}_{\Gamma^c}) \Op_L^l(B_2)\Op_L(\mathbf{1}_{\Gamma^c})\Big) \mathbf{1}_\Lambda\Big\|_1 \nonumber  \\  
	&\leq g^{|1|}(t_0)\Big( \big\| \mathbf{1}_\Lambda \Op_L(\mathbf{1}_\Gamma) \Op_L^l(B_1)\Op_L(\mathbf{1}_\Gamma)\mathbf{1}_{\Lambda^c}\big\|_1  %\nonumber \\
%	& \hspace*{4em} 
+ \big\| 1_\Lambda \Op_L(\mathbf{1}_{\Gamma^c}) \Op_L^l(B_2)\Op_L(\mathbf{1}_{\Gamma^c})\mathbf{1}_{\Lambda^c}\big\|_1 \Big)\nonumber \\ 
	&\leq C_1 g^{|1|}(t_0) \ L^{d-1}\log L.
\end{align}
Here, $C_1$ is the constant from (\ref{Combination-Analytic-Statement}). For its independence of $g$, $\Lambda$ and $\Gamma$ we refer to Lemma \ref{Trace class estimates}.
The power series expansion of the function $g$ implies
\begin{multline}
g\Big(\Op_L(\mathbf{1}_\Gamma)\Op_L^l(B_1)\Op_L(\mathbf{1}_\Gamma)+\Op_L(\mathbf{1}_{\Gamma^c}) \Op_L^l(B_2)\Op_L(\mathbf{1}_{\Gamma^c})\Big)\\=g\big(\Op_L(\mathbf{1}_\Gamma)\Op_L^l(B_1)\Op_L(\mathbf{1}_\Gamma)\big)+g\big(\Op_L(\mathbf{1}_{\Gamma^c}) \Op_L^l(B_2)\Op_L(\mathbf{1}_{\Gamma^c})\big).
\end{multline} 
Applying Lemma $\ref{Projections-Analytic}$ to $\Op_L^l(B_1)$  and $\Op_L(\mathbf{1}_\Gamma)$, followed by Lemma $\ref{Trace class estimates}$, yields
\begin{multline}
	\big\|g\big(\Op_L(\mathbf{1}_\Gamma)\Op_L^l(B_1)\Op_L(\mathbf{1}_\Gamma)\big)
		-\Op_L(\mathbf{1}_\Gamma)g\big(\Op_L^l(B_1)\big)\Op_L(\mathbf{1}_\Gamma)\big\|_1 \\ 
	\leq g^{|1|}(t_0) \big\|\Op_L(\mathbf{1}_\Gamma)\Op_L^l(B_1) \Op_L(\mathbf{1}_{\Gamma^c})\big\|_1 
	\leq C L^{d-1}.
\end{multline}
Applying Lemma $\ref{Projections-Analytic}$ again to $\Op_L^l(B_2)$ and $\Op_L(\mathbf{1}_{\Gamma^c})$, we conclude
\begin{multline}
\mathbf{1}_\Lambda \  g\Big(\Op_L(\mathbf{1}_\Gamma)\Op_L^l(B_1)\Op_L(\mathbf{1}_\Gamma)+\Op_L(\mathbf{1}_{\Gamma^c}) \Op_L^l(B_2)\Op_L(\mathbf{1}_{\Gamma^c})\Big) \mathbf{1}_\Lambda  \\ \sim  \mathbf{1}_\Lambda \Big( \Op_L(\mathbf{1}_\Gamma)g\big(\Op_L^l(B_1)\big)\Op_L(\mathbf{1}_\Gamma)+ \Op_L(\mathbf{1}_{\Gamma^c})g\big(\Op_L^l(B_2)\big)\Op_L(\mathbf{1}_{\Gamma^c})\Big)\mathbf{1}_\Lambda.
\end{multline}

\noindent
Finally, Lemma $\ref{Operator-Analytic}$ gives
\begin{align}
	\mathbf{1}_\Lambda \Op_L(\mathbf{1}_\Gamma) g\big(\Op_L^l(B_1)\big) \Op_L(\mathbf{1}_\Gamma)
		\mathbf{1}_\Lambda
	\sim T_L\big(g(B_1);\Lambda,\Gamma\big),
\end{align}
as well as
\begin{align}
	\mathbf{1}_\Lambda \Op_L(\mathbf{1}_{\Gamma^c}) g\big(\Op_L^l(B_2)\big)
		\Op_L(\mathbf{1}_{\Gamma^c}) \mathbf{1}_\Lambda
	\sim T_L\big(g(B_2);\Lambda,\Gamma^c\big).
\end{align}
The claim follows, because 
$D_L\big(g(B_1),g(B_2);\Lambda,\Gamma\big) = T_L\big(g(B_1);\Lambda,\Gamma\big)
	+T_L\big(g(B_2);\Lambda,\Gamma^c\big)$.
\end{proof}

We also need the following estimate for the coefficient $\mathfrak{W}_1.$

\begin{lem}
	\label{Coefficient-Analytic}
	Let $A_1,A_2\in C^{\infty}_{b}(\R^{d} \times \R^{d}, \C^{n\times n})$ be matrix-valued symbols. Let the function $g$ be analytic in a 
	disc of radius $R>t_0$ with $g(0)=0$. Let $\partial\Lambda$ and $\partial\Gamma$ have finite $(d-1)$-dimensional surface measure induced by Lebesgue measure on $\R^d$. 
	Then
	\begin{align}
		\big|\mathfrak{W}_1\big(\mathfrak{U}(g;A_1,A_2);\partial\Lambda,\partial\Gamma\big)\big|
		\leq \frac{|\partial\Lambda| |\partial\Gamma|}{(2\pi)^{2}} \, t_{0} g^{|1|}(t_0).
	\end{align}
\end{lem}

\begin{proof}
It suffices to show 
\begin{align}
\sup_{(x,\xi)\in\partial\Lambda\times\partial\Gamma}\Big|\big(\mathfrak{U}(g;A_1,A_2)\big)(x,\xi)\Big|\leq  \frac{t_{0}}{(2\pi)^{2}} \, g^{|1|}(t_0).
\end{align}
To start with we choose the test function $g$ as a monomial of order $m\in\N$ and work towards a pointwise estimate for the function 
\begin{align}
	\label{Coefficient-Analytic-PW}
	(2\pi)^2\mathfrak{U}(\id^m;A_1,A_2) 
	&= \int_0^1 \frac{\tr_{\C^n}\big[\big(A_1t+A_2(1-t)\big)^m-A_1^mt-A_2^m(1-t)\big]}{t(1-t)}\,\d t 
			\nonumber \\ 
	&= \tr_{\C^n}\bigg[ \int_0^1\bigg( \frac{(A_1t)^m-A_1^mt}{t(1-t)} + 
			\frac{\big(A_2(1-t)\big)^m-A_2^m(1-t)}{t(1-t)} \nonumber \\ 
	&\hspace{2.4cm} + \sum_{k=1}^{m-1}\frac{1}{k!(m-k)!}\sum_{\pi\in \mathcal S_m}Z_{\pi(1)}^{(k)}\cdots Z_{\pi(m)}^{(k)}
			\,\frac{t^k(1-t)^{m-k}}{t(1-t)} \bigg)\,\d t \bigg],
\end{align}
where the $Z_{j}^{(k)}$ for $j\in \{1,\ldots,m\}$ and $k \in \{1,\ldots,m-1\}$ are defined as in the expansion \eqref{Expansion-Matrices} with $X=A_2$, $Y=A_1$ and $p=m$.
The integral $\int_0^1 \frac{t^k(1-t)^{m-k}}{t(1-t)}\, \d t$ is bounded above by $1$ for every 
$k\in\{1,\ldots, m-1\}$. For the first integrand on the right-hand side of \eqref{Coefficient-Analytic-PW}, we get 
$|\int_0^1 \frac{t^m-t}{t(1-t)}\,\d t| \leq m-1$. By a change of variables, we also get 
$|\int_0^1 \frac{(1-t)^m-(1-t)}{t(1-t)} \, \d t| \leq m-1$ for the second integrand. 
Introducing the norm 
\begin{align}\label{Definition-Norm-Infinity-1}
\|A\|_{\infty,1}:= \sup_{(x,\xi)\in\partial\Lambda\times\partial\Gamma}\tr_{\C^n}|A(x,\xi)|
\end{align}
for matrix-valued functions $A$ on $\partial\Lambda\times\partial\Gamma$, we estimate the right-hand side of (\ref{Coefficient-Analytic-PW}) from above by
\begin{multline}
 	(m-1)\|A_1\|_{\infty,1}^m + (m-1)\|A_2\|_{\infty,1}^m 
		+ \sum_{k=1}^{m-1} \mybinom{m}{k} \|A_1\|_{\infty,1}^k
			\|A_2\|_{\infty,1}^{m-k}\nonumber \\ 
	\leq (m-1)\big(\|A_1\|_{\infty,1}+\|A_2\|_{\infty,1}\big)^m.
\end{multline}
Therefore, we conclude from the power-series expansion \eqref{g-fct-def} of $g$ that
\begin{align}
	\sup_{(x,\xi)\in\partial\Lambda\times\partial\Gamma}\big|\mathfrak{U}(g;A_1,A_2)\big|
	&\leq \sup_{(x,\xi)\in\partial\Lambda\times\partial\Gamma} \sum_{m=2}^\infty
		 |\omega_m|\, \big|\mathfrak{U}(\id^m;A_1,A_2)\big| \nonumber \\ 
	&\leq \frac{1}{(2\pi)^2}\sum_{m=2}^\infty |\omega_m|(m-1)t_0^m = \frac{t_{0}}{(2\pi)^2} \,g^{|1|}(t_0),
\end{align}
where we used that $\|A_1\|_{\infty,1}+\|A_2\|_{\infty,1}\leq t_{A_1}+t_{A_2}\leq t_0$ by (\ref{Definition_T_A}) and the definition of $g^{|1|}$ in (\ref{Definition_G_1}).
\end{proof}
We are now ready to close the asymptotics.

\begin{thm}
	\label{Asymptotics-Analytic}
	Let $A_1,A_2\in C^{\infty}_{b}(\R^{d} \times \R^{d}, \C^{n\times n})$ be matrix-valued symbols with 
	$A_2$ being compactly supported in the 
	second variable. Let $\Lambda$ be a bounded piece-wise $C^1$-admissible domain and $\Gamma$ 
	be a bounded piece-wise $C^3$-admissible domain. Let the function $h$ be analytic in a disc of 
	radius $R>t_0$ with $h(0)=0$, then
	\begin{align}
		\tr_{L^2(\R^d)\otimes\C^n}\left[h\big(D_L(A_1,A_2;\Lambda,\Gamma)\big)\right]
		=& \ L^d \Big( \mathfrak{W}_0\big(\tr_{\C^n}[h(A_1)];\Lambda,\Gamma\big) 
			+ \mathfrak{W}_0\big(\tr_{\C^n}[h(A_2)];\Lambda,\Gamma^c\big)\Big)\nonumber\\
		&+ L^{d-1}\log L \ \mathfrak{W}_1\big(\mathfrak{U}(h;A_1,A_2);\partial\Lambda,
				\partial\Gamma\big)\nonumber\\
		&+ o(L^{d-1}\log L),
	\end{align}
	as $L\rightarrow\infty$. Here, $t_0$ was defined in Lemma \ref{Compact-Support-Analytic} and 
	\eqref{Definition_T_A}, the coefficients 
	$\mathfrak{W}_0$ and $\mathfrak{W}_1$ in \eqref{coeff-w0-def} and \eqref{coeff-w1-def},
	respectively, and the symbol $\mathfrak{U}$ in \eqref{frakU-def}.
\end{thm}

\begin{proof}
Let $p\in\N$ and approximate $h$ by a polynomial in the following way
\begin{align}
	h=g_p+h_p, \ \ g_p(z):=\sum_{m=1}^p \omega_m z^m,  \qquad 
	h_p(z):= h(z)- g_{p}(z) =\sum_{m=p+1}^\infty \omega_m z^m,
\end{align}
where $\omega_{m} := h^{(m)}(0)/m!$ for $m\in\N$.
Let $\epsilon>0$. Then one can choose $p\in\N$ such that $ h_p^{|1|}(t_0)<\frac{\epsilon}{2}$. Lemma $\ref{Combination-Analytic}$ applied to $h_p$ yields
\begin{align}
	\big\|h_p\big(D_L(A_1,A_2)\big) - D_L\big(h_p(A_1),h_p(A_2)\big)\big\|_1 
	\leq C_1 \frac{\epsilon}{2} L^{d-1}\log L  + C_2  L^{d-1}.
\end{align}
As
\begin{align}
	\tr_{L^2(\R^d)\otimes\C^n}\big[ h\big(D_L(A_1,A_2)\big) &- D_L\big(h(A_1),h(A_2)\big)\big]
		\nonumber \\ 
	&\leq  \tr_{L^2(\R^d)\otimes\C^n}\big[g_p\big(D_L(A_1,A_2)\big)
		- D_L\big(g_p(A_1),g_p(A_2)\big)\big] \nonumber \\
	&\quad + \big\|h_p\big(D_L(A_1,A_2)\big) - D_L\big(h_p(A_1),h_p(A_2)\big)\big\|_1,
\end{align}
and the asymptotic formula is already established for $g_p$, we get
\begin{align}
	\underset{L\rightarrow\infty}{\limsup}\frac{\tr_{L^2(\R^d)\otimes\C^n}
		\big[ h\big(D_L(A_1,A_2)\big) - D_L\big(h(A_1),h(A_2)\big)\big]}{L^{d-1}\log L}
	\leq \mathfrak{W}_1(g_p) + C_1\frac{\epsilon}{2},
\end{align}
where 
$\mathfrak{W}_1(g_p):=\mathfrak{W}_1\big(\mathfrak{U}(g_p;A_1,A_2);\partial\Lambda,\partial\Gamma\big)$.
For the coefficient $\mathfrak{W}_1(h_p)$, Lemma \ref{Coefficient-Analytic} yields
\begin{align}
|\mathfrak{W}_1(h_p)|\leq C h_p^{|1|}(t_0)<C\frac{\epsilon}{2}.
\end{align}
Hence
\begin{align}
\mathfrak{W}_1(g_p)\leq \mathfrak{W}_1(h)+|\mathfrak{W}_1(h_p)|<\mathfrak{W}_1(h)+C\frac{\epsilon}{2}.
\end{align}
Therefore, 
\begin{align}
	\underset{L\rightarrow\infty}{\limsup}\frac{\tr_{L^2(\R^d)\otimes\C^n}
		\big[h\big(D_L(A_1,A_2)\big) - D_L\big(h(A_1),h(A_2)\big)\big]}{L^{d-1}\log L}
	\leq  \mathfrak{W}_1(h)+C\epsilon.
\end{align}
The corresponding lower bound for the liminf is found in the same way. As $\epsilon$ is arbitrary, we get
\begin{align}
	\underset{L\rightarrow\infty}{\lim} \frac{\tr_{L^2(\R^d)\otimes\C^n}
		\big[h\big(D_L(A_1,A_2)\big) - D_L\big(h(A_1),h(A_2)\big)\big]}{L^{d-1}\log L}
	= \mathfrak{W}_1(h).
\end{align}
The trace of $D_L(h(A_1),h(A_2))$ gives rise to the volume terms by Lemma \ref{Volume-Analytic}, and we obtain the desired asymptotics
\begin{align}
	\tr_{L^2(\R^d)\otimes\C^n}\big[h\big(D_L(A_1,A_2;\Lambda,\Gamma)\big)\big] 
	&=  L^d \Big( \mathfrak{W}_0\big(\tr_{\C^n}[h(A_1)];\Lambda,\Gamma\big)
			+\mathfrak{W}_0\big(\tr_{\C^n}[h(A_2)];\Lambda,\Gamma^c\big) \Big) \nonumber\\
	&\quad + L^{d-1}\log L \ \mathfrak{W}_1\big(\mathfrak{U}(h;A_1,A_2);\partial\Lambda,\partial\Gamma\big)
			\nonumber\\ 
	&\quad + o(L^{d-1}\log L),
\end{align}
as $L\rightarrow\infty$.
\end{proof}

%%%%%%%%%%%%%%%%%%%%%%%%%%%%%%%%%%%%%%%%%%%%%%%%%%%%%%%

\subsection{Smooth functions}
\label{subsec:smooth}

We continue with arbitrarily often differentiable functions $g$ vanishing at zero. In order to apply $g$ to an operator we restrict ourselves to the self-adjoint operator $G_L$, see (\ref{Definition_GL}) for the definition.

The trace of $D_L\big(g(\Real A_1),g(\Real A_2)\big)$ can again be computed explicitly and coincides (up to area terms) with the expected volume term. This is stated as
\begin{lem}\label{Volume-Smooth}
Let $A_1,A_2\in C^{\infty}_{b}(\R^{d} \times \R^{d}, \C^{n\times n})$ be matrix-valued symbols with $A_2$ being compactly supported in the second variable. Further, let $\Lambda$ and $\Gamma$ be bounded admissible domains. Then for any function $g\in C^\infty(\R)$ with $g(0)=0$, we have
\begin{multline}
	\tr_{L^2(\R^d)\otimes\C^n}\Big[D_L\big(g(\Real A_1),g(\Real A_2)\big)\Big] \\ 
	= L^d \Big[ \mathfrak{W}_0\big(\tr_{\C^n}[g(\Real A_1)];\Lambda,\Gamma\big)
		+\mathfrak{W}_0\big(\tr_{\C^n}[g(\Real A_2)];\Lambda,\Gamma^c\big) \Big]+O(L^{d-1}),
\end{multline}
as $L\rightarrow\infty.$
\end{lem}
\begin{proof}
Let $\phi,\psi\in C_c^\infty(\R^d)$ be as in (\ref{Definition-Phi-Psi}), i.e. $\phi|_\Lambda = 1$  and   $\psi|_\Gamma= 1$. First note that the function $g$ being smooth guarantees that $g(\Real A_j)\in C^{\infty}_{b}(\R^{d} \times \R^{d}, \C^{n\times n})$ is a well-defined symbol for $j\in\{1,2\}$. Further, the property $g(0)=0$ ensures that these symbols are again compactly supported whenever the corresponding symbol $A_j$ is compactly supported.
Therefore, the previous results (\ref{Similarity-Real-Part}), (\ref{Gamma-Reduction}) and (\ref{Gamma-Complement-Reduction}) for $C^{\infty}_{b}$-symbols apply, and we get
\begin{align}
	D_L\big(g(\Real A_1),g(\Real A_2)\big) &\sim 
	\mathbf{1}_\Lambda \Op_L^l\big(g(\Real A_1)\phi\psi\big) \Op_L(\mathbf{1}_\Gamma)
			\mathbf{1}_\Lambda  \notag\\
	&\quad + \mathbf{1}_\Lambda \Op_L^l\big(g(\Real A_2)\phi\big)
		\Op_L(\mathbf{1}_{\Gamma^c})\mathbf{1}_\Lambda.
\end{align} 
The trace of the right-hand side can be computed explicitly by integrating its kernel along the diagonal which yields
\begin{multline}
	\tr_{L^2(\R^d)\otimes\C^n}\left[\mathbf{1}_\Lambda \Op_L^l\big(g(\Real A_1)\phi\psi\big)
			\Op_L(\mathbf{1}_\Gamma)\mathbf{1}_\Lambda 
		+\mathbf{1}_\Lambda \Op_L^l\big(g(\Real A_2)\phi\big) \Op_L(\mathbf{1}_{\Gamma^c})
			\mathbf{1}_\Lambda\right]\\ 
	= L^d \Big[ \mathfrak{W}_0\big(\tr_{\C^n}[g(\Real A_1)];\Lambda,\Gamma\big)
			+\mathfrak{W}_0\big(\tr_{\C^n}[g(\Real A_2)];\Lambda,\Gamma^c\big) \Big].
\end{multline}
\end{proof}

Now we want to estimate the trace norm
\begin{align}\label{Result-Smooth}
	\big\|g\big(G_L(A_1,A_2)\big)-D_L\big(g(\Real A_1),g(\Real A_2)\big)\big\|_1.
\end{align}
The line of argumentation is similar to the analytic case, where each ingredient is replaced by a counterpart for smooth functions.

\begin{lem}\label{Compact-Support-Smooth}
Let $A_1,A_2 \in C^{\infty}_{b}(\R^{d} \times \R^{d}, \C^{n\times n})$ be matrix-valued symbols with $A_2$ being compactly supported in the second variable. Further, let $\Lambda$ and $\Gamma$ be bounded admissible domains. Then there exist symbols $B_1,B_2\in C^{\infty}_{b}(\R^{d} \times \R^{d}, \C^{n\times n})$ which are compactly supported in both variables and with $(A_1-B_1)|_{\Lambda\times\Gamma}=0,(A_2-B_2)|_{\Lambda\times\R^d}=0$ such that for any smooth function $g\in C_c^\infty(\R)$  with $g(0)=0$ we have
\begin{align}\label{Compactification-Smooth-g-inside}
	D_L\big(g(\Real A_1),g(\Real A_2)\big) \sim D_L\big(g(\Real B_1),g(\Real B_2)\big)
\end{align} 
and
\begin{align}\label{Compactification-Smooth-g}
	g\big(G_L(A_1,A_2)\big)\sim g\big(G_L(B_1,B_2)\big).
\end{align} 
\end{lem}
\begin{proof} Let $L \ge 1$. We find the compactly supported symbols by introducing smooth cutoff functions in the same way as in the proof of Lemma \ref{Compact-Support-Analytic}: Let $\phi,\psi\in C_c^\infty(\R^d)$ be as in (\ref{Definition-Phi-Psi}) and define $B_1:=A_1 \phi\psi$ and $B_2:=A_2 \phi$.

Note that in the proof of (\ref{Compactification-Analytic-g-inside}) it was sufficient that $g(A_j),g(B_j)\in C^{\infty}_{b}(\R^{d} \times \R^{d}, \C^{n\times n})$ are well-defined symbols for $j\in\{1,2\}$, which are compactly supported whenever the respective symbols $A_j,B_j$ are compactly supported. For the Hermitian matrix-valued symbols $\Real A_j, \Real B_j$ this is the case if $g$ is smooth with $g(0)=0$. Therefore, we get
\begin{align}
	D_L\big( g(\Real A_1), g(\Real A_2)\big) 
	\sim D_L\big( g(\Real B_1), g(\Real B_2)\big).
\end{align}
This proves (\ref{Compactification-Smooth-g-inside}), even without the requirement that $g$ is compactly supported.

In order to prove (\ref{Compactification-Smooth-g}), we apply (\ref{Compactification-Smooth-g-inside}) with $g=\id$, i.e.
\begin{align}\label{Compactification-Smooth}
D_L(\Real A_1,\Real A_2)\sim D_L(\Real B_1,\Real B_2).
\end{align} 
Using (\ref{Similarity-Real-Part}) on both sides, this extends to
\begin{align}
 G_L\big(A_1,A_2\big)\sim G_L\big(B_1, B_2\big).
\end{align}
We write 
\begin{align}\label{Compactification-Fourier}
	g\big(G_L(A_1,A_2)\big) - g\big(G_L(B_1,B_2)\big)
	=\frac{1}{\sqrt{2\pi}}\int_{\R} \left( \e^{\i tG_L(A_1,A_2)}-\e^{\i tG_L(B_1,B_2)}\right) 
		\hat{g}(t) \,\d t,
\end{align}
where $\hat{g}\in\mathcal{S}(\R)$ is the Fourier transform of $g$. The operator $E(t):= \e^{\i tG_L(A_1,A_2)}-\e^{\i tG_L(B_1,B_2)}$ satisfies the following differential equation
\begin{align}
\i\partial_t E(t)+G_L(A_1,A_2)E(t)=-\big(G_L(A_1,A_2)-G_L(B_1,B_2)\big) \e^{\i tG_L(B_1,B_2)}=: M(t).
\end{align}
As $E(0)=0$, we get
\begin{align}
E(t)=-\i\int_0^t \e^{\i(t-s)G_L(A_1,A_2)} M(s)\,\d s.
\end{align}
As $G_L(B_1,B_2)$ is self-adjoint, we have
\begin{align}
\|M(t)\|_1\leq \|G_L(A_1,A_2)-G_L(B_1,B_2)\|_1\leq CL^{d-1},
\end{align}
where $C$ is independent of $t$. The operator $G_L(A_1,A_2)$ is also self-adjoint. Therefore,
\begin{align}
\|E(t)\|_1\leq CL^{d-1}t.
\end{align}
By (\ref{Compactification-Fourier}) we conclude that
\begin{align}
	\big\|g\big(G_L(A_1,A_2)\big)-g\big(G_L(B_1,B_2)\big)\big\|_1\leq CL^{d-1}.
\end{align}
\end{proof}

Next we quote a special case of \cite[Cor.\ 2.11]{sobolevc2} adapted to our situation.
\begin{lem}\label{Projections-C2}
Let $g\in C_c^2(\R)$ and $q\in \,]0,1[\,$. Let $X$ be a self-adjoint operator on a dense domain $\mathcal{D}$ in a separable Hilbert space $\mathcal{H}$ and $P$ be an orthogonal projection on $\mathcal{H}$ such that $P\mathcal{D} \subseteq \mathcal{D}$ and $PX(\mathbb{1}_{\mathcal{H}}-P)\in\mathcal{T}_q$. Then there exists a constant $C>0$, independent of $X,P$ and $g$, such that 
\begin{align}
\|g(PXP)P-Pg(X)\|_1\leq C \max_{0 \leq k \leq 2} \|g^{(k)}\|_\infty \|PX(\mathbb{1}_{\mathcal{H}}-P)\|_q^q.
\end{align}
\end{lem}

\begin{lem}\label{Operator-Smooth}
Let $A\in C^{\infty}_{b}(\R^{d} \times \R^{d}, \C^{n\times n})$ be a matrix-valued symbol with compact support in both variables and let $g\in C_c^\infty(\R)$ be a smooth function with compact support such that $g(0)=0$. Let the matrix-valued amplitude $A_g\in C^{\infty}_{b}(\R^{d} \times \R^{d} \times \R^{d}, \C^{n\times n})$ be defined as follows
\begin{align}
A_g(x,y,\xi):= g\Big(\tfrac{1}{2}\big(\Real A(x,\xi) + \Real A(y,\xi)\big)\Big)
\end{align}
for $x,y,\xi\in\R^d$.
Then we have
\begin{align}
g\big(\Real \Op_L^l(\Real A)\big) \sim \Op_L^{lr}(A_g)\sim  \Op_L^l\big(g(\Real A)\big).
\end{align}
\end{lem}
\begin{proof}
Let $L \ge 1$. First note that $\Real \Op_L^l(\Real A)=\Op_L^{lr}(A_{\id}).$ Using linearity, we write
\begin{align}
	\label{Operator-Fourier}
	g\big(\Op_L^{lr}(A_{\id})\big)-\Op_L^{lr}(A_g) =\frac{1}{\sqrt{2\pi}}\int_{\R} \left( \e^{\i t\Op_L^{lr}(A_{\id})}-\Op_L^{lr}(\e^{\i tA_{\id}})\right) \hat{g}(t) \,\d t,
\end{align}
where $\hat{g}\in\mathcal{S}(\R)$ is again the Fourier transform of $g$.  The operator $E(t):= \Op_L^{lr}(\e^{\i tA_{\id}})- \e^{\i t\Op_L^{lr}(A_{\id})}$ satisfies the following differential equation
\begin{align}
\i\partial_t E(t)+\Op_L^{lr}(A_{\id})E(t)=\Op_L^{lr}(A_{\id})\Op_L^{lr}(\e^{\i tA_{\id}})-\Op_L^{lr}(A_{\id}\e^{\i tA_{\id}})=: M(t).
\end{align}
As $E(0)=0$, we get
\begin{align}\label{E-differential-equation}
E(t)=-\i\int_0^t \e^{\i(t-s)\Op_L^{lr}(A_{\id})} M(s)\,\d s.
\end{align}
The next goal is to estimate the trace norm of $M(t)$ for given $t\in\R$. The triangle inequality yields
\begin{align}\label{smooth-M-estimate}
	\|M(t)\|_1 &\leq \big\|\Op_L^{lr}(A_{\id})\Op_L^{lr}(\e^{\i tA_{\id}})
			-\Op_L^{lr}(A_{\id})\Op_L^{l}(\e^{\i t\Real A})\big\|_1 \nonumber \\ 
		&\quad + \big\|\Op_L^{lr}(A_{\id})\Op_L^{l}(\e^{\i t\Real A})
			-\Op_L^{l}(\Real A)\Op_L^{l}(\e^{\i t\Real A})\big\|_1 \nonumber \\ 
		&\quad + \big\|\Op_L^{l}(\Real A)\Op_L^{l}(\e^{\i t\Real A})
			-\Op_L^{l}(\Real A \e^{\i t\Real A})\big\|_1\nonumber \\ 
		&\quad + \big\|\Op_L^{l}(\Real A \e^{\i t\Real A}) 
			-\Op_L^{lr}( A_{\id} \e^{\i tA_{\id}})\big\|_1.
\end{align} 
We will now estimate each of these trace norms individually.
Let $\phi\in C_c^\infty(\R^{d})$ be a function such that $\phi(x)A(x,\xi)=A(x,\xi)$ for all $x,\xi\in\R^{d}$. We decompose the amplitude $\e^{\i tA_{\id}}$ as
\begin{equation}
\e^{\i tA_{\id}(x,y,\xi)} = \mathbb{1}_{n}+\phi(x)\big(\e^{\i tA_{\id}(x,y,\xi)}-\mathbb{1}_{n}\big)+\big(1-\phi(x)\big)\big(\e^{\i tA_{\id}(x,y,\xi)}-\mathbb{1}_{n}\big),
\label{e-amplitude-sum}
\end{equation}
where the second term is compactly supported in the variable $x$, and the third term is compactly supported in the variable $y$. 
Similarly, we write 
\begin{equation}
\e^{\i t\Real A(x,\xi)} = \mathbb{1}_{n}+\phi(x)\big(\e^{\i t \Real A(x,\xi)} -\mathbb{1}_{n}\big).
\end{equation}
Therefore and since $\e^{\i tA_{\id}(x,x,\xi)} = \e^{\i t\Real A(x,\xi)}$ for all $x,\xi \in\R^{d}$, 
the estimate \eqref{Left-LeftRight-Precise} is applicable and yields
\begin{align}
\big\|&\Op_L^{lr}(\e^{\i tA_{\id}})-\Op_L^{l}(\e^{\i t\Real A})\big\|_1 \nonumber \\
&\leq  \big\|\Op_L^{lr}\big(\phi (\e^{\i tA_{\id}}-\mathbb{1}_{n})\big)-\Op_L^{l}\big(\phi(\e^{\i t\Real A}-\mathbb{1}_{n})\big)\big\|_1+\big\|\Op_L^{lr}\big((1-\phi)(\e^{\i tA_{\id}}-\mathbb{1}_{n})\big)\big\|_1 \nonumber \\ 
&\leq C_1L^{d-1} \Big( \mathbf{N}^{(d+1,d+1,d+2)}\big(\phi(\e^{\i tA_{\id}}-\mathbb{1}_{n})\big)+\mathbf{N}^{(d+1,d+1,d+2)}\big((1-\phi)(\e^{\i tA_{\id}}-\mathbb{1}_{n})\big)\Big).
\label{e-amplitude-sum2}
\end{align}
Writing $A_{\id}(x,y,\xi)=\tfrac{1}{2}\Real A(x,\xi)+\tfrac{1}{2}\Real A(y,\xi)$, the estimate \eqref{Left-Right-Precise} yields
\begin{equation}
\big\|\Op_L^{lr}(A_{\id})-\Op_L^{l}(\Real A)\big\|_1 = \tfrac{1}{2}\big\|\Op_L^{r}(\Real A)-\Op_L^{l}(\Real A)\big\|_1 \leq C_2L^{d-1} \mathbf{N}^{(d+1,d+2)}(\Real A). 
\end{equation}
As the symbol $A$ is compactly supported in both variables, the estimate \eqref{Merge-Symbols-Precise-2} yields
\begin{multline}
	\big\|\Op_L^{l}(\Real A)\Op_L^{l}(\e^{\i t\Real A})	-\Op_L^{l}(\Real A \e^{\i t\Real A})\big\|_1 \\
	\leq C_3L^{d-1} \mathbf{N}^{(d+1,d+2)}(\Real A) \,\mathbf{N}^{(d+1,d+2)}(\e^{\i t\Real A}).
\end{multline}
We define the amplitudes $A^{(1)}(x,y,\xi):= \Real A(x,\xi)\e^{\i tA_{\id}(x,y,\xi)}$ and $A^{(2)}(x,y,\xi):= \Real A(y,\xi)\e^{\i tA_{\id}(x,y,\xi)}$, which are compactly supported in the variables $x$, respectively $y$. With estimate \eqref{Left-LeftRight-Precise} we obtain
\begin{align}\label{amplitude-estimate-4}
\big\|\Op_L^{l}(&\Real A \e^{\i t\Real A}) 
			-\Op_L^{lr}( A_{\id} \e^{\i tA_{\id}})\big\|_1\nonumber \\ &\leq \tfrac{1}{2}\big\|\Op_L^{l}(\Real A \e^{\i t\Real A}) -\Op_L^{lr}( A^{(1)})\big\|_1+\tfrac{1}{2}\big\|\Op_L^{l}(\Real A \e^{\i t\Real A}) -\Op_L^{lr}( A^{(2)})\big\|_1\nonumber \\ &\leq C_4L^{d-1} \mathbf{N}^{(d+1,d+1,d+2)}(A^{(1)}).
\end{align}
Combining the estimates \eqref{e-amplitude-sum2} -- \eqref{amplitude-estimate-4}, we obtain the following bound for the right-hand side of \eqref{smooth-M-estimate}
\begin{align}
\label{4er-summe}
&C_1L^{d-1} \|\Op_L^{lr}(A_{\id})\| \, \Big( \mathbf{N}^{(d+1,d+1,d+2)}\big(\phi(\e^{\i tA_{\id}}-\mathbb{1}_{n})\big)+\mathbf{N}^{(d+1,d+1,d+2)}\big((1-\phi)(\e^{\i tA_{\id}}-\mathbb{1}_{n})\big)\Big)\nonumber \\
	&\quad + C_2 L^{d-1} \|\Op_L^{l}(\e^{\i t\Real A})\| \, \mathbf{N}^{(d+1,d+2)}(\Real A)\nonumber\\
	&\quad + C_3L^{d-1} \mathbf{N}^{(d+1,d+2)}(\Real A) \,\mathbf{N}^{(d+1,d+2)}(\e^{\i t\Real A})
		\nonumber \\ 
	&\quad + C_4 L^{d-1} \mathbf{N}^{(d+1,d+1,d+2)}(A^{(1)}).  
\end{align}
In order to proceed we use Lemma~ \ref{Bounded-Lemma} to bound $ \|\Op_L^{lr}(A_{\id})\|$ uniformly in $L \ge 1$ and 
to estimate $\|\Op_L^{l}(\e^{\i t\Real A})\| \le C \mathbf{N}^{(d,d+1)}(\e^{\i t \Real A})$. It remains to estimate the $t$-dependence of the symbol norms. Each differentiation of the matrix exponentials $\e^{\i tA_{\id}}(x,y,\xi)$, respectively
$\e^{\i t \Real A}(x,\xi)$, with respect to the Cartesian components of $x,y,\xi$, respectively $x,\xi$, brings down a factor of $t$ according to Duhamel's formula
\begin{equation}
	\frac{\dd}{\dd s} \e^{B(s)} = \int_{0}^{1} \e^{\sigma  B(s)} \frac{\dd B}{\dd s} \e^{(1-\sigma)  B(s)} \dd\sigma.
\end{equation} 
Here, $s \mapsto B(s)$ is $C^{\infty}(\R, \C^{n\times n})$ and $s$ stands for any of the Cartesian components of $x,y,\xi$.
Together with the unitarity of $\e^{\i tA_{\id}}(x,y,\xi)$ and $\e^{\i t \Real A}(x,\xi)$, we infer from \eqref{4er-summe} and \eqref{smooth-M-estimate} that 
\begin{equation}
  \|M(t)\|_{1} \le C L^{d-1} t^{(3d+4)},
\end{equation}
where the constant $C$ is independent of $L$ and $t$. 
As the operator $\Op_L^{lr}(A_{\id})$ is self-adjoint, we conclude from \eqref{E-differential-equation} that
\begin{align}
\|E(t)\|_1\leq CL^{d-1}t^{(3d+5)}.
\end{align}
By \eqref{Operator-Fourier} we conclude that
\begin{align}
g\big(\Real \Op_L^l(\Real A)\big)\sim \Op_L^{lr}(A_g).
\end{align}
With the same function $\phi$ as in \eqref{e-amplitude-sum}, we write $A_g(x,y,\xi)=\phi(x) A_g(x,y,\xi) + \big(1-\phi(x)\big) A_g(x,y,\xi)$ and note that $\phi A_g$ is compactly supported in the variable $x$ and $(1-\phi) A_g$ is compactly supported in the variable $y$. Applying relation \eqref{Left-LeftRight} from Lemma \ref{Exchange-Lemma}, yields
\begin{multline}
 \Op_L^{lr}(A_g) = \Op_L^{lr}(\phi A_g)+\Op_L^{lr}\big((1-\phi)A_g\big) \\ \sim \Op_L^l\big(\phi g(\Real A)\big)+ \Op_L^l\big((1-\phi)g(\Real A)\big) = \Op_L^l\big(g(\Real A)\big).
\end{multline}
\end{proof}

We now combine the results for smooth functions in order to obtain the estimate for \eqref{Result-Smooth}.

\begin{lem}\label{Combination-Smooth}
	Let $A_1,A_2\in C^{\infty}_{b}(\R^{d} \times \R^{d}, \C^{n\times n})$ be matrix-valued symbols with $A_2$ being compactly supported 
	in the second variable. Further, let $\Lambda$ and $\Gamma$ be bounded admissible domains. 
	For any function $g\in C^\infty_c(\R)$ with $g(0)=0$ there exist constants $C_1,C_2>0$ with $C_1$ 
	independent of $g$, $\Lambda$ and $\Gamma$, such that 
	\begin{equation}
		\label{Combination-Smooth-Statement}
		\big\|g\big(G_L(A_1,A_2)\big)- D_L\big(g(\Real A_1),g(\Real A_2)\big)\big\|_1 
		\leq C_1  \max_{0 \leq k \leq 2} \|g^{(k)}\|_\infty  L^{d-1}\log L  + C_2  L^{d-1} 
	\end{equation}
	for every $L\geq 2$.
\end{lem}

\begin{proof}
The proof resembles the one of Lemma \ref{Combination-Analytic}. Let $L\geq 2$.
As Lemma $\ref{Compact-Support-Smooth}$ yields
\begin{align}
	g\big(G_L(A_1,A_2)\big) \sim g\big(G_L(B_1,B_2)\big)
\end{align} 
and
\begin{align}
	D_L\big(g(\Real A_1),g(\Real A_2)\big) \sim D_L\big(g(\Real B_1),g(\Real B_2)\big),
\end{align}
it suffices to show \eqref{Combination-Smooth-Statement} with $A_1$ and $A_2$ replaced by the compactly supported symbols $B_1$ and $B_2$.

Lemma \ref{Projections-C2} with $0<q<1$ applied to $g$, $\big(G_L(B_1,B_2;\R^d,\Gamma)\big)$ and $\mathbf{1}_\Lambda$ combined with Lemma \ref{Trace class estimates} yields 
\begin{align}\label{Combination-Smooth-Projections}
	\Big\|g\big(G_L(B_1,B_2)\big) &- \mathbf{1}_\Lambda g\big(G_L(B_1,B_2;\R^d,\Gamma)\big) 
			\mathbf{1}_\Lambda\Big\|_1 \nonumber \\ 
	&= \Big\|\big[g\big(G_L(B_1,B_2)\big)\mathbf{1}_\Lambda 
			-\mathbf{1}_\Lambda g\big(G_L(B_1,B_2;\R^d,\Gamma)\big)\big] \mathbf{1}_\Lambda\Big\|_1 
			\nonumber  \\  
	&\leq  C' \max_{0 \leq k \leq 2} \|g^{(k)}\|_\infty \Big(\| \mathbf{1}_\Lambda \Op_L(\mathbf{1}_\Gamma)
			\Real \Op_L^l(\Real B_1) \Op_L(\mathbf{1}_\Gamma) \mathbf{1}_{\Lambda^c}\|_1  \nonumber \\ 
	&\hspace*{3.5cm} + \| 1_\Lambda \Op_L(\mathbf{1}_{\Gamma^c}) 
			\Real \Op_L^l(\Real B_2) \Op_L(\mathbf{1}_{\Gamma^c}) \mathbf{1}_{\Lambda^c}\|_1 \Big)\nonumber \\
	&\leq C_1 \max_{0 \leq k \leq 2} \|g^{(k)}\|_\infty L^{d-1}\log L  .
\end{align}
Here $C'$ is the constant from Lemma \ref{Projections-C2}, and $C_1$ is already the constant 
	appearing in the claim in \eqref{Combination-Smooth-Statement}. For its independence of $g$, $\Lambda$ and $\Gamma$ we refer to Lemmas \ref{Projections-C2} and \ref{Trace class estimates}.

By an elementary property of the functional calculus we get
\begin{align}
	g\big(G_L(B_1,B_2;\R^d,\Gamma)\big) 
	&= g\big( \Op_L(\mathbf{1}_\Gamma)\Real \Op_L^l(\Real B_1)\Op_L(\mathbf{1}_\Gamma)\big)\nonumber\\
	&\quad +g\big( \Op_L(\mathbf{1}_{\Gamma^c})\Real \Op_L^l(\Real B_2)\Op_L(\mathbf{1}_{\Gamma^c})\big).
\end{align} 
Similarly to \eqref{Combination-Smooth-Projections}, we now apply Lemma \ref{Projections-C2} to $\Real \Op_L^l(\Real B_1)$  and $\Op_L(\mathbf{1}_\Gamma)$, followed by Lemma \ref{Trace class estimates}, which yields
\begin{multline}
	\Big\|g\big(\Op_L(\mathbf{1}_\Gamma)\Real \Op_L^l(\Real B_1)\Op_L(\mathbf{1}_\Gamma)\big)
		- \Op_L(\mathbf{1}_\Gamma)g\big( \Real \Op_L^l(\Real B_1)\big)\Op_L(\mathbf{1}_\Gamma)\Big\|_1 \\ 
	\leq C \max_{0 \leq k \leq 2} \|g^{(k)}\|_\infty \big\|\Op_L(\mathbf{1}_\Gamma)
		\Real \Op_L^l(\Real B_1) \Op_L(\mathbf{1}_{\Gamma^c})\big\|_q^q \leq C L^{d-1}.
\end{multline}
Applying Lemma \ref{Projections-C2} again to $\Real \Op_L^l(\Real B_2)$ and $\Op_L(\mathbf{1}_{\Gamma^c})$, we conclude
\begin{align}
	\mathbf{1}_\Lambda g\big(G_L(B_1,B_2;\R^d,\Gamma)\big) \mathbf{1}_\Lambda 
	\sim \mathbf{1}_\Lambda \Big( & \Op_L(\mathbf{1}_\Gamma)g\big(\Real \Op_L^l(\Real B_1)\big)
		\Op_L(\mathbf{1}_\Gamma)\nonumber\\
	&+ \Op_L(\mathbf{1}_{\Gamma^c})g \big(\Real \Op_L^l(\Real B_2)\big)\Op_L(\mathbf{1}_{\Gamma^c})\Big)
		\mathbf{1}_\Lambda.
\end{align}
Finally Lemma \ref{Operator-Smooth} gives
\begin{align}
\mathbf{1}_\Lambda \Op_L(\mathbf{1}_\Gamma)g\big(\Real \Op_L^l(\Real B_1)\big)\Op_L(\mathbf{1}_\Gamma)\mathbf{1}_\Lambda\sim T_L(g(\Real B_1);\Lambda,\Gamma),
\end{align}
as well as
\begin{align}
	\mathbf{1}_\Lambda \Op_L(\mathbf{1}_{\Gamma^c})g \big(\Real \Op_L^l(\Real B_2)\big) 
		\Op_L(\mathbf{1}_{\Gamma^c})\mathbf{1}_\Lambda
	\sim T_L\big(g(\Real B_2);\Lambda,\Gamma^c\big),
\end{align}
which concludes the proof, since 
\begin{equation}
	D_L\big(g(\Real B_1),g(\Real B_2);\Lambda,\Gamma\big)
	=T_L\big(g(\Real B_1);\Lambda,\Gamma\big) + T_L\big(g(\Real B_2);\Lambda,\Gamma^c\big).
\end{equation}
\end{proof}

\begin{lem}
	\label{Coefficient-C2}
	Let $A_1,A_2\in C^{\infty}_{b}(\R^{d} \times \R^{d}, \C^{n\times n})$ be Hermitian matrix-valued symbols. 
	Let the function $g\in C^1(\R)$ be differentiable with $g(0)=0$. Let $\partial\Lambda$ and $\partial\Gamma$ have finite $(d-1)$-dimensional surface measure induced by Lebesgue measure on $\R^d$. Then there exists a constant $C>0$, which does not depend on $g$, 
	such that
	\begin{align}
		\big|\mathfrak{W}_1\big(\mathfrak{U}(g;A_1,A_2);\partial\Lambda,\partial\Gamma\big)\big|
		\leq C \|g'\|_\infty.
	\end{align}
\end{lem}

\begin{proof}
As the surface measure of $\partial\Lambda$ and $\partial\Gamma$ is finite, it suffices to show 
\begin{align}
\sup_{(x,\xi)\in\partial\Lambda\times\partial\Gamma}\big|\big(\mathfrak{U}(g;A_1,A_2)\big)(x,\xi)\big|\leq  C \|g'\|_\infty.
\end{align}
We use the fact that $\frac{1}{t(1-t)}=\frac{1}{t}+\frac{1}{(1-t)}$ to work towards a pointwise estimate for the function 
\begin{align}\label{Coefficient-C2-Two-Integrals}
(2\pi)^2\mathfrak{U}(g;A_1,A_2)\nonumber =&\int_0^1 \frac{\tr_{\C^n}\big[g\big(A_1t+A_2(1-t)\big)-g(A_1)t-g(A_2)(1-t)\big]}{t}\,\d t \nonumber \\&+\int_0^1 \frac{\tr_{\C^n}\big[g\big(A_1t+A_2(1-t)\big)-g(A_1)t-g(A_2)(1-t)\big]}{1-t}\,\d t.
\end{align}
For the first integral we get 
\begin{multline}\label{Coefficient-C2-Integral-1}
	\int_0^1 \frac{\tr_{\C^n}\big[g\big(A_1t+A_2(1-t)\big)-g(A_1)t-g(A_2)(1-t)\big]}{t}\,\d t \\
		= \int_0^1 \frac{\tr_{\C^n}\big[g\big(A_1t+A_2(1-t)\big)-g(A_2)\big]}{t}\,\d t 
			-\tr_{\C^n}\big[g(A_1)-g(A_2)\big].
\end{multline}
We now require an estimate for the difference $\tr_{\C^n}[g(M_1)-g(M_2)]$  for Hermitian matrices $M_1,M_2\in\C^{n\times n}$ in terms of $\|g'\|_\infty$. We estimate
\begin{equation} \label{Matrix-C1-Estimate}
	\big|\tr_{\C^n}[ g(M_1)-g(M_2)]\big| \leq \int_0^1 \Big|\frac{\d}{\d t}
		\tr_{\C^n}\big[g\big(M_2+t(M_1-M_2)\big)\big]\Big| \,\d t 
	\leq \|g'\|_\infty  \tr_{\C^n}|M_1-M_2|. 
\end{equation}
Noting that $A_1t+A_2(1-t)-A_2=A_1t-A_2t$, an application of (\ref{Matrix-C1-Estimate}) yields
\begin{align}
\tr_{\C^n}\big[g\big(A_1t+A_2(1-t)\big)-g(A_2)\big]\leq \|g'\|_\infty (\|A_1\|_{\infty,1}+\|A_2\|_{\infty,1})\,t,
\end{align}
where the norms are defined in (\ref{Definition-Norm-Infinity-1}).
Using this, we estimate the last line of (\ref{Coefficient-C2-Integral-1}) by
\begin{multline}
\|g'\|_\infty \big(\|A_1\|_{\infty,1}+\|A_2\|_{\infty,1}\big)\int_0^1 \frac{t}{t}\,\d t \ + \ \|g'\|_\infty \big(\|A_1\|_{\infty,1}+\|A_2\|_{\infty,1}\big) \nonumber \\ = 2\|g'\|_\infty \big(\|A_1\|_{\infty,1}+\|A_2\|_{\infty,1}\big).
\end{multline}
Similarly, we get for the second integral
\begin{align}
	\int_0^1 &\frac{\tr_{\C^n}\big[g\big(A_1t+A_2(1-t)\big) - g(A_1)t-g(A_2)(1-t)\big]}{1-t}\,\d t 
		\nonumber \\
	&=  \int_0^1 \frac{\tr_{\C^n}\big[g\big(A_1t+A_2(1-t)\big)-g(A_1)\big]}{1-t}\,\d t 
		+ \tr_{\C^n}[g(A_1)-g(A_2)] \nonumber \\
	&\leq   \ 2\|g'\|_\infty \big(\|A_1\|_{\infty,1}+\|A_2\|_{\infty,1}\big).
\end{align}
 Therefore, we have $\sup_{(x,\xi)\in\partial\Lambda\times\partial\Gamma}\big|\big(\mathfrak{U}(g;A_1,A_2)\big)(x,\xi)\big|\leq   C \|g'\|_\infty$.
\end{proof}

With this, we are again ready to close the asymptotics.

\begin{thm}
	\label{Asymptotics-Smooth}
	Let $A_1,A_2\in C^{\infty}_{b}(\R^{d} \times \R^{d}, \C^{n\times n})$ be matrix-valued symbols with $A_2$ being compactly supported in the 
	second variable. Let $\Lambda$ be a bounded piece-wise $C^1$-admissible domain and $\Gamma$ 
	be a bounded piece-wise $C^3$-admissible domain. Let the function $h\in C^\infty(\R)$ be 
	smooth with $h(0)=0$, then 
	\begin{align}
		\tr_{L^2(\R^d)\otimes\C^n}\big[h\big(G_L(& A_1,A_2;\Lambda, \Gamma)\big)\big] \nonumber \\ 
		=& \ L^d \Big[\mathfrak{W}_0\big(\tr_{\C^n}[h(\Real A_1)];\Lambda,\Gamma\big)
			 + \mathfrak{W}_0\big(\tr_{\C^n}[h(\Real A_2)];\Lambda,\Gamma^c\big)\Big]\nonumber\\
		&+ L^{d-1}\log L \ \mathfrak{W}_1\big(\mathfrak{U}(h;\Real A_1,\Real A_2);\partial\Lambda,
				\partial\Gamma\big)\nonumber\\ 
		&+o(L^{d-1}\log L),
	\end{align}
	as $L\rightarrow\infty$. Here, the coefficients $\mathfrak{W}_0$ and $\mathfrak{W}_1$ were defined 
	in \eqref{coeff-w0-def} and \eqref{coeff-w1-def},	respectively, and the symbol $\mathfrak{U}$ 
	in \eqref{frakU-def}.
\end{thm}
\begin{proof}
As the operator $G_L(A_1,A_2;\Lambda, \Gamma)$ is bounded, we can assume that the support $\supp(h)$ of $h$ is compact. Let $\epsilon>0$ be arbitrary.
By the theorem of Stone--Weierstrass we can find a polynomial $g_\epsilon$ defined on $\supp(h)$ such that $h_\epsilon:=h-g_\epsilon$ satisfies the bound 
\begin{align}
\max_{0 \leq k \leq 2} \|h_\epsilon^{(k)}\|_\infty<\frac{\epsilon}{2},
\end{align}
where the supremum norm is the one for functions on $\supp(h)$.
Applying Lemma $\ref{Combination-Smooth}$ yields
\begin{equation}
	\big\| h_\epsilon\big(G_L(A_1,A_2)\big) - D_L\big(h_\epsilon(\Real A_1),h_\epsilon( \Real A_2)\big)
		\big\|_1
	\leq C_1\frac{\epsilon}{2} L^{d-1}\log L  + C_2  L^{d-1}.
\end{equation}
As
\begin{align}
	\tr_{L^2(\R^d)\otimes\C^n}\big[h\big(G_L(A_1,A_2)\big) &- D_L\big(h(\Real A_1),h(\Real A_2)\big)\big]
		\nonumber \\ 
	&\leq  \tr_{L^2(\R^d)\otimes\C^n}\big[g_\epsilon\big(G_L(A_1,A_2)\big) - 
		D_L\big(g_\epsilon(\Real A_1),g_\epsilon(\Real A_2)\big)\big]\nonumber \\
	&\quad+ \big\|h_\epsilon\big(G_L(A_1,A_2)\big) 
		- D_L\big(h_\epsilon(\Real A_1),h_\epsilon(\Real A_2)\big)\big\|_1,
\end{align}
and the asymptotic formula is already established for $g_\epsilon$, we get
\begin{align}
	\underset{L\rightarrow\infty}{\limsup}\frac{\tr_{L^2(\R^d)\otimes\C^n}\big[h\big(G_L(A_1,A_2)\big)
		- D_L\big(h(\Real A_1),h(\Real A_2)\big)\big]}{L^{d-1}\log L}
	\leq \mathfrak{W}_1(g_\epsilon) + C_1\frac{\epsilon}{2},
\end{align}
where $\mathfrak{W}_1(g_\epsilon):=\mathfrak{W}_1\big(\mathfrak{U}(g_\epsilon;\Real A_1,\Real A_2);\partial\Lambda,\partial\Gamma\big).$
Lemma \ref{Coefficient-C2} yields
\begin{align}
|\mathfrak{W}_1(h_\epsilon)|\leq C \|h_\epsilon'\|_\infty<C\frac{\epsilon}{2}.
\end{align}
Hence
\begin{align}
\mathfrak{W}_1(g_\epsilon)\leq \mathfrak{W}_1(h)+|\mathfrak{W}_1(h_\epsilon)|<\mathfrak{W}_1(h)+C\frac{\epsilon}{2}.
\end{align}
Therefore, 
\begin{align}
	\underset{L\rightarrow\infty}{\limsup}\frac{\tr_{L^2(\R^d)\otimes\C^n}\big[h\big(G_L(A_1,A_2)\big)
		- D_L\big(h(\Real A_1),h(\Real A_2)\big)\big]}{L^{d-1}\log L}
	\leq  \mathfrak{W}_1(h) + C\epsilon.
\end{align}
The corresponding lower bound for the liminf is found in the same way. As $\epsilon$ is arbitrary, we get
\begin{align}
	\underset{L\rightarrow\infty}{\lim} \frac{\tr_{L^2(\R^d)\otimes\C^n}\big[h\big(G_L(A_1,A_2)\big)
		- D_L\big(h(\Real A_1),h(\Real A_2)\big)\big]}{L^{d-1}\log L}
	=  \mathfrak{W}_1(h).
\end{align}
The trace of $D_L\big(h(\Real A_1),h(\Real A_2)\big)$ gives rise to the volume terms by Lemma \ref{Volume-Smooth}, and we obtain the desired asymptotics
\begin{align}
	\tr_{L^2(\R^d)\otimes\C^n}\big[h\big(G_L(&A_1,A_2;\Lambda, \Gamma)\big)\big] \nonumber \\ 
	=&\ L^d \big[\mathfrak{W}_0\big(\tr_{\C^n}[h(\Real A_1)];\Lambda,\Gamma\big)+\mathfrak{W}_0\big(\tr_{\C^n}[h(\Real A_2)];\Lambda,\Gamma^c\big)\big]\nonumber\\&+L^{d-1}\log L \ \mathfrak{W}_1\big(\mathfrak{U}(h;\Real A_1,\Real A_2);\partial\Lambda,\partial\Gamma\big)\nonumber\\ &+o(L^{d-1}\log L),
\end{align}
as $L\rightarrow\infty$.
\end{proof}

%%%%%%%%%%%%%%%%%%%%%%%%%%%%%%%%%%%%%%%%%%%%%%%%%%%%%%%%%%%%%%%%%%%%%%%%%

\subsection{More general functions}
\label{subsec:Hoelder}

We now consider a more general class of test functions which includes all R\'enyi entropy functions. Studying these test functions is greatly motivated by the applications to entanglement entropies of non-interacting Fermi gases.

\begin{ass}
	\label{H-Functions}
	Let $\gamma\in\,]0,1]$ and let $\mathcal X: =\{x_1,x_2,\ldots,x_N\}\subset\R,N\in\N,$ be a finite collection of different points 
	on the real line. Let $U_j\subset\R$, $j\in\{1,\ldots , N\}$, be pairwise disjoint 
	neighbourhoods of the points $x_j\in \mathcal X$. Given a function 
	$h\in C(\R) \cap C^2(\R \setminus \mathcal X)$, we assume the existence of a constant $C>0$ such that for every  $k\in\{0,1,2\}$ the estimate 
	\begin{align}\label{H-Functions-Bound}
		\Big|\frac{\d^k}{\d x^k}\big[h-h(x_j)\big](x)\Big|\leq C |x-x_j|^{\gamma-k} 
	\end{align}
	holds for every $x\in U_j\setminus\{x_j\}$ and every $j\in\{1,\ldots , N\}$. In particular, this 
	implies that $h$ is H\"older continuous at the points of $\mathcal X$.
\end{ass}

As Lemma \ref{Operator-Smooth} is not available for $C_c^2$-functions, we will restrict ourselves to the case that the symbols $A_1$ and $A_2$ only depend on the variable $\xi$ from now on. 
In Lemma \ref{Projections-H-Functions} and the proof of Theorem \ref{Asymptotics-Hoelder}, which is the main result of this section, we will decompose $h=h_1+h_2$, where $h_1\in C^2(\R)$ and $h_2$ is a finite sum of functions of the following type.

\begin{ass}
	\label{Definition-Hoelder-Sobolev}
	Let $\gamma\in\,]0,1]$, $x_0\in\R$, $R>0$ and $I:= \,]x_0-R,x_0+R[\,$. 
	Given a function $g\in C_{c}(I) \cap C^2(I\setminus \{x_0\})$, we assume finiteness of the norms 
	\begin{align}
		\label{Hoelder-Sobolev-Bound}
		\bnorm{g}_{l}:= \max_{0\leq k\leq l}\sup_{x\in I\setminus \{x_0\}} 
		\big[|g^{(k)}(x)||x-x_0|^{-\gamma+k}\big] < \infty,
	\end{align}
	for $l \in\{0,1,2\}$. In particular, this implies that $g$ is H\"older continuous 
	at $x_{0}$ with $g(x_0)=0$.
\end{ass}

Before we turn to the main theorem of this section, we first collect two additional ingredients. The first one is a version of Lemma \ref{Projections-C2} for functions satisfying Assumption \ref{Definition-Hoelder-Sobolev}. We quote a special case of \cite[Thm.\ 2.10]{sobolevc2} adapted to our situation.

\begin{thm}\label{Projections-Hoelder}
Let $g$ satisfy Assumption \ref{Definition-Hoelder-Sobolev}. Let $q\in \,]0,\gamma[\,$. Further, let $X$ be a self-adjoint operator with dense domain $\mathcal{D}$ in a separable Hilbert space $\mathcal{H}$ and $P$ be an orthogonal projection on $\mathcal{H}$ such that $P\mathcal{D} \subseteq \mathcal{D}$ and $PX(\mathbb{1}_{\mathcal{H}}-P)\in\mathcal{T}_q$. Then
\begin{align}
\|g(PXP)P-Pg(X)\|_1\leq C \bnorm{g}_2 R^{\gamma-q} \|PX(\mathbb{1}_{\mathcal{H}}-P)\|_q^q,
\end{align}
with a positive constant $C$ independent of $X,P,g$ and $R$.
\end{thm}
We now combine this estimate with Lemma \ref{Projections-C2} to get a similar estimate for functions satisfying Assumption \ref{H-Functions}.
\begin{lem}\label{Projections-H-Functions}
Let $h$ satisfy Assumption \ref{H-Functions} and be compactly supported. Then there exists $R_{0}>0$ such that for all $R\in\,]0,R_0]$ the function $h$ can be decomposed as $h=\sum_{j=1}^Nh_{R,j}+g_{R}$, with $h_{R,j}$ satisfying Assumption \ref{Definition-Hoelder-Sobolev} with $x_{0}=x_{j}$ and the same H\"older exponent $\gamma\in\,]0,1]$ as $h$ for every $j\in\,\{1,\ldots,N\}$  and $g_R\in C_c^2(\R)$. Further, let $q\in \,]0,\gamma[\,$, $X$ be a self-adjoint operator with dense domain $\mathcal{D}$ in a separable Hilbert space 
$\mathcal{H}$ and $P$ be an orthogonal projection on $\mathcal{H}$ such that $P\mathcal{D} \subseteq \mathcal{D}$ and $PX(\mathbb{1}_{\mathcal{H}}-P)\in\mathcal{T}_q$. Then there exist constants $C_j>0$, $j\in\{0,\ldots ,N\}$, which are independent of $X$, $P$, $h$ and $R$ such that
\begin{align}
\|h(PXP)P-Ph(X)\|_1\leq \Big(\sum_{j=1}^{N} C_j \bnorm{h_{R,j}}_2 R^{\gamma-q} + C_0 \max_{0 \leq k \leq 2} \|g_{R}^{(k)}\|_\infty \Big) \|PX(\mathbb{1}_{\mathcal{H}}-P)\|_q^q.
\end{align}
\end{lem}
\begin{proof}
Let $g\in C^2_c(\R)$ be a function such that $g(x_j)=h(x_j)$ for all $x_j\in \mathcal X$, $j\in\{1,\ldots , N\}$. Then the function $h-g$ also satisfies Assumption \ref{H-Functions} and $(h-g)(x_j)=0$ for all $x_j\in \mathcal X$, $j\in\{1,\ldots , N\}$. \\ Let $R_0>0$ be small enough such that $\,]x_j-R_{0},x_j+R_{0}[\,\subset U_j$ for every $j\in\{1,\ldots , N\}$. For any $R\in\,]0,R_0]$ write $h-g=\sum_{j=1}^N h_{R,j}+f_R$ with $h_{R,j}(x):=[h(x)-g(x)]\zeta((x-x_j)R^{-1})$, where $\zeta\in C_c^\infty(\R)$ with $\zeta(x)=1$ for all $x\in \,]-\frac{1}{2},\frac{1}{2}[\,$ and $\zeta(x)=0$ for all $x\notin \,]-1,1[\,$ as well as $||\zeta||_\infty=1$. Then $f_R\in C_c^2(\R)$ and $\,]x_j-R,x_j+R[\,\subset U_j$ for every $j\in\{1,\ldots , N\}$. We now verify that each $h_{R,j}$ satisfies Assumption \ref{Definition-Hoelder-Sobolev} with $x_0=x_j$ and $\bnorm{h_{R,j}}_2$ bounded uniformly in $R\in\,]0,R_0]$. Indeed, let $j\in\{1,\ldots , N\}$ and $R\in\,]0,R_0]$ be arbitrary. Then for every $x\in \,]x_j-R,x_j+R[\,\setminus\{x_j\}$ the bound \eqref{H-Functions-Bound} applied to $h-g$ yields
\begin{align} 
	|h_{R,j} & (x)||x-x_j|^{-\gamma} \notag \\
	&=\big|\big[h(x)-g(x)\big]\zeta\big((x-x_j)R^{-1}\big)\big||x-x_j|^{-\gamma}
		\leq \big|h(x)-g(x)\big||x-x_j|^{-\gamma}\leq C 
\end{align}
and
\begin{align}
	|h_{R,j}'&(x)||x-x_j|^{-\gamma+1} \notag \\ 
	&=\Big|\big[h'(x)-g'(x)\big]\zeta\big((x-x_j)R^{-1}\big)+\big[h(x)-g(x)\big]R^{-1}\zeta'\big((x-x_j)R^{-1}\big)\Big|\,|x-x_j|^{-\gamma+1} \notag \\ 
	&\leq \big|h'(x)-g'(x)\big||x-x_j|^{-\gamma+1}+ ||\zeta'||_\infty\big|h(x)-g(x)\big|\frac{|x-x_j|}{R}|x-x_j|^{-\gamma}\leq C.
\end{align}
The second derivative works in the same way.
We now have the decomposition $h=\sum_{j=1}^N h_{R,j}+f_R+g=\sum_{j=1}^N h_{R,j}+g_R$, with $g+f_R=:g_R\in C_c^2(\R)$.

It remains to apply Lemma \ref{Projections-C2} to $g_R$ and Theorem \ref{Projections-Hoelder} to each $h_{R,j}$ to obtain the desired estimate:
\begin{align}
\|h(PXP)P-Ph(X)\|_1 &\leq \sum_{j=1}^{N} \|h_{R,j}(PXP)P-Ph_{R,j}(X)\|_1+\|g_{R}(PXP)P-Pg_{R}(X)\|_1 \nonumber\\ &\leq \Big(\sum_{j=1}^{N} C_j \bnorm{h_{R,j}}_2 R^{\gamma-q} + C_0 \max_{0 \leq k \leq 2} \|g_{R}^{(k)}\|_\infty \Big) \|PX(\mathbb{1}_{\mathcal{H}}-P)\|_q^q.
\end{align}
\end{proof}

The second ingredient is an estimate for the coefficient $\mathfrak{W}_1$. For this, we need an estimate similar to (\ref{Matrix-C1-Estimate}). While this is easy to come by in the scalar-valued case, it is more difficult in the case with two matrix-valued symbols. In order to do so, we quote a simpler version of 
\cite[Thm.\ 2.4]{sobolevc2}, which we will only need for a finite-dimensional Hilbert space.

\begin{thm}\label{No-Projections-Hoelder}
Let $g$ satisfy Assumption \ref{Definition-Hoelder-Sobolev}. Let $q\in \,]0,\gamma[\,$. Further, let $X_1,X_2$ be self-adjoint operators with dense domains $\mathcal{D}_{1}, \mathcal{D}_{2}$ in a separable Hilbert space $\mathcal{H}$ such that $\mathcal{D}_{1} \cap \mathcal{D}_{2}$ is dense in $\mathcal{H} $ and $X_1-X_2\in\mathcal{T}_q$. Then
\begin{align}
\|g(X_1)-g(X_2)\|_1\leq C \bnorm{g}_2 R^{\gamma-q} \|X_1-X_2\|_q^q,
\end{align}
with a positive constant $C$ independent of $X_1,X_2,g$ and $R$.
\end{thm}

Now we turn to the coefficient. The next Lemma allows the matrix-valued symbols to depend on the space variable, even though this will not be needed in the main theorem.

\begin{lem}\label{Coefficient-Hoelder}
Let $g$ satisfy Assumption \ref{Definition-Hoelder-Sobolev}. Let $A_1,A_2\in C^{\infty}_{b}(\R^{d} \times \R^{d}, \C^{n\times n})$ be Hermitian matrix-valued symbols and let $\partial\Lambda$ and $\partial\Gamma$ have finite $(d-1)$-dimensional surface measure induced by Lebesgue measure on $\R^d$. Let $q\in \,]0,\gamma[\,$.  Then there exists a constant $C>0$, independent of $g$ and $R$ such that
\begin{align}
	\big|\mathfrak{W}_1\big(\mathfrak{U}(g;A_1,A_2);\partial\Lambda,\partial\Gamma\big)\big|
	\leq C R^{\gamma -q} \bnorm{g}_2.
\end{align}
\end{lem}

\begin{proof}
As the surface measure of $\partial\Lambda$ and $\partial\Gamma$ is finite, it suffices to show 
\begin{align}
\sup_{(x,\xi)\in\partial\Lambda\times\partial\Gamma}\big|\big(\mathfrak{U}(g;A_1,A_2)\big)(x,\xi)\big|\leq  C R^{\gamma -q} \bnorm{g}_2.
\end{align}

As in the proof of Lemma \ref{Coefficient-C2}, the strategy is to estimate pointwise, use the fact that $\frac{1}{t(1-t)}=\frac{1}{t}+\frac{1}{(1-t)}$ and consider the two integrals in (\ref{Coefficient-C2-Two-Integrals}) separately. 
For Hermitian matrices $M_1,M_2\in\C^{n\times n}$, an application of Theorem \ref{No-Projections-Hoelder} with $\mathcal{H}=\C^n$ and $X_1=M_1,$ $X_2=M_2$ yields
\begin{align}\label{Distance-Hoelder}
\|g(M_1)-g(M_2)\|_1\leq C \bnorm{g}_2 R^{\gamma-q} \|M_1-M_2\|_q^q,
\end{align}
with $C$ independent of $M_1,M_2,g$ and $R$. Here the trace and $q$-norms are the appropriate matrix norms. 

The first integral is as in (\ref{Coefficient-C2-Integral-1}), and we get
\begin{align}
	\int_0^1 & \frac{\tr_{\C^n}\big[g\big(A_1t+A_2(1-t)\big)-g(A_2)\big]}{t}\,\d t 
		- \tr_{\C^n}[g(A_1)-g(A_2)] \nonumber\\ 
	&\leq C \bnorm{g}_2 R^{\gamma-q}\left( \int_0^1 \frac{\|(A_1t+A_2(1-t)-A_2\|_q^q}{t}\,\d t 
		+ \|A_1-A_2\|_q^q \right) \nonumber \\ 
	&\leq  C \bnorm{g}_2 R^{\gamma -q}\big(\|A_1\|_q^q+\|A_2\|_q^q\big)
		\left(\int_0^1 \frac{t^{q}}{t}\,\d t+1\right)
\end{align}
by applying (\ref{Distance-Hoelder}).
We estimate the second integral in an analogous way and use the fact that the estimates hold for all $(x,\xi)\in\partial\Lambda\times\partial\Gamma$ to get the result.
\end{proof}

We are now ready to close the asymptotics for test functions satisfying Assumption~\ref{H-Functions}, in particular for the R\'enyi entropy functions.
\begin{thm}
	\label{Asymptotics-Hoelder}
	Let $A_1,A_2\in C^{\infty}_{b}(\R^{d}, \C^{n\times n})$ be matrix-valued symbols, which only depend on the 
	momentum variable $\xi$. We assume that $A_2$ is compactly supported in $\xi$. Let $\Lambda$ be a bounded piece-wise 
	$C^1$-admissible domain and $\Gamma$ be a bounded piece-wise $C^3$-admissible domain. 
	Let the function $h$ satisfy Assumptions \ref{H-Functions} and assume 
	that $h(0)=0$. Then 
	\begin{align}
		\tr_{L^2(\R^d)\otimes\C^n}\big[h\big(G_L(&A_1,A_2;\Lambda, \Gamma)\big)\big] \nonumber \\ 
		=& \ L^d \Big[\mathfrak{W}_0\big(\tr_{\C^n}[h(\Real A_1)];\Lambda,\Gamma\big)
			+\mathfrak{W}_0\big(\tr_{\C^n}[h(\Real A_2)];\Lambda,\Gamma^c\big)\Big]\nonumber\\
		&+ L^{d-1}\log L \ \mathfrak{W}_1\big(\mathfrak{U}(h;\Real A_1,\Real A_2);\partial\Lambda,
			\partial\Gamma\big)\nonumber\\ 
		&+ o(L^{d-1}\log L),
	\end{align}
	as $L\rightarrow\infty$. Here, the coefficients $\mathfrak{W}_0$ and $\mathfrak{W}_1$ were defined 
	in \eqref{coeff-w0-def} and \eqref{coeff-w1-def},	respectively, and the symbol $\mathfrak{U}$ 
	in \eqref{frakU-def}.
\end{thm}
\begin{proof}
As the operator $G_L(A_1,A_2;\Lambda, \Gamma)$ is bounded, we can assume that $h$ is compactly supported. By Lemma \ref{Projections-H-Functions} there is some $R_0>0$ such that for every $R\in\,]0,R_0]$ we get the decomposition $h=\sum_{j=1}^Nh_{R,j}+g_{R}$, with $h_{R,j}$ satisfying Assumption \ref{Definition-Hoelder-Sobolev} with $x_{0}=x_{j}$ for every $j\in\,\{1,\ldots,N\}$ and $g_R\in C_c^2(\R)$. Let $\epsilon>0$ and $R\in\,]0,R_0]$ be arbitrary. Then we find a polynomial $p_{R,\epsilon}$ 
such that 
\begin{equation}\label{Approximation-Polynomial-Epsilon}
\max_{0 \leq k \leq 2} \|g_{R,\epsilon}^{(k)}\|_\infty<\epsilon
\end{equation}
for $g_{R,\epsilon}:=g_{R}-p_{R,\epsilon}$ by the theorem of Stone--Weierstrass. 

We now study the function $h_{R,\epsilon}:=h-p_{R,\epsilon}=\sum_{j=1}^Nh_{R,j}+g_{R,\epsilon}$. As it is decomposed in the same manner as $h$ in Lemma \ref{Projections-H-Functions}, the estimate from Lemma \ref{Projections-H-Functions} holds with $h=h_{R,\epsilon}$, $g_R=g_{R,\epsilon}$,
$X=G_L\big(A_1,A_2;\R^d,\Gamma\big)$, $P=\mathbf{1}_\Lambda$ and arbitrary 
$q\in \,]0,\gamma[\,$. We obtain
\begin{multline}\label{Asymptotics-Hoelder-Estimate-1}
\Big\|h_{R,\epsilon}\big(G_L(A_1,A_2)\big) - \mathbf{1}_\Lambda h_{R,\epsilon}\big(G_L(A_1,A_2;\R^d,\Gamma) \big)
	\mathbf{1}_\Lambda\Big\|_1 \\
	\leq \Big(\sum_{j=1}^{N} C_j \bnorm{h_{R,j}}_2 R^{\gamma-q} + C_0 \max_{0 \leq k \leq 2} \|g_{R,\epsilon}^{(k)}\|_\infty \Big) \Big\|\mathbf{1}_\Lambda G_L(A_1,A_2;\R^d,\Gamma) 
	\mathbf{1}_{\Lambda^c}\Big\|_q^q.
\end{multline} 
As the symbols $A_1$ and $A_2$ only depend on the variable $\xi$, the functional calculus yields the following equality
\begin{align}\label{Combination-Xi}
	D_L\big(h_{R,\epsilon}(\Real A_1),h_{R,\epsilon}(\Real A_2)\big) = \mathbf{1}_\Lambda h_{R,\epsilon}\big(G_L(A_1,A_2;\R^d,\Gamma) \big)
	\mathbf{1}_\Lambda.
\end{align}
With this  and Lemma \ref{Trace class estimates} we conclude from \eqref{Asymptotics-Hoelder-Estimate-1} that
\begin{equation}
\Big\|h_{R,\epsilon}\big(G_L(A_1,A_2)\big) - D_L\big(h_{R,\epsilon}(\Real A_1),h_{R,\epsilon}( \Real A_2)\big)\Big\|_1 \leq C(R^{\gamma-q}+\epsilon) L^{d-1}\log L,
\end{equation}
where the constant $C$ is independent of $L$, $R$ and $\epsilon$. Here we additionally used \eqref{Approximation-Polynomial-Epsilon} and that $\bnorm{h_{R,j}}_2$ is bounded uniformly in $R\in\,]0,R_{0}]$ (cf.\ the proof of Lemma \ref{Projections-H-Functions}).
Using that
\begin{align}
	\tr_{L^2(\R^d)\otimes\C^n}\big[ & h(G_L(A_1, A_2)) - D_L\big(h(\Real A_1),h(\Real A_2)\big)\big]\notag \\ 
	&\leq  \tr_{L^2(\R^d)\otimes\C^n}\big[p_{R,\epsilon}\big(G_L(A_1,A_2)\big) 
			- D_L\big(p_{R,\epsilon}(\Real A_1),p_{R,\epsilon}(\Real A_2)\big)\big]\notag \\
	&\quad + \big\|h_{R,\epsilon}\big(G_L(A_1,A_2)\big) - D_L\big(h_{R,\epsilon}(\Real A_1),h_{R,\epsilon}(\Real A_2)\big)\big\|_1,
\end{align}
and that the asymptotic formula is already established for polynomials, we get
\begin{multline}\label{Asymptotics-Hoelder-Estimate-2}
	\underset{L\rightarrow\infty}{\limsup}\frac{\tr_{L^2(\R^d)\otimes\C^n}\big[ h\big(G_L(A_1,A_2)\big)
		- D_L\big(h(\Real A_1),h(\Real A_2)\big)\big]}{L^{d-1}\log L}  \\ 
	\leq \mathfrak{W}_1(p_{R,\epsilon})+C(R^{\gamma-q}+\epsilon),
\end{multline}
where $\mathfrak{W}_1(p_{R,\epsilon}):=\mathfrak{W}_1\big(\mathfrak{U}(p_{R,\epsilon};\Real A_1,\Real A_2);\partial\Lambda,\partial\Gamma\big).$
For the coefficient $\mathfrak{W}_1(h_{R,\epsilon})$, Lemmas \ref{Coefficient-C2} and \ref{Coefficient-Hoelder} yield
\begin{align}
|\mathfrak{W}_1(h_{R,\epsilon})|\leq  \sum_{j=1}^N |\mathfrak{W}_1(h_{R,j})|+|\mathfrak{W}_1(g_{R,\epsilon})| \leq C(R^{\gamma-q}+\epsilon),
\end{align}
where the constant $C$ is again independent of $L$, $R$ and $\epsilon$, and we again used \eqref{Approximation-Polynomial-Epsilon} and that $\bnorm{h_{R,j}}_2$ is bounded uniformly in $R\in\,]0,R_{0}]$.
Therefore, 
\begin{align}
\mathfrak{W}_1(p_{R,\epsilon})\leq \mathfrak{W}_1(h)+|\mathfrak{W}_1(h_{R,\epsilon})|\leq\mathfrak{W}_1(h)+C(R^{\gamma-q}+\epsilon).
\end{align}
Combining this with \eqref{Asymptotics-Hoelder-Estimate-1}, we obtain
\begin{equation}
	\underset{L\rightarrow\infty}{\limsup}\frac{\tr_{L^2(\R^d)\otimes\C^n}\big[h\big(G_L(A_1,A_2)\big)
		- D_L\big(h(\Real A_1),h(\Real A_2)\big)\big]}{L^{d-1}\log L}  
	\leq  \mathfrak{W}_1(h)+C(R^{\gamma-q}+\epsilon).
\end{equation}
The corresponding lower bound for the liminf is found in the same way. As $R$ and $\epsilon$ are arbitrarily small, we conclude
\begin{align}\label{Asymptotics-Hoelder-Limit-Result}
	\underset{L\rightarrow\infty}{\lim} \frac{\tr_{L^2(\R^d)\otimes\C^n}\big[h\big(G_L(A_1,A_2)\big)
		- D_L\big(h(\Real A_1),h(\Real A_2)\big)\big]}{L^{d-1}\log L}
	= \mathfrak{W}_1(h).
\end{align}
As the symbols $A_1$ and $A_2$ only depend on the variable $\xi$, the operator $D_L\big(h(\Real A_1),h(\Real A_2)\big)$ gives rise to the volume terms by integrating its kernel along the diagonal, and we obtain the desired asymptotic formula
\begin{align}
	\tr_{L^2(\R^d)\otimes\C^n}\big[h\big(G_L(&A_1,A_2;\Lambda, \Gamma)\big)\big] \nonumber \\ 
	&= L^d \big[\mathfrak{W}_0\big(\tr_{\C^n}[h(\Real A_1)];\Lambda,\Gamma\big)
		+\mathfrak{W}_0\big(\tr_{\C^n}[h(\Real A_2)];\Lambda,\Gamma^c\big)\big]\nonumber\\
	&\quad + L^{d-1}\log L \ \mathfrak{W}_1\big(\mathfrak{U}(h;\Real A_1,\Real A_2);\partial\Lambda,\partial\Gamma\big)\nonumber\\ 
	&\quad +o(L^{d-1}\log L),
\end{align}
as $L\rightarrow\infty$.
\end{proof}

\begin{rem}\label{Remark-Gamma-GammaC}
The condition that $\Gamma$ is bounded in Theorems \ref{Asymptotics-Analytic}, \ref{Asymptotics-Smooth} and \ref{Asymptotics-Hoelder} can be replaced by the condition that $A_1$ is compactly supported in the second variable.

To see this note that the asymptotics for polynomials is already established under this condition (cf. Theorem \ref{Asymptotics-Gamma-Gamma-Complement} and Corollary \ref{Asymptotics-Gamma-Gamma-Complement-G_L}) and treat $A_1$ analogously to $A_2$ throughout the relevant proofs in this chapter. Of particular note here is the fact that this is possible, as we never used the fact that the complement of $\Gamma^c$ is bounded. 
 
The condition that $\Gamma$ is bounded in the estimates for the coefficient (Lemma \ref{Coefficient-Analytic}, \ref{Coefficient-C2} and \ref{Coefficient-Hoelder}) is no longer necessary, as with $A_1$ and $A_2$ being compactly supported in the second variable the symbol $\mathfrak{U}(g;A_1,A_2)$ is also compactly supported in the second variable.
\end{rem}

\begin{rem}\label{Remark-Lambda-LambdaC}
	Sometimes it is useful to consider domains $\Lambda$ which are not bounded themselves but only their complement is bounded. Then, an asymptotic formula as in Theorem~\ref{Asymptotics-Hoelder} does not hold, as the operator $h(G_L(A_1,A_2;\Lambda,\Gamma))$ is in general not trace-class. Nevertheless one gets the following interim result, cf.\ \eqref{Asymptotics-Hoelder-Limit-Result}, under the same assumptions as in Theorem \ref{Asymptotics-Hoelder} except that now $\Lambda^c$ needs to be bounded instead of $\Lambda$
	\begin{multline}
		\label{Lambda-Complement-Hoelder-Statement}
		\lim_{L\rightarrow\infty} \frac{\tr_{L^2(\R^d)\otimes\C^n} 
			\big[h\big(G_L(A_1,A_2;\Lambda,\Gamma)\big) - G_L\big(h(\Real A_1), h(\Real A_2);
			\Lambda,\Gamma\big)\big]}{L^{d-1}\log L} \\ 
		=  \mathfrak{W}_1\big(\mathfrak{U}(h;\Real A_1,\Real A_2);\partial\Lambda,\partial\Gamma\big).
	\end{multline}
\end{rem}

\begin{proof}
	The crucial part is to establish (\ref{Lambda-Complement-Hoelder-Statement}) for polynomial test 
	functions. The extension to more general functions  works as in Theorem \ref{Asymptotics-Hoelder}.
	Define $\phi\in C_c^\infty(\R^d)$ with $\phi\vert_{\Lambda^c}=1$. Recall that the symbols $A_1$ and $A_2$ in Theorem \ref{Asymptotics-Hoelder} only depend on the variable $\xi$. Note that 
	$\phi G_L(A_1,A_2;\R^d,\Gamma)\sim G_L(A_1,A_2;\R^d,\Gamma)\phi$ by Lemma \ref{Commutation-Lemma} 
	and \ref{Exchange-Lemma} as well as $(1-\phi^p)1_{\Lambda}=(1-\phi^p)$ for every $p\in\N$. Therefore, 
	for every $p\in\N$ we get 
	\begin{align}
		\label{Lambda-Complement-Hoelder-1}
		\phi^p\big(G_L(A_1,A_2;\Lambda,\Gamma)\big)^p \sim \big(G_L(\phi A_1,\phi A_2;\Lambda,\Gamma)\big)^p
	\end{align}
	and
\begin{multline}\label{Lambda-Complement-Hoelder-2}
	(1-\phi^p)\big(G_L(A_1,A_2;\Lambda,\Gamma)\big)^p-(1-\phi^p) G_L\big((\Real A_1)^p,(\Real A_2)^p;\Lambda,\Gamma\big) \\  \sim (1-\phi^p)\big(G_L(A_1, A_2;\R^d,\Gamma)\big)^p -(1-\phi^p) G_L\big((\Real A_1)^p,(\Real A_2)^p;\R^d,\Gamma\big) = 0.
\end{multline}
Combining (\ref{Lambda-Complement-Hoelder-1}) and (\ref{Lambda-Complement-Hoelder-2}), we conclude
\begin{multline}
	\big(G_L(A_1,A_2;\Lambda,\Gamma)\big)^p - G_L\big((\Real A_1)^p,(\Real A_2)^p;\Lambda,\Gamma\big)\\ 
	\sim \big(G_L(\phi A_1,\phi A_2;\Lambda,\Gamma)\big)^p 
		-G_L\big((\phi\Real A_1)^p,(\phi\Real A_2)^p;\Lambda,\Gamma\big).
\end{multline}
Due to the presence of $\phi$ in both symbols we can treat $\Lambda$ as a bounded domain, and Corollary
\ref{Asymptotics-Gamma-Gamma-Complement-G_L} yields 
\begin{multline}
	\underset{L\rightarrow\infty}{\lim} \frac{\tr_{L^2(\R^d)\otimes\C^n}\big[
		\big(G_L(\phi A_1,\phi A_2;\Lambda,\Gamma)\big)^p
		-G_L\big((\phi\Real A_1)^p,(\phi\Real A_2)^p;\Lambda,\Gamma\big)\big]}{L^{d-1}\log L} \\
	=  \mathfrak{W}_1\big(\mathfrak{U}(\id^p;\phi\Real A_1,\phi\Real A_2);
			\partial\Lambda,\partial\Gamma\big).
\end{multline}
As $\phi\vert_{\partial\Lambda}=1$, we get the desired result for polynomials.
\end{proof}

\section*{Acknowledgements}
The authors are grateful to the anonymous referees for their valuable comments and suggestions.
This work was partially supported by the Deutsche Forschungsgemeinschaft (DFG, German Research Foundation) 
	-- TRR 352 ``Mathematics of Many-Body Quantum Systems and Their Collective Phenomena" -- Project-ID 470903074.

%\bibliographystyle{mybib}
%\bibliography{references-Ruth,references}
\providecommand{\noopsort}[1]{} \providecommand{\singleletter}[1]{#1}
  \providecommand{\noopsort}[1]{} \providecommand{\singleletter}[1]{#1}

\end{document}